\documentclass[10pt, 3p, times, fullpage,  reqno]{amsart}
\newcommand {\norm}[1] {\left\| #1 \right\|}
\usepackage{a4wide}
\usepackage{amssymb}
\usepackage[usenames, dvipsnames]{color}
\usepackage{hyperref}

\usepackage{verbatim}

\usepackage{amscd,amsmath,amstext,amsfonts,amsbsy,amssymb,amsthm,mathrsfs,float}
\usepackage{multicol,graphicx}
\usepackage{array,multirow}
\usepackage{fancyhdr,booktabs}
\usepackage{amsfonts}
\usepackage{amscd,amsmath,amstext,amsfonts,amsbsy,amssymb,amsthm,mathrsfs,float}
\usepackage{multicol,graphicx}
\usepackage{array,multirow}
\usepackage{fancyhdr,booktabs}
\usepackage{hyperref}
\usepackage{enumerate}
\usepackage{amssymb,amsmath,graphicx,amsthm}
\usepackage{amsmath,amsfonts,latexsym,amsthm,amssymb}
\usepackage{color}
\usepackage{amssymb}
\usepackage{color}
\usepackage{enumerate,fullpage}
\usepackage{amssymb,amsmath,graphicx,amsthm}
\usepackage{amsmath,amsfonts,latexsym,amsthm,amssymb}

\newtheorem{theorem}{{\bf Theorem}}[section]
\theoremstyle{definition} \newtheorem{definition}[theorem]{\bf Definition}

\theoremstyle{plain} \newtheorem{lemma}[theorem]{Lemma}
\newtheorem{proposition}{Proposition}[section]

\newtheorem{remark}{Remark}[section]

\usepackage{amssymb}
\usepackage{bbm}
\usepackage{anysize}
\marginsize{1in}{1in}{1in}{1in}

\newcommand{\be}{\begin{eqnarray*}}
	\newcommand{\en}{\end{eqnarray*}}
\newcommand{\bes}{\begin{eqnarray}}
\newcommand{\ens}{\end{eqnarray}}
\newcommand{\al}{\alpha}

\newcommand{\la}{\lambda}

\def\nn{\nonumber}

\def\bq{\begin{equation}}
\def\eq{\end{equation}}
\def\bqq{\begin{eqnarray*}}
	\def\eqq{\end{eqnarray*}}
\title[Final value problem]{Existence and regularity results for terminal  value problem for nonlinear   fractional wave  equations  }
\author[N.H. Tuan]{Nguyen Huy Tuan}
\address[N.H. Tuan]{Department of Mathematics and Computer Science, University of Science-VNUHCM, 227 Nguyen Van Cu Str., Dist. 5, Ho Chi Minh City, Viet Nam}
\email{nhtuan@hcmus.edu.vn}

	\author[T. Caraballo] {Tom\'as  Caraballo}
\address{Departamento de Ecuaciones Diferenciales y Análisis Numérico C/ Tarfia s/n, Facultad de Matemáticas, Universidad de Sevilla, Sevilla 41012, Spain}
\email{caraball@us.es}

\author[T.B. Ngoc]{Tran Bao Ngoc}
\address{Institute of Research and Development, Duy Tan University, Da Nang 550000, Vietnam}%
\email{tranbaongoc3@duytan.edu.vn}

\author[Y. Zhou] {Yong Zhou }
\address{Faculty of Mathematics and Computational Science, Xiangtan University, Xiangtan,  P.R. China}%
\email{yzhou@xtu.edu.cn }

\begin{document}
	
	\begin{abstract}
	We consider the terminal value problem (or called final value problem, initial inverse problem, backward in time problem) of determining the initial value, in a general class of time-fractional wave equations with Caputo derivative, from a given final value.  We are concerned with the existence, regularity of solutions upon the terminal value. Under several assumptions on the nonlinearity, we address and show the well-posedness (namely, the existence, uniqueness, and   continuous dependence) for the terminal value problem.  Some regularity results for the mild solution and its derivatives of first and fractional orders are also derived. The effectiveness of our
methods are showed by applying the results to two interesting models:  Time fractional Ginzburg-Landau equation, and Time fractional Burgers equation, where time and spatial regularity estimates are obtained.

\vspace*{0.2cm}

	\end{abstract}
	\maketitle
	%\tableofcontents
 
	\section{Introduction} \thispagestyle{empty}

\subsection{Statement of the problem}
As we know, the derivatives of positive integer orders of a differentiable function are determined by its properties only in an infinitesimal neighborhood of the considered point. As a result, partial differential equations with integer-order derivatives cannot describe processes with memory. The fact that fractional calculus is a powerful tool for describing the effects of power-law memory. If an integer-order derivative is replaced by a fractional one, typically Caputo, or   Riemann-Liouville, Grunwald, Letnikov,  Weyl derivatives, then we have time-fractional PDEs.   Historically, the Riemann-Liouville and Caputo fractional derivatives are the most important ones.  Time-fractional differential equations have recently become a topic of active research, and they also have many applications for modeling physics situations or describing a wide class of processes with memory, such as transport theory, viscoelasticity, rheology, and non-Markovian stochastic processes. 

\vspace*{0.1cm}

In this paper, we consider the following  fractional wave  equation  
\begin{align}
\hspace*{1.18cm} \left\{ \hspace*{0.1cm} \begin{array}{lllllllcccccccc}
\displaystyle \frac{\partial^{\alpha}}{\partial t^\alpha}  u(x,t)   &=& - \mathcal{A}u(x,t)+  G(t,u(x,t))  ,&&  x\in \Omega, \,\,\, 0<t\le T, \vspace*{0.1cm} \\
\hspace*{0.45cm} u(x,t)  &=& 0,  &&  x\in \partial \Omega, 0<t<T, \vspace*{0.1cm} \\
\hspace*{0.45cm} u_t(x,0)  &=& 0, &&  x\in  \Omega, \hspace*{0.23cm} 0<t<T,  
\end{array}\right. 	\label{mainprob1}
\end{align} 
where $G$ is called a source function which will be defined later. 
The time fractional derivative   $\partial^\alpha/\partial t^\alpha$,  $1 <\alpha < 2$,  is understood as {\it the left–sided
	Caputo fractional derivative of order $\alpha$} with respect to $t$, which is defined by
$$
\displaystyle \frac{\partial^{\alpha}}{\partial t^\alpha} v(t) = \frac{1}{\Gamma(2-\alpha )} \int_0^t \frac{\partial^{2}}{\partial s^2} v(s)  (t-s)^{1-\alpha }   ds, 
$$
whereupon  $\Gamma$ is the Gamma function.   For $\alpha=2$, we recovery the usual time derivative of second order $\displaystyle \partial^{2}/\partial t^2  $.  Caputo
derivatives have been used recently to model fractional diffusion in plasma turbulence,
see \cite{Ca}. Another advantage of using   Caputo derivatives in modeling physical problems
is that   Caputo derivatives of constant functions are zero. 
In Problem \eqref{mainprob1},  let us assume that  $\Omega$  
is a nonempty open set and possesses a Lipschitz continuous boundary in $\mathbb{R}^N$, $N\ge 1$, $T>0$,  and let   $\mathcal A$  be a symmetric and uniformly elliptic operator on $\Omega$  defined by 
\begin{align*}
\mathcal{A} v(x) =  - \sum_{m=1}^N \frac{\partial }{\partial x_m} \left( \sum_{n=1}^N a _{mn}(x) \frac{\partial}{\partial x_n}v(x) \right) + q (x)v(x),\quad x\in\overline{\Omega},
\end{align*} where  $a _{ij}\in C^1\left(\overline{\Omega}\right)$,  $q\in C\left(\overline{\Omega};[0,+\infty)\right)$, and $\displaystyle a_{mn}=a_{nm}, 1\le m,n\le N$. We assume also that there exists a constant $b_0>0$ such that, for  $x\in \overline{\Omega}$, $y=(y_1,y_2,...,y_N)\in \mathbb{R}^N$,  
$$\displaystyle \sum_{1\le m,n\le N}   a _{mn}(x)  y_m y_n \ge b_0 |y|^2.$$

This paper  considers the initial inverse problem of determining the initial value  $ u(x,0)= u_0(x)$ from its final value  $ u(x,T)$.  We focus to study existence, uniqueness and regularity   of  mild solutions of Problem \eqref{mainprob1}  associated with the final value condition 
\begin{align} \label{mainprob4}
u(x,T)=f(x),  \quad  x\in \Omega ,   
\end{align}
where  $f$ belongs to an  appropriate  space.  

\vspace*{0.1cm}

In the study of \eqref{mainprob1}, we are mainly motivated by problems arising in  anomalous diffusion phenomena. Anomalous diffusion and wave equations are great interest in physics. They are frequently
used for the super-diffusive models of anomalous diffusion such as diffusion in heterogeneous media. These fractional differential equations have another important issue in the probability theory related to non-Markovian diffusion processes with memory.  Fractional wave equations also describe evolution processes intermediate between diffusion and wave propagation \cite{Mai1,Mai2,Mai}. In \cite{Mai},  it has been shown that the fractional wave equation governs the propagation of mechanical diffusion-waves in viscoelastic media. Such waves are relevant in acoustics, seismology, medical imaging, etc.
The physical background for a time-space fractional diffusion-wave equation can be seen in \cite{Chen}.

\subsection{Motivations}
If   condition \eqref{mainprob4} is replaced  by
\begin{align} \label{mainprob44}
u(x,0)=\overline f(x),  \quad  x\in \Omega ,   
\end{align}
then we receive the {\it direct problem}  or {\it initial
	value problem  (IVP)}  of   \eqref{mainprob1}.  Some  quasi-linear equations of the form  \eqref{mainprob1} and \eqref{mainprob44} with standard time derivative $(\alpha=2)$ have been
extensively studied in literature and several results.  A global well-posedness has been
proved both in the subcritical case 
by Ginibre and Velo \cite{3}, and in the critical case, 
by Grillarkis \cite{5},  Shatah and Struwe \cite{22,23}, and some references therein.

\vspace*{0.1cm}

In fractional derivative cases, such as  Caputo or Riemann–Liouville derivatives, Problem  \eqref{mainprob1} and \eqref{mainprob44}  have been considered with  $G=0$ or $G=G(x,t)$ by some authors see,  e.g. \cite{Samko,Dong,Podlubny,Lax,Brezis,Diethelm} and also \cite{r2,r3,r4}.  

\vspace*{0.1cm}

Concerning to Problem \eqref{mainprob1} and \eqref{mainprob44}  with derivative $\partial^\alpha/\partial t^\alpha$ for $0<\al<1$, some authors developed and obtained    interesting results. A. Carvaho et  al.  \cite{Andrade}   established a local theory of mild solutions   where $\mathcal A$ is
a sectorial (nonpositive) operator.  	B.H. Guswanto \cite{Gu} studied the existence and uniqueness of a local
mild solution to a class of initial value problems for nonlinear fractional evolution equations. Besides, The existence, uniqueness and regularity of solutions  are established in some previous works see, for instant,  \cite{Nochetto,Yamamoto1,Yamamoto2}.  
Such research is rapidly developing, and here we do not intend to give any comprehensive lists
of references.  

\vspace*{0.1cm}

In our knowledge, studying the initial value problem for the fractional wave equations  in the nonlinear case are still limited.  Recently,   M. Yamamoto et al  \cite{Yamamoto2} have studied Problem \eqref{mainprob1} and (\ref{mainprob44}) with a linear inhomogeneous source, i.e.  $G=G(x,t)$,  and then further investigated   local solutions with a nonlinear source.  Very recently, M. Warma et al \cite{Al} considered the existence and regularity of local and global weak solutions with a suitable growth assumption on the nonlinearity $G$.
 Except for 	 
the works \cite{Yamamoto2,Pi,Al},  there are very few results on Problem   \eqref{mainprob1}-\eqref{mainprob4} to the best of our knowledge. 

\vspace*{0.1cm}

 In practice, there are some physical models which are not subjected to  initial value problems.  Some  phenomena cannot be observed at the time $t=0$, and only can be measured at a terminal time $t=T>0$.  Hence, a final value condition  appears  instead of the respectively initial value one.  It has great importance in engineering areas and aimed at detecting the previous state of a physical field from its present information. In a few sentences, we explain the  presence of the equation $u_t(x,0)=0$. By \cite{Yang},   the system  (\ref{mainprob1})-(\ref{mainprob4}) in the 2-dimensional case can be considered as a description for an imaging process, namely, to recover an  exact picture from its blurry form.
 The condition  $u_t(x,0)=0$ means that the distribution does not change on the interval $(0,t_0)$ with $t_0$ is near zero.
 So, the necessary of studying {\it terminal value problems} or {\it  final value problems (FVPs)} or {\it backward problems} are certain. 

\vspace*{0.1cm}

The     final value problem   (\ref{mainprob1})-(\ref{mainprob4}) with {\it  derivatives of integer-orders } has been  treated for a long time, e.g., see \cite{Showater,Carasso,Baumeister}.  In  \cite{Carasso}, A. Carasso et al. considered the following final value problem for the traditional wave equation  (i.e. $\alpha =2$)        
$$
\left\{\begin{array}{lllllllccccccc}
u_{tt} &=& (\Delta  + k)^2 u, && x\in \Omega, \hspace*{0.35cm} 0<t<T,\\
u = \Delta u &=& 0, && x\in \partial\Omega, ~ 0<t<T,\\
u_t(x,0)&=&g(x), && x\in \Omega,\\
u(x,T)&=&f(x), && x\in \Omega, 
\end{array} \right.
$$
where $k$ is a given positive number which may  equal several eigenvalues
of $-\Delta $, and $f$, $g$ are given functions.
For now, little
research has been done on the inverse problems of time-space fractional diffusion equations. FVPs for fractional PDEs can be roughly divided into two topics.
 The first one contains problems that are to study the ill-posedness and propose some regularization methods for approximating a  sought solution. 
We can list some well-known results, for example, J. Jia et al \cite{Jia},  J. Liu et al \cite{Liu}, some papers of M. Yamamoto and his group see \cite{Ya5,Ya6,Ya8,Ya9,Ya10},   B. Kaltenbacher et al \cite{Bar,Bar1},   W. Rundell et al \cite{Run,Run1},    J. Janno see \cite{Ja1,Ja2}, etc.
The second topic contains problems that are to study the existence and regularity of solutions such as \cite{NHTuanN,Luc}. Investigating the existence and regularity of solutions of  ODE/PDE models plays an important role  in both the development of the ODE/PDE theory and its applications in real-life problems. Furthermore, studying regularity helps to improve the smoothness and stability of solutions in    different spaces, and hence makes the numerical simulation valuable. The second topic   has been not well treated in the literature. \\

As far as we know, there are few works to analyze Problem \eqref{mainprob1}-\eqref{mainprob4} which provides existence, uniqueness results, and
  regularity estimates.  The main difficulty in the analysis of Problem (\ref{mainprob1})-(\ref{mainprob4}) and the essential difference from the traditional problems come by the nonlocality of the time-fractional derivative $\partial^\alpha/\partial t^\alpha$. The major question for this work in our mind that 

\vspace*{0.1cm}
  
{\it  What is the regularity of the corresponding solution $u$ (output data) if the given  data (input data) $f,G$  are regular?}

\vspace*{0.1cm}

Our goal   in this paper  is to  find     suitable Banach spaces  for the given data $(f,G)$ in order to obtain existence and  regularity results for the corresponding solution. The regularity estimates are important in the analysis of time discretization schemes for Problems (\ref{mainprob1})-(\ref{mainprob4}) in the future.

\vspace*{0.1cm}

The difficulties of a final value problem can be briefly described as follows (see Remark \ref{Diff} for more details). Firstly, since the fractional derivative $\displaystyle \partial_t^{\,\alpha} $ is non-locally defined on the time interval $(0,t)$, we cannot convert a final value problem to an initial-value problem by using some substitution methods. Secondly, the formulation of mild solutions of a final value problem is more complex than the initial problem respectively. This positively promotes us to construct new solution techniques  to deal with Problem (\ref{mainprob1})-(\ref{mainprob4}). Some more details can be found in Subsection 3.1, where the 
explicit representation of solutions relies on the eigenfunctions expansion and the Mittag-Leffler functions.   details,  
\vspace*{0.2cm}

Let us describe the main results of this paper in two cases as follows. The first case is related to the properties of solutions under a globally Lipschitz (GL) assumption on the nonlinearity corresponding to two first theorems, while the second one concerns a critical nonlinearity  corresponding to the third theorem. 
 In the first theorem, we obtain the regularity results of solutions and its derivatives of first and fractional orders under the (GL) assumptions $\mathscr H_1$ (see page 7). 
The key idea of its proof is based on a Picard iteration argument and techniques of finding appropriate spaces for $f$. Choosing spaces of $f$ and $G$ is a difficult and nontrivial task when we study the regularity of the solution. Although applications of our problem under $\mathscr H_1$ are not wide,  
the analysis and techniques here are helpful tools for studying the next results. Moreover,  the existence of a mild solution in the space $L^\infty$ may not obtain by considering $\mathscr H_1$. This can be overcome by considering the (GL) assumption $\mathscr H_2$ of the nonlinearity which is presented in the second theorem. The third theorem uses the contraction mapping principle to prove the existence of a mild solution in the    critical case. 
As we know,  nonlinear PDEs with critical nonlinearities are an interesting topic. We  can mention this in \cite{Carvalho} and references therein. 
Studying the initial value problem for (\ref{mainprob1}) in the critical case is also a challenging problem, therefore investigating the regularity of the mild solution and its derivatives are very difficult. 

\subsection{Outline}
The outline of this paper is as follows. In Section 2, we introduce some terminology used throughout this work.  Moreover,  we obtain a precise representation of solutions by using Mittag-Leffler functions. In Section 3, we investigate the well-posedness, and regularity of a mild solution to  Problem (\ref{mainprob1})-(\ref{mainprob4}).  
Three main results on the existence, uniqueness (in some suitable class of functions),   regularity of the mild solution and its derivatives are proved under suitable assumptions on the terminal data and the nonlinearity. In Section 4, we apply the
theoretical results to some typically fractional diffusion: Time fractional Ginzburg-Landau equation and  Burgers  equation.  Finally, in Section 5, we provide full proofs to the main theorems established in Section 3.

\section{Preliminaries} 	
In this section we will recall some properties that will be useful for the study of the well posedness of  Problem (\ref{mainprob1})-(\ref{mainprob4}). We start by introducing some functional spaces. Then we will recall some properties of  Mittag-Leffler functions.
  Let  the operator $\mathcal{L}$ be considered on $L^2(\Omega)$ with respect to domain   $D(\mathcal{L})=    W_0^{1,2}(\Omega) \cap  W^{2,2}(\Omega),$ 
where $L^2(\Omega)$, $W_0^{1,2}(\Omega)$, $W^{2,2}(\Omega)$ are the usual Sobolev spaces. Then the spectrum of $ \mathcal{L}$ is a non-decreasing sequence of positive real numbers $\displaystyle \{ \lambda_j  \}_{j=1,2,...}$ satisfying 
$\lim_{j\to \infty}\lambda_j=\infty$. Moreover,  there exists a positive constant $c_{\mathcal L}$ such that
$ 
\lambda_j \ge c_ {\mathcal L} j^{2/d},
$
for all  $j\ge 1$, see \cite{Courant}.  
Let us denote by $\displaystyle  \{ \varphi_j  \}_{j=1,2,...} \subset D(\mathcal L)$ 
the set of eigenfunctions of $\mathcal{L}$, i.e., $\displaystyle \mathcal L  \varphi_j = \lambda_j \varphi_j$, and $\varphi_j=0$ on $\partial \Omega$, for all  $j\ge 1$. 
The sequence $\displaystyle  \{ \varphi_k  \}_{k=1,2,...}$ forms an orthonormal basis of $L^2(\Omega)$, see e.g.  \cite{Kato}. 
For a given real number $\gamma\ge 0$, the Hilber scale space 
$$
\displaystyle \mathbb H^{ \gamma}(\Omega): =\left\lbrace v 
%{\,\,  = \sum_{j=1}^\infty \langle v,\varphi_j \rangle \varphi_j } 
\in L^2(\Omega): \sum_{j=1}^\infty   \lambda_j^{2\gamma}  \langle v,\varphi_j \rangle   ^2 <\infty \right\rbrace 
$$   
($\langle \cdot,\cdot \rangle$ is the usual product of $L^2(\Omega)$)	endowed with the norm 
$  \Vert v\Vert^2_{\mathbb H^{ \gamma}(\Omega)}:= \sum_{j=1}^\infty   \lambda_j^{2\gamma} \left\vert \langle v,\varphi_j \rangle \right\vert^2.$  
{ We have $\mathbb H^0(\Omega)=  L^2(\Omega)$ if $\gamma=0$,  and $\mathbb{H}^{\frac{1}{2}}(\Omega) = W_0^{1,2}(\Omega)$}. 
We denote by $\mathbb H^{-\gamma}(\Omega)$ the dual space of $\mathbb H^{\gamma}(\Omega)$ provided that the dual space of $  L^2({\Omega})$ is identified with itselt, e.g. see \cite{McLean}. The space  $\mathbb H^{-\gamma}(\Omega)$ is a Hilbert space with respect to the norm 
$ 
\norm{v}^2_{\mathbb H^{-\gamma}(\Omega)}= \sum_{j=1}^\infty   \lambda_j^{-2\gamma}   \langle v,\varphi_ j  \rangle _{-\gamma,\gamma}  ^2 ,  
$
for $v\in  \mathbb H^{-\gamma}(\Omega)$  where $\langle.,.\rangle_{-\gamma,\gamma}$ is the dual product between $\mathbb H^{-\gamma}(\Omega)$ and $\mathbb H^{\gamma}(\Omega)$. We note that 
\begin{align}
\langle\tilde{v},v\rangle_{-\gamma,\gamma}  =\langle \tilde{v},v \rangle , \quad \textrm{for } \tilde{v}\in   L^2(\Omega), v\in \mathbb H^{\gamma}(\Omega).\label{gg}
\end{align}  
For given numbers $p\ge 1$ and $\nu \in \mathbb{R}$, we let $ L^p (0,T;\mathbb H^\nu(\Omega))$ be the space of all functions $w:(0,T)\to \mathbb H^\nu(\Omega)$ such that  
$$
\Vert w\Vert_{ L^p (0,T; \mathbb H^\nu(\Omega))} := \left(\int_{0}^T \Vert w(t)\Vert_{\mathbb H^\nu(\Omega)}^p dt \right)^{1/p} <\infty.$$ 
Let us denote by $  C((0,T];\mathbb H^\nu (\Omega))$ the set of all continuous functions which map  $(0,T]$ into $\mathbb H^\nu(\Omega)$. For a given number $ \eta>0$, we denote by $\displaystyle  C^\eta((0,T];\mathbb H^\nu (\Omega))$ the space of all functions $w$ in $ C((0,T];\mathbb H^\nu (\Omega))$ such that
$ \Vert w\Vert_{ C^\eta((0,T];\mathbb H^\nu (\Omega))}:=\sup_{0<t\le T} t^\eta \Vert w(t)\Vert_{\mathbb H^\nu (\Omega)}<\infty , $    see \cite{Dang}.

\subsection{Fractional Sobolev spaces} We recall some Sobolev embeddings as follows. Let $\Omega$ be a nonempty open set with a Lipschitz continuous boundary in $\mathbb{R}^N$, $N\ge 1$. Let us recall that the notation $W^{s,p}(\Omega)$, $s\in \{0,1,2,...\}$, $p\ge 1$, refers to the standard Sobolev one, e.g. see \cite{Adams}. In the case $0\le s\le 1$ is  a positive real number, the intermediary space $W^{s,p}(\Omega)=\left[ L^p(\Omega);W^{1,p}(\Omega) \right]_s$ can be defined by 
\begin{align*}
W^{s,p}(\Omega)= \left\{ u\in L^p(\Omega): \frac{|u(x)-u(x')|}{|x-x'|^{\frac{N}{p}+s} } \in L^p(\Omega\times \Omega) \right\}. \nn
\end{align*}   
Since $\Omega$ 
is a nonempty open set and possesses a Lipschitz continuous boundary in $\mathbb{R}^N$, then the following Sobolev embedding  holds  
\begin{align}
\displaystyle W^{\sigma,p}(\Omega) \hookrightarrow W^{\gamma,q}(\Omega)  \qquad \textrm{if } \quad \left\{
\begin{array}{lllccc}
1 \le p , q < \infty , \vspace*{0.2cm} \\
0 \le  \gamma \le \sigma<\infty , \vspace*{0.2cm} \\
\displaystyle  \sigma-\gamma \ge \frac{N}{p} - \frac{N}{q}. \vspace*{0.2cm}
\end{array} \right. \label{nhung1}
\end{align} 
By letting $p=2$, $\gamma=0$  in  (\ref{nhung1}), one  obtains that  $ W^{\sigma,2}(\Omega)   \hookrightarrow L^q(\Omega)$ with $1\le q<\infty$,  $0\le \sigma <\infty$, and $\sigma \ge \frac{N}{2}-\frac{N}{q}$. Henceforth, setting $0\le \sigma <\frac{N}{2}$ infers that $1\le  q \le \frac{2N}{N-2\sigma}$. Summarily, we obtain the following embedding  
\begin{align}
\displaystyle W^{{\sigma_1},2}(\Omega)  \hookrightarrow L^{q_1}(\Omega)  \qquad \textrm{if } \quad  0\le {\sigma_1} <\frac{N}{2}, ~ 1\le  {q_1} \le \frac{2N}{N-2{\sigma_1}}. \label{nhung1a}
\end{align}  
This  {  implies   $W_0^{{\sigma_1},2}(\Omega)  \hookrightarrow L^{q_1}(\Omega)$, and so   	$L^{{q}_1^*}(\Omega) =\big[ L^{q_1}(\Omega) \big]^*  \hspace*{0.1cm} \hookrightarrow \big[ W_0^{{\sigma_1},2}(\Omega) \big]^*= W^{-{\sigma_1},2}(\Omega)$} with respect to $ -\frac{N}{2} < -{\sigma_1} \le 0$ and $ {q_1^*} \ge \left(\frac{2N}{N-2{\sigma_1}}\right) \left\{ \left(\frac{2N}{N-2{\sigma_1}}\right) -1 \right\}^{-1}= \frac{2N}{N+2{\sigma_1}}$. Thus, 
\begin{align}
\displaystyle  L^{q_2}(\Omega)  \hspace*{0.1cm} \hookrightarrow  W^{{\sigma_2},2}(\Omega)   \qquad \textrm{if } \quad  -\dfrac{N}{2} <  {\sigma_2} \le 0, ~ {q_2} \ge  \dfrac{2N}{N-2{\sigma_2}}.   \label{nhung1b}
\end{align}   

On the other hand,  the Hilbert scale spaces   and the fractional Sobolev space are related to each other by the following embeddings \vspace*{0.2cm}
\begin{align}
\mathbb{H}^s(\Omega) \hookrightarrow W^{2s,2}(\Omega) \hookrightarrow L^2(\Omega) ,\hspace*{0.08cm} \quad \textrm{if} \quad s\ge { 0 .}  \vspace*{0.2cm}  \label{HSem}
\end{align}

\subsection{On the Mittag-Leffler functions} 
An important function in the integral formula of solutions of many differential equations involving in the Caputo fractional derivatives is the
Mittag-Leffler function, which is defined by
$$\displaystyle E_{\alpha,\beta}(z)=\sum_{k=1}^{\infty} \frac{z^k}{\Gamma(\alpha k +\beta)}, \quad z\in \mathbb{C},$$
for $\alpha>0$ and $\beta\in \mathbb{R}$. It will be used to represent the solution of Problem (\ref{mainprob1})-(\ref{mainprob4}). We recall the  following lemmas, for which their proofs can be found in many books, e.g. see \cite{Samko,Podlubny,Diethelm}. The first one is very basic and also useful in many estimates of this paper. The second one  helps us to find the derivatives of first order  and  fractional order $\alpha $ of  the mild solution of  problem (\ref{mainprob1})-(\ref{mainprob4}). The third one is important and helps us to deal with the Mittag-Leffler function corresponding to the final time $T$. In the last one, we combine the first and  third ones to derive some more important estimates. The proof of  this lemma can be found in the Appendix.  In this paper, we always consider $T$ satisfies   assumption (\ref{Tassumption}) below. 

\begin{lemma} \label{MLineqna} Given $1<\alpha<2$ and $\beta \in \{1;\alpha\}$,   then there exist   positive constants $m_\alpha$ and ${ M}_\al$,   depending only on $\alpha$ such that 
	\begin{align*}
	\frac{m_\al}{1+t}  \le  | E_{\alpha,\beta}(-t) |  \le \frac{ {  M}_\al}{1+t}, \quad \textrm{ for all } t\ge 0.
	\end{align*} 
\end{lemma}

\begin{lemma}\label{DeriML} Assume  $1<\alpha <2$, $\lambda>0$, and $t>0$. Then the following differentiation formula hold \vspace*{0.15cm}
	\begin{itemize}
		\item[a)] $\displaystyle \partial_t  E_{\alpha ,1}(-\lambda t^\alpha ) \hspace*{0.15cm}  = -\lambda t^{\alpha -1}  E_{\alpha ,\alpha  }(-\lambda t^\alpha ), \hspace*{-0.06cm} \quad  \textrm{and} \quad 
		\partial_t   (t^{\alpha -1}E_{\alpha ,\alpha  }(-\lambda t^\alpha )  ) \hspace*{0.15cm} = t^{\alpha -2}E_{\alpha ,\alpha -1}(-\lambda t^\alpha )$;
		
		\item[b)] $\partial_t ^{\,\alpha}     E_{\alpha ,1 }(-\lambda t^\alpha )  =  -\lambda E_{\alpha ,1}(-\lambda t^\alpha ),  
		\hspace*{1.05cm} \textrm{and} \quad
		\partial_t^{\,\alpha}     ( t^{\alpha -1}E_{\alpha ,\alpha  }(-\lambda t^\alpha )  )  =  -\lambda t^{\alpha -1}E_{\alpha ,\alpha  }(-\lambda t^\alpha )$. 
	\end{itemize}
\end{lemma}

\begin{lemma} [see \cite{Wei}] \label{TTT}
	Let $1<\alpha <2$. If the number $T$ is large enough then \begin{equation}
	E_{\alpha ,1}(-\lambda_j T^\alpha ) \ne 0, \quad \textrm{for all }  j\in\mathbb{N},  j\ge 1, \label{Tassumption}
	\end{equation} 
	and there exist two positive  constants $m_\al$,  and $M_\al$ such that
	\begin{equation}
	\frac{m_\al}{1+\lambda_j T^\alpha} \le  \Big|	E_{\alpha ,1}(-\lambda_j T^\alpha ) \Big| \le 	\frac{ M_\al}{1+\lambda_j T^\alpha}. \nn
	\end{equation} 
\end{lemma}

\begin{lemma} \label{nayvesom} Given $1<\alpha<2$ and  $0 < \theta <1$, it holds  that \vspace*{0.2cm}
	\begin{itemize}
		\item[a)] $z^{\al-1} E_{\al, \al} (-\lambda_j z^\al) \le { M}_\al \lambda_j^{-\theta} z^{\al(1-\theta)-1}$;
		\item[b)] $\mathscr E_{\al, T} (-\lambda_j t^\al)  \le M_\al m_\al^{-1}  \left(\la_1^{-1}+  T^\alpha\right) \lambda_j^\theta t^{-\al(1-\theta)}$ with $\displaystyle \mathscr E_{\al, T} (-\lambda_j t^\al):= \frac{E_{\alpha ,1}(-\lambda_j t^\alpha )}{E_{\alpha ,1}(-\lambda_j  T^\alpha )}$.
	\end{itemize}
\end{lemma}

From now on,	we will use $a \lesssim b$ to denote the existence of a constant $C>0$, which may depend only on $\al, T$ such that $a \le C b.$

\section{Existence and regularity of the terminal value problem (\ref{mainprob1})-(\ref{mainprob4})}

\subsection{Mild solutions}
For considering   solutions of  partial differential equations, many  authors often study the formulation of   solutions in the  classical, weak, or mild  sense.  In this work, we will study mild solutions of FVP (\ref{mainprob1})-(\ref{mainprob4}). 

There are many works considering the precise formulation of mild solutions to IPVs for time fractional wave equations, such as \cite{Nochetto,Yamamoto1,Yamamoto2,Kian,Chen,Al,Mai1,Li,
Andrade}, by using complex integral representations on Banach spaces or  spectral representations on Hilbert scale spaces of the Mittag-Leffer operators. To study FVPs for time fractional wave equations, the precise formulation of mild solutions can be derived by using spectral representations of the inverse Mittag-Leffer operators, such as  \cite{NHTuanN,Run,Ya9,Ja1,Bar,wei,Jia,Liu,Nieto3}. In what follows, we state a definition of mild solutions to FVP (\ref{mainprob1})-(\ref{mainprob4}) where the precise formulation can be obtained by some simple computations.

Additionally, for a given two-variables function $w=w(x,t)$, we will write $w(t)$ instead of $w(.,t)$ and understand $w(t)$ as a function of the spatial variable $x$.   

\begin{definition}
	A function $u$ in $ L^p(0,T;\mathbb{H}^{\nu}(\Omega))$ or $   C^{\eta} (0,T; \mathbb{H}^{\nu}(\Omega))$ (with some suitable numbers $p\ge 1, \nu\ge 0$ or $\eta> 0$) is called a mild solution of Problem  \eqref{mainprob1}-\eqref{mainprob4} if 
	it satisfies the following equation
	\begin{align} \label{Lmildsolution}
	u(t)=  	   {\bf B}_\alpha(t,  T) f &+ \int_0^t   {\bf P}_\alpha  (t-r) G(r,u(r)) dr - \int_0^{ T}    {\bf B}_\alpha(t,  T)    { \bf  P}_{\alpha } (T-r) G(r,u(r)) dr   
	\end{align}
in the sense of $\mathbb{H}^\nu(\Omega)$,	where, for $0 \le  t\le T$, the solution operator $  {\bf B}_{\alpha }$, $ {\bf P}_{\alpha }$ are given by
	\begin{equation}
	\begin{aligned}
	\label{operatorAT}
	{\bf B}_\alpha(t,  T)  v := \sum_{j=1}^\infty \frac{E_{\alpha ,1}(-\lambda_j t^\alpha )}{E_{\alpha ,1}(-\lambda_j  T^\alpha )} \langle v,\varphi_j\rangle\varphi_j , \quad 
	{\bf P}_\alpha(t)  v:= \sum_{j=1}^\infty    t^{\alpha -1} E_{\alpha ,\alpha }(-\lambda_j t^\alpha )  \langle v,\varphi_j \rangle\varphi_j.  
	\end{aligned} 
	\end{equation}
	where $v= \sum_{j=1}^\infty \langle v,\varphi_j\rangle\varphi_j  $.
\end{definition}

%	 
%	We recall that the mild formulation of FVP (\ref{mainprob1})-(\ref{mainprob4}) is given by (\ref{Lmildsolution}) which can be  
%	rewritten   in the form
%	\begin{align}  
%	u(t)=  	 {\bf B}_\alpha(t,  T)  \left( f - \int_0^{ T}     {\bf P}_{\alpha } (T-r)  G(r,u(r)) dr \right) &+ \int_0^t    {\bf P}_{\alpha } (t-r) G(r,u(r)) dr . \label{khuyab}
%	\end{align}
\begin{remark} \label{Diff}   A mild formulation of this IVP (\ref{mainprob1}), (\ref{mainprob44}) is given by 
	\begin{align}  
 u(t)= \mathbf B_\alpha^{(0)}(t) \overline{f} &+ \int_0^t   {\bf P}_{\alpha }  (t-r) G(r,u(r)) dr ,   \label{khuyaa} 
	\end{align}
	where $\mathbf B_\alpha^{(0)}(t) w:= \sum_{j=1}^\infty    E_{\alpha ,1 }(-\lambda_j t^\alpha )  \langle w,\varphi_j \rangle\varphi_j$, see  \cite{Nochetto,Yamamoto1,Yamamoto2,Kian,Chen,Al,Mai1,Li,
Andrade}, etc.
	It should be pointed out some core differences between the IVP (\ref{mainprob1}), (\ref{mainprob44}), and the FVP (\ref{mainprob1})-(\ref{mainprob4}) for fractional wave equations as what follows  
	\begin{itemize}
		\item The solution operator ${\bf B}_\alpha(t,  T)$ is weaker than $\mathscr B_\alpha(t)$. Indeed, one can see that if $v\in L^2(\Omega)$ then $\mathscr B_\alpha(t)v\in L^\infty(0,T; L^2(\Omega))$ and $${\bf B}_\alpha(t,  T)v\notin  L^\infty(0,T; L^2(\Omega)) \cup C([0,T];L^2(\Omega)).$$  Therefore, it is actually difficult to establish the existence of mild solutions, especially in the critical nonlinear case; 
		\item  Mild formulation of the FVP (\ref{mainprob1})-(\ref{mainprob4}) contains more terms than the IVP (\ref{mainprob1}), (\ref{mainprob44}). In particular, estimating  the last term of (\ref{Lmildsolution}) requires very clever techniques in  acting ${\bf B}_\alpha(t,  T)$, ${\bf P}_\alpha  (t-r)$ on $G(r,u(r))$. In the critical nonlinear case, it is very difficult to determine where does the quantity $${\bf B}_\alpha(t,  T)  {\bf P}_\alpha  (t-r) G(r,u(r))$$
belongs?, and also how to bound this quantity such that its integration on the whole interval $(0,T)$ is convergent? 		
		
		\item The Gronwall's inequality can be applied when we estimate   solutions of the IVP (\ref{mainprob1}), (\ref{mainprob44}). However, it cannot when we estimate   solutions of the FVP (\ref{mainprob1})-(\ref{mainprob4}) since (\ref{Lmildsolution}) contains the integral on $(0,T)$.  
	\end{itemize}
	Hence,  studying  FVP (\ref{mainprob1})-(\ref{mainprob4})  is a difficult task. 
\end{remark}

\subsection{Well-posedness of Problem (\ref{mainprob1})-(\ref{mainprob4})  in the  globally Lipschitz case}

In this section, we   study the well-posedness of Problem (\ref{mainprob1})-(\ref{mainprob4}), and regularity of the solution which we consider the following globally  Lipschitz assumptions on  $G$: \vspace*{0.2cm}
\begin{itemize}
	\item[($\mathscr H_1$)] 
	The function	$G: [0,T] \times  \mathbb H^{\nu}({\Omega}) \to  \mathbb H^{\nu}({\Omega})$ satisfies    $G(t,0)=0$, and there exists a non-negative function $   L_1 \in L^\infty(0,T) $ such that
	\begin{align}
	\norm{G(t,w_1)-G(t,w_2)}_{  \mathbb  H^{\nu}({\Omega})} \le\,&   L_1(t) \norm{ w_1  - w_2}_{ \mathbb H^{\nu}({\Omega})}, \label{Lipschitz2}
	\end{align}
	for all $0\le t \le T$, and $w_1,w_2 \in \mathbb  H^{\nu}({\Omega})$. \vspace*{0.2cm}
	\item[($\mathscr H_2$)]    The function	$G:  [0,T] \times   C\left([0,T]; \mathbb H^\nu(\Omega)\right) \cap  L^q(0,T; \mathbb H^{\sigma}(\Omega)) \to H^{\nu+1}({\Omega}) $ satisfies    $G(t,0)=0$, and there exists a non-negative function $  L_2 \in L^\infty(0,T) $ such that
	\begin{align}
	\norm{G(t,w_1)-G(t,w_2)}_{   H^{\nu+1}({\Omega})} \le\,&   L_2  \Big\|  w_1-w_2\Big\|_{C\left([0,T]; \mathbb H^\nu(\Omega)\right) \cap  L^q(0,T; \mathbb H^{\sigma}(\Omega)) }, \label{Lipschitz2}
	\end{align}
	for all $0\le t \le T$, and $w_1,w_2\in    C\left([0,T]; \mathbb H^\nu(\Omega)\right) \cap   L^q(0,T; \mathbb H^{\sigma}(\Omega))$, where
	$$
	\Big\|  w_1-w_2\Big\|_{C\left([0,T]; \mathbb H^\nu(\Omega)\right) \cap  L^q(0,T; \mathbb H^{\sigma}(\Omega)) }= \|w_1-w_2\|_{  C\left([0,T];\mathbb H^\nu(\Omega)\right) }+\|v_1-v_2\|_{  L^q(0,T; \mathbb H^{\sigma}(\Omega))} , 
	$$
	and $\nu\ge 0$, $q\ge 1$, $\sigma\ge 0$.
\end{itemize}    \vspace*{0.2cm}

In order to establish our main results, it is useful to note that  $$0<\frac{\alpha-1}{\alpha}<\frac{1}{2}<\frac{1}{\alpha}<1,\quad \textrm{as}\quad 1<\alpha<2.$$
Besides, we   recall that the Sobolev embedding   
$ \mathbb H^{\nu+\theta}(\Omega) \hookrightarrow  \mathbb H^{\nu}(\Omega) $ holds as $\nu\ge 0$ and $\theta\ge 0$, so there exists a positive constant $C_1(\nu,\theta)$ such that
\begin{align}
\big\| v  \big\|_{\mathbb H^{\nu}(\Omega)} \le   C_1(\nu, \theta)   \big\|v \big\|_{\mathbb H^{\nu+\theta}(\Omega)}, \label{Embed1}
\end{align} 
for all $v \in \mathbb{H}^{\nu+\theta}(\Omega)$.  In addition, for the reader  convenience, the important constants (which may appear in some proofs)  are summarily  given by \textbf{(AP.4.)} in the Appendix. 

\vspace*{0.2cm}
  
The first result in  Theorem  \ref{t1}  ensures the existence of   a mild  solution in  $ L^p (0, T ;   \mathbb H^\nu({\Omega}))$ under  appropriate assumptions on $p$, the final value data $f$, and the assumption $(\mathscr H_1)$ on the nonlinearity $G$. The idea is to construct a Cauchy sequence in $ L^p (0, T ;   \mathbb H^\nu({\Omega}))$ which will bounded by a power function and must   converge  to a mild solution of Problem (\ref{mainprob1})-(\ref{mainprob4}). The solution is then bounded by the power function. After that, time continuity  and  spatial regularities can be consequently derived. Furthermore, we also discuss on the existence  of the derivatives $\partial_t$, $\partial_t^\alpha$ of the mild solution  in some appropriate  spaces. 
\begin{theorem} \label{t1}
	Assume that $f \in  \mathbb H^{\nu+\theta}(\Omega) $ and   $G$ sastisfies $(\mathscr H_1)$ such that  $  \| L_1\|_{L^\infty (0,T)} \in \big( 0; 	\mathscr M^{-1}_1  \big)$ with $\nu \ge 0$ and  $\theta$ satisfies that  $\displaystyle \frac{\alpha-1}{\alpha}<\theta<1$, where the constant $\mathscr M_1$ is given by \textbf{(AP.4)} in the Appendix. 	Then Problem (\ref{mainprob1})-(\ref{mainprob4}) has a unique mild solution $$ u \in    L^p (0,    T ;  \mathbb H^{\nu} (\Omega)) \cap C^{\alpha(1-\theta)} \big((0,T];\mathbb H^{\nu} (\Omega)\big)  ,$$ for all $p \in \left[1; \frac{1}{\alpha(1-\theta)}\right)$, which corresponds to the estimate 
	\begin{align}
	\left\| u(t) \right\|_{\mathbb H^{\nu} (\Omega)} \lesssim  t^{-\alpha(1-\theta)}  \|f \|_{   \mathbb H^{\nu+\theta}(\Omega) }. \label{vanphong}
	\end{align}
The following spatial and time  regularities also hold:   
	\begin{itemize}
		\item[a)] Let   $\theta'$   satisfy that  $\displaystyle \frac{\alpha-1}{\alpha}<\theta'\le \theta$. Then $ u\in \displaystyle     L^p (0,    T ;  \mathbb H^{\nu+\theta-\theta'} (\Omega)) $, for all $p \in \left[1, \frac{1}{\alpha(1-\theta')}\right)$, which corresponds to the estimate
		$$\left\| u(t) \right\|_{\mathbb H^{\nu+\theta-\theta'} (\Omega)} \lesssim  t^{-\alpha(1-\theta')}  \|f \|_{   \mathbb H^{\nu+\theta}(\Omega) }   .$$
		\item[b)] Let  $1-\theta < \nu' \le 2-\theta$. Then $u\in  C^{\,\eta_{glo}} \hspace*{-0.05cm} \left([0,T];\mathbb H^{\nu-\nu'} (\Omega)\right)$  and  
		$$\left\| u(\widetilde t)- u(t) \right\|_{\mathbb H^{\nu-\nu'} (\Omega)}      \lesssim (\widetilde t-t)^{\,\eta_{glo}} \|f \|_{   \mathbb H^{\nu+\theta}(\Omega) }. $$
		Here $\eta_{glo}$ is defined in the Appendix.
		
		\item[c)] Let    $\displaystyle 0 \le  \nu_1 \le \min \left\{ 1-\theta;~\frac{\alpha-1}{\alpha} \right\}$. Then $ \partial_t  u \in   L^p (0,    T ;  \mathbb H^{\nu-\nu_1-\frac{1}{\alpha}}(\Omega))$, for all $p \in \left[1; \frac{1}{\alpha(1-\theta-\nu_1)}\right)$, which corresponds to the estimate  
		\begin{align*}
		\Vert \partial_t 
		u(t)\Vert_{\mathbb H^{\nu-\nu_1-\frac{1}{\alpha}}(\Omega)}   \lesssim t^{- \alpha (1-{ \theta } -\nu_1  )} \|f\|_{\mathbb{H}^{\nu+\theta}(\Omega)} .
		\end{align*}	
		\item[d)]   Let     $\displaystyle \frac{\alpha-1}{\alpha}-\theta <  \nu_\alpha \le \min \left\{  \frac{1}{\alpha} -\theta;   \frac{\alpha-1}{\alpha} \right\}$, then $ \partial_t^{\alpha}  u \in   L^p (0,    T ;  \mathbb H^{\nu-\nu_\alpha-\frac{1}{\alpha}}(\Omega))$, for any  $p \in \left[1; \frac{1}{\alpha\min \left\{ (1-{ \theta } -\nu_\alpha  );  (1-\theta) \right\}}\right)$, which corresponds to the estimate  
		\begin{align}
		\Vert \partial_t^\alpha 
		u(t)\Vert_{\mathbb H^{\nu-\nu_\alpha-\frac{1}{\alpha}}(\Omega)}   
		\lesssim t^{- \alpha \min \left\{ (1-{ \theta } -\nu_\alpha  );  (1-\theta) \right\} } \|f\|_{\mathbb{H}^{\nu+\theta}} . \nn
		\end{align}
	\end{itemize}
 	The hidden constants  (as using the notation $\lesssim$)  depend  only on $\al, \nu, \theta, T$ in the inequality (\ref{vanphong}), on $\al, \nu, \theta, \theta', T$ in Part a, on $\al, \nu, \theta, \nu', T$ in Part b, on $\al, \nu, \theta, \nu_1, T$ in Part c, and on $\al, \nu, \theta, \nu_\alpha, T$ in Part d.  
\end{theorem}

{
\begin{remark} We note that, the inequality (\ref{vanphong}) also guarantees the continuous dependence of the solution on the final value $f$. In fact, if we denote $u(f)$ and $u(\widetilde{f})$ by the mild solutions of Problem (\ref{mainprob1})-(\ref{mainprob4}) corresponding to the final value $f$ and $\widetilde{f}$, then one can see that
\begin{align*}
\left\| u(f) - u(\widetilde{f}) \right\|_{C^{\alpha(1-\theta)} \big((0,T];\mathbb H^{\nu} (\Omega)\big)} \lesssim     \|f - \widetilde{f} \|_{   \mathbb H^{\nu+\theta}(\Omega) }.
\end{align*} 
This concludes that Problem (\ref{mainprob1})-(\ref{mainprob4}) is well-posedness on $C^{\alpha(1-\theta)} \big((0,T];\mathbb H^{\nu} (\Omega)\big)$. 
\end{remark}
}

\begin{proposition} By Theorem \ref{t1}, the smoothness of the mild solution  can be summarized together as    
	\begin{align*}
	\begin{array}{llllcccc}
	\displaystyle u\in \left\{ \bigcup_{1\le p < \frac{1}{\alpha(1-\theta')} }  L^p (0,    T ;  \mathbb H^{\nu+\theta-\theta'} (\Omega)) \right\}   \cap  C^{\alpha(1-\theta)} \big((0,T];\mathbb H^{\nu} (\Omega)\big)  ,  
	\end{array} 
	\end{align*}
and 
\begin{align*}
	\left\{
	\begin{array}{llllcccc}
	\displaystyle \partial_t  u \in \bigcup_{1 \le p <  \frac{1}{\alpha(1-\theta-\nu_1)}}   L^p (0,    T ;  \mathbb H^{\nu-\nu_1-\frac{1}{\alpha}}(\Omega)), \hspace*{0.92cm}\\ \partial_t^{\alpha}  u \in \bigcup_{1\le p <  \frac{1}{\alpha(1-\theta-\nu_\alpha)}}   L^p (0,    T ;  \mathbb H^{\nu-\nu_\alpha-\frac{1}{\alpha}}(\Omega)),   
	\end{array} 
	\right.
	\end{align*}
where the values of the parameters are given in Theorem \ref{t1}. Moreover, the spatial regularity in Part b shows how the best spatial regularity that the mild solution $u$ can achieve. Then, by using some suitable Sobolev embeddings, one can derive the Gradient and Laplacian estimates for the solution on $L^q$ spaces.   
\end{proposition}

\begin{remark} In fact, one can investigate the continuity of the first order derivative $\partial_t u$ which is established in Part d of the above theorem. Moreover, if the nonlinearity $G$ is continuous in the time variable $t$, for instant, $G$ verifies that 
	\begin{align*}
	\left\| G(t_1,v_1)-G(t_2,v_2) \right\|_{\mathbb{H}^\nu(\Omega)} \lesssim |t_1-t_2|^{\mathrm{positive~exponent}} + \left\| v_1-v_2 \right\| _{\mathbb{H}^\nu(\Omega)},
	\end{align*}
	then one can establish the continuity of the fractional derivative $\partial^\alpha_t$ of the solution. 
\end{remark}
In Theorem \ref{t1}, under assumption $(\mathscr H_1)$, we do not obtain the regularity results of $u$ in the spaces $   C\left([0,T]; \mathbb H^\nu(\Omega)\right)$ or $  L^\infty (0,    T ;  \mathbb H^{\nu} (\Omega))$. The main reason is that the information at the initial time  $u(0)$ does not actually exist on $\mathbb{H}^\nu(\Omega)$. To overcome this restriction, we are going to consider the existence of a mild solution in the spaces $   C\left([0,T]; \mathbb H^\nu(\Omega)\right)$ or $  L^\infty (0,    T ;  \mathbb H^{\nu} (\Omega))$ by imposing the assumption ($\mathscr H_2$) on the nonlineariy $G$. In addition, it is necessary to suppose a   smoother assumption on the final value data $f$. 
In the following  theorem, we will build up this   existence and also a regularity result for the mild solution by using the Banach fixed-point theorem. Let us recall the fact that  the embedding $\mathbb{H}^{\nu+1}(\Omega)  \hookrightarrow \mathbb{H}^\sigma(\Omega)$ holds as $0\le \sigma\le \nu+1$.  So,  there exists a positive constant $C_2 (\nu, \sigma)$ such that
\begin{align}
\big\| v \big\|_{\mathbb{H}^\sigma(\Omega)} \le C_2(\nu,\sigma)  \big\| v \big\|_{\mathbb{H}^{\nu+1}(\Omega)},  \label{Embed2}
\end{align}
for all $v \in \mathbb{H}^{\nu+1}(\Omega)$.    

\begin{theorem} \label{theo2}
	Let  $ \frac{\alpha q-1}{\alpha q}<\theta<1$,  $0\le \nu\le \sigma \le \nu +1 $ and $1\le q < \frac{1}{\alpha(1-\theta)}$. Assume that $f \in  \mathbb H^{\nu+\theta+1}(\Omega) $, and  $G$ sastisfies $(\mathscr H_2)$ with $   L_2 \in \left (0;   \mathscr M_2^{-1} \right )$ where $\mathscr M_2$ is given by \textbf{(AP.4)} in the Appendix.  
	Then,  Problem (\ref{mainprob1})-(\ref{mainprob4}) has a unique mild solution $$u \in   C\left([0,T]; \mathbb H^\nu(\Omega)\right) \cap  L^q(0,T; \mathbb H^{\sigma}(\Omega)).  $$ 
	Moreover, we have
	\begin{align} \label{rs1}
 	\sup_{0\le t\le T}  \big\|u(t)\big\|_{  \mathbb H^\nu(\Omega)   } + \left(\int_0^T  \big\|u(t)\big\|^q_{  \mathbb H^{\sigma}(\Omega) } dt \right)^{1/q} \lesssim \|f\|_{\mathbb H^{\nu+\theta+1}(\Omega)}  .
	\end{align}
	
\end{theorem}

\subsection{Well-posedness  of Problem (\ref{mainprob1})-(\ref{mainprob4}) under critical nonlinearities case}  
The previous subsection states the results in the  globally Lipschitz case,  they cannot virtually be applied in many models such as  
time fractional Ginzburg-Landau, Allen–Cahn, Burgers, Navier-Stokes, Schrodinger, etc. equations.  
In this subsection, we state the well-posedness of Problem (\ref{mainprob1})-(\ref{mainprob4}) under the critical nonlinearities case.

\begin{theorem} \label{locally1}
	Assume that $\al \in (1,2), ~\sigma \in (-1,0) $,~$0< \nu< 1+\sigma$ and   $s>0$.
	Let $\vartheta$ such that   $\vartheta \in (\nu-\sigma, 1)$ and set  $\mu=\nu-\sigma$.
	Let $\zeta$ satisfy that
	\begin{equation}
	\zeta  < \min \Big( \al^{-1}- (1+s) \vartheta ,  \vartheta(1-s)-\nu+\sigma \Big).
	\end{equation}
	The function $G$ satisfies $(\mathscr H_3)$  if $G: [0,T] \times \mathbb{H}^{\nu}(\Omega) \longrightarrow \mathbb{H}^{ {\sigma}}(\Omega)$   such that $G(0)=0$ and 
	\begin{align} \label{H3}
	\|G(v_1)-G(v_2)\|_{\mathbb{H}^{ {\sigma}}(\Omega)} \le {L}_3(t)  \Big( 1+\|v_1\|^{s}_{\mathbb{H}^{\nu}(\Omega)}+\|v_2\|^{s}_{\mathbb{H}^{\nu}(\Omega)} \Big) \|v_1-v_2\|_{\mathbb{H}^{\nu}(\Omega)}.
	\end{align}
	where ${L}_3$ satisfies that
	$
	{L}_3(t) t^{\al \zeta } \in  L^\infty (0,T).
	$
	Set
	\[
	\mathfrak X_{\al,\vartheta, \nu,T }(\mathcal R) := \Big\{ w \in   C^{\alpha \vartheta }((0,T];\mathbb{H}^{\nu}(\Omega)),~~\|w\|_{  C^{\alpha \vartheta }((0,T];\mathbb{H}^{\nu}(\Omega))} \le \mathcal R \Big\}.
	\]
	If $f\in \mathbb{H}^{\nu+ (1-\vartheta) }(\Omega)$ and 
	$ K_0T^{s\alpha \vartheta} \in \big(  0;\min\big\{\frac{1}{2}\overline{\mathscr N_2}^{\,-1};\mathcal N_f\big\}  \big)$ with $K_0= \|	{L}_3(t) t^{\al \zeta }\|_{L^\infty(0,T)} 
	$, where the constants are formulated by \textbf{(AP.4)} in the Appendix, 
	then
	Problem (\ref{mainprob1})-(\ref{mainprob4}) has a unique mild solution $u \in 	\mathfrak X_{\al,\vartheta, \nu,T }(\widehat{\mathcal R})  $ with      respect to the estimate  
	\begin{align}
	\left\|u(t)\right\|_{\mathbb{H}^{\nu}(\Omega)} \lesssim t^{-\alpha \vartheta } \norm{f} _{\mathbb{H}^{\nu+(1- \vartheta)}(\Omega)}. \label{ubound}
	\end{align} 
	Moreover,  we obtain the  following spatial and time  regularities     
	\begin{itemize}
		\item[a)]   Let $\vartheta \le \vartheta' \le 1$ and $\alpha\vartheta-1\le\beta \le \alpha \vartheta$, then $t^\beta u \in  L^p (0,    T ;  \mathbb H^{\nu+(\vartheta'-\vartheta)} (\Omega))$, for all $p\in \left[1;\frac{1}{\alpha\vartheta-\beta}\right)$, with respect to the estimate 
		\begin{align*}
		\big\| u(t) \big\|_{\mathbb H^{\nu+(\vartheta'-\vartheta)} (\Omega) } \lesssim t^{-\alpha\vartheta'} \left\|f \right\|_{\mathbb{H}^{\nu+(1-\vartheta)}(\Omega)} . 
		\end{align*}
		\item[b)] Let $\vartheta<\eta \le \vartheta+1$ then $u\in \mathscr C^{ \,\eta_{cri} } \hspace*{-0.05cm} \left([0,T];\mathbb H^{\nu-\eta} (\Omega)\right)$  and   
		$$\left\| u(\widetilde t)- u(t) \right\|_{\mathbb H^{\nu-\eta} (\Omega)}      \lesssim (\widetilde t-t)^{\,\eta_{cri}} \left\|f \right\|_{\mathbb{H}^{\nu+(1-\vartheta)}(\Omega)} . $$
		where $\eta_{cri}$ is defined in the Appendix. 
	\end{itemize}
	 	The hidden constants  (as using the notation $\lesssim$)  depend  only on $\alpha,\mu ,\vartheta ,\zeta,s,T$ in the inequality (\ref{ubound}), on $\alpha,\mu ,\vartheta,\vartheta' ,\zeta,s,T$ in Part a, and  on $\alpha,\mu ,\vartheta,\eta ,\zeta,s,T$ in Part b. 
\end{theorem}
 
	\begin{remark} One can actually establish the existence of the derivatives $\partial_t u$ and $\partial_t^\alpha$ of the solution as follows
		\begin{itemize}
		\item[i)] Assume that $\vartheta < \frac{\mu+1}{2}$ and let $\vartheta_1\in \left[\vartheta ; \frac{\nu-\sigma+1}{2}\right)$, then $\partial_t u(t) \in \mathbb H^{\sigma+\vartheta_1-1}({\Omega})$ for each $t>0$ which corresponds to the estimate 
		\begin{align}  
		&\left\| \partial_t u(t) \right\|_{\mathbb H^{\sigma+\vartheta_1-1}({\Omega})}     \lesssim t^{-\alpha\left(2\vartheta_1-\mu-\frac{\alpha-1}{\alpha}\right)} \left\| f \right\| _{\mathbb{H}^{\nu+(1-\vartheta)}(\Omega)}.\label{der1}
		\end{align}
		\item[ii)] Assume that $\vartheta \le \frac{2\nu-2\sigma+1}{3}$ and let $\vartheta_\alpha\in \left[ \frac{\nu-\sigma+2}{3}; \frac{\nu-\sigma+5}{3} \right)$, then $\partial_t^\alpha u(t) \in \mathbb H^{\sigma+\vartheta_\alpha-2} (\Omega)$ for each $t>0$  which corresponds to the estimate  
		\begin{align}
		\left\| \partial_t^\alpha u(t) \right\|_{\mathbb H^{\sigma+\vartheta_\alpha-2} (\Omega)} \lesssim  t^{- \max \left\{\alpha \left( \vartheta_\alpha - \frac{\mu+2}{3}\right) ; \alpha\left( 2\vartheta - \mu \right) \right\} } \|f\|_{\mathbb H^{\nu+(1-\vartheta)}(\Omega)} . \label{der2} 
		\end{align}
	\end{itemize}
	The above results are weaker than Parts c, d of Theorem \ref{t1} since the powers in (\ref{der1}), (\ref{der2}) are really less than $-1$. So, we cannot obtain the existence  of $\partial_tu$, $\partial^\alpha_t u$ in the $L^p$ spaces with respect to the time variable. This obviously comes from the critical property of the nonlinearity. 
	\end{remark}

\section{Applications}
In this section we apply the theory developed in this work to some well-known equations. The classes of {\it time fractional Ginzburg-Landau  equation} and  {\it time fractional  Burgers equation},  are studied in $L^q$ ($q\ge 1$)
settings via interpolation–extrapolation scales and dual interpolation–extrapolation
scales of Sobolev spaces. We will discuss both time and spatial regularity of solutions by considering 
\begin{itemize}
\item The time continuities of solutions on $L^q$ spaces with respect to the intervals $(0;T]$, $[0,T]$;
\item The Gradient and Laplacian estimates for the solutions on $L^q$ spaces. 
\end{itemize}

\subsection{Time fractional Ginzburg-Landau equation}
We discuss now an application of our methods to a final value problem for a time fractional Ginzburg-Landau equation which is stated as follows  
\begin{align}
\left\{ \begin{array}{llllllcccccc}
\displaystyle  \partial_t^{\,\alpha }  u(x,t) +\mathcal L u(x,t) &=& \rho(t) |u(x,t)|^{s} u(x,t)  , &&  x \in\Omega , \hspace*{0.25cm} t\in (0,T),\\
u(x,t) &=& 0, && x \in \partial 
\Omega,\, t\in (0,T), \\
\partial_t  u(x,0) &=& 0, && x \in \Omega,   \\
\end{array}
\right. \label{GinLan}
\end{align}
associated with the final value data 
\eqref{mainprob4} and 
where $s>0$ is a given number.  

\begin{theorem} Assume that $2\le N\le 4$. \vspace*{0.2cm}
	
	\noindent a) \textbf{The case $0<s < \dfrac{4}{N}$:} Let the numbers $\varrho$, $\nu$, $\sigma$, $\mu$, $\vartheta$, $\vartheta'$ respectively satisfy that 
	$$ \varrho\in \left(0; \frac{Ns}{8}\right],~ \nu\in \left[ \frac{N}{4}-\frac{\varrho}{s} ;\frac{N}{4}\right),~  \mu\in \left(\mu_0;  \frac{N+4}{8} \right), ~ \vartheta\in \left( \mu;\frac{N+4}{8}\right), ~ \vartheta'\in  \left[ \vartheta + \frac{4-N}{8} ;1\right),$$
	where $\mu=\nu-\sigma$ and $\mu_0:=\max\left\{ \nu;s\left(\frac{N}{4}-\nu\right) \right\}$. If $f\in \mathbb{H}^{\nu+(1-\vartheta)}$, and $\rho(t)\le C_{\rho} t^{b}$ with $b >- \min \left\{ \frac{1}{\alpha}-(1+s)\vartheta ; (\vartheta-\mu)-s\vartheta  \right\}$ and  $C_\rho$ is small enough, then Problem (\ref{GinLan}) has a unique mild solution $u$ such that \vspace*{0.1cm}
	\begin{itemize}
		\item[•] \textbf{(Time regularity)} $u\in   C^{\alpha\vartheta}((0,T];L^4(\Omega)) \cap   C^{ \,\eta_{cri} } \hspace*{-0.05cm} \left([0,T];\mathbb H^{\nu-\eta} (\Omega)\right)$ where $\vartheta<\eta \le \vartheta+1$. This solution satisfies the estimate 
		\begin{align}
		t^{ \alpha \vartheta } \left\|u(t)\right\|_{L^4(\Omega)}\mathbf{1}_{t > 0} + \frac{\left\| u(t+\gamma)- u(t) \right\|_{\mathbb H^{\nu-\eta} (\Omega)}     }{\gamma^{\,\eta_{cri}}} \mathbf{1}_{t\ge 0} \lesssim  \norm{f} _{\mathbb{H}^{\nu+(1- \vartheta)}(\Omega)}. \label{bogiaoduca}
		\end{align}
		\item[•] \textbf{(Spatial regularity)} For each $t>0$, $u(t)$ belongs to $W^{1,\frac{4N}{3N-8\nu}} (\Omega)$ and verifies the estimate  
		\begin{align}
		t^{ \alpha\vartheta'} \big\| \nabla u(t) \big\|_{L^{\frac{4N}{3N-8\nu}}(\Omega)} + t^{ \alpha\vartheta'} \big\|(-\Delta)^{  \vartheta'-\vartheta } u(t) \big\|_{L^4(\Omega)}    \lesssim \left\|f \right\|_{\mathbb{H}^{\nu+(1-\vartheta)}(\Omega)}. \label{bogiaoducca}
		\end{align}
	\end{itemize} 
	
	\vspace*{0.2cm}
	
	\noindent b) \textbf{The case $s \ge  \dfrac{4}{N}$:}  Let the numbers  $\nu$, $\sigma$, $\mu$, $\vartheta$, $\vartheta'$ respectively satisfy that 
	$$   \nu\in \left[ \frac{N}{4}-\frac{1}{2s} ;\frac{N}{4}\right),~  \mu\in \left(\mu_0;  \frac{N+4}{8} \right), ~ \vartheta\in \left( \mu;\frac{N+4}{8}\right), ~ \vartheta'\in  \left[ \vartheta + \frac{4-N}{8} ;1\right),$$
	whereupon $\mu=\nu-\sigma$  and $\mu_0:=\max\left\{ \nu;s\left(\frac{N}{4}-\nu\right) \right\}$. If $f\in \mathbb{H}^{\nu+(1-\vartheta)}$, and $\rho(t)\le C_{\rho} t^{b}$ such that $b >- \min \left\{ \frac{1}{\alpha}-(1+s)\vartheta ; (\vartheta-\mu)-s\vartheta  \right\}$ and  $C_\rho$ is small enough, then Problem (\ref{GinLan}) has a unique mild solution $u$ such that \vspace*{0.1cm}
	\begin{itemize}
		\item[•] \textbf{(Time regularity)} $u\in   C^{\alpha\vartheta}((0,T];L^{Ns}(\Omega)) \cap   C^{ \,\eta_{cri} } \hspace*{-0.05cm} \left([0,T];\mathbb H^{\nu-\eta} (\Omega)\right)$ where $\vartheta<\eta \le \vartheta+1$. This solution satisfies the estimate 
		\begin{align}
		t^{ \alpha \vartheta } \left\|u(t)\right\|_{L^{Ns}(\Omega)}\mathbf{1}_{t > 0} + \frac{\left\| u(t+\gamma)- u(t) \right\|_{\mathbb H^{\nu-\eta} (\Omega)}     }{\gamma^{\,\eta_{cri}}} \mathbf{1}_{t\ge 0} \lesssim  \norm{f} _{\mathbb{H}^{\nu+(1- \vartheta)}(\Omega)}. \label{bogiaoducc}
		\end{align}
		\item[•] \textbf{(Spatial regularity)} For each $t>0$, $u(t)$ belongs to $W^{1,\frac{4N}{3N-8\nu}} (\Omega)$ and verifies the estimate  
		\begin{align}
		t^{ \alpha\vartheta'} \big\| \nabla u(t) \big\|_{L^{\frac{4N}{3N-8\nu}}(\Omega)} + t^{ \alpha\vartheta'} \big\|(-\Delta)^{  \vartheta'-\vartheta } u(t) \big\|_{L^{Ns}(\Omega)}    \lesssim \left\|f \right\|_{\mathbb{H}^{\nu+(1-\vartheta)}(\Omega)}. \label{bogiaoduccb}
		\end{align}
	\end{itemize} 
	
\end{theorem}

\begin{proof} 
	a) This proof will be based on applying and improving Theorem \ref{locally1} actually. We firstly exhibit some explanations justifying that the assumptions in this part are suitable.    \vspace*{0.15cm}
	\begin{itemize}
		\item $0<\varrho <\dfrac{1}{2}$    since    $\dfrac{Ns}{8}<\dfrac{N\left(\frac{4}{N}\right)}{8} = \dfrac{1}{2}$ as the assumption $s<\dfrac{4}{N}$; \vspace*{0.15cm}
		\item $0<\varrho \le   \dfrac{3N-4}{8}s $   since   $\dfrac{Ns}{8} \le \frac{Ns}{\left(\frac{8N}{3N-4}\right)} = \dfrac{3N-4}{8}$ by the fact that $ 8 \ge  \dfrac{8N}{3N-4}$; \vspace*{0.15cm}
		\item $\nu\ge   \dfrac{N}{8}$ since   $\nu \ge \dfrac{N}{4}-\dfrac{\varrho}{s} \ge \dfrac{N}{4} - \dfrac{N}{8} = \dfrac{N}{8}$  as the assumption $\nu\ge \dfrac{N}{4}-\dfrac{\varrho}{s}$ and $\varrho  \le \dfrac{Ns}{8}$. \vspace*{0.15cm}
		\item The interval $\displaystyle \left(\mu_0;  \frac{N+4}{8} \right)$ is not really empty. 
		Indeed, it is easy to see from $\nu<\dfrac{N}{4}$ and $\dfrac{N}{4}<\dfrac{N+4}{8}$ that $\nu<\dfrac{N+4}{8}$. Moreover, we have   $$\displaystyle  s\left(\dfrac{N}{4}-\nu\right) \le  s \left( \frac{N}{4}- \left( \frac{N}{4} - \frac{\varrho}{s} \right) \right) = \varrho <\frac{1}{2} < \dfrac{N+4}{8},$$   by  using   assumption $\nu>\dfrac{N}{4}-\dfrac{\varrho}{s}$ and noting that $\varrho <\dfrac{1}{2}$. \vspace*{0.15cm}
		\item The interval $\displaystyle \left[ \vartheta + \frac{4-N}{8} ;1\right)$ is also not empty as $\vartheta<\dfrac{N+4}{8}$. \vspace*{0.15cm}
		\item The number $\sigma$ belongs to $\left(-\dfrac{N}{4};0\right)$ since $\sigma=\nu-\mu \le \mu_0 - \mu <0$, and furthermore 
		$$ \nu-\mu > \frac{N}{4} -\frac{\varrho}{s}-\frac{N+4}{8} \ge \frac{N}{4} - \frac{3N-4}{8} - \frac{N+4}{8} = -\frac{N}{4}, $$ 
		by using the assumption $\nu\ge \dfrac{N}{4}-\dfrac{\varrho}{s}$ and the fact that $\varrho <  \dfrac{3N-4}{8}s $.  
	\end{itemize}
	
	Secondly, we obtain some important Sobolev embeddings which help to establish the existence of a mild solution. By applying the embeddings (\ref{nhung1b})-(\ref{HSem}), and   the dualities $\left[\mathbb{H}^{-\sigma}(\Omega)\right]^*=\mathbb{H}^{\sigma}(\Omega)$, $\left[W^{-2\sigma,2}(\Omega)\right]^*=W^{2\sigma,2}(\Omega)$, one can see that   \vspace*{0.15cm}
	\begin{itemize}
		\item The Sobolev embedding  $\displaystyle  L^{\frac{2N}{N-4\sigma}}(\Omega)   \hookrightarrow  W^{2\sigma,2}(\Omega) $ holds  as $-\dfrac{N}{2} <   2\sigma < 0, ~ \dfrac{2N}{N-4\sigma} =  \dfrac{2N}{N-2(2\sigma)}$; \vspace*{0.15cm}
		\item The Sobolev embedding $\mathbb{H}^{-\sigma} (\Omega)   \hookrightarrow W^{-2\sigma,2} (\Omega)$ holds  as $-\sigma > 0$, which implies that $W^{2\sigma,2}(\Omega)\hookrightarrow \mathbb{H}^{\sigma}(\Omega)$ holds.  \vspace*{0.15cm}
	\end{itemize}
	As a consequence of the above embeddings, we obtain the following Sobolev  embedding
	\begin{align}
	\displaystyle  L^{\frac{2N}{N-4\sigma}}(\Omega) \hspace*{0.1cm} \hookrightarrow  \mathbb{H}^{\sigma}(\Omega). \label{nhacdoquaA}
	\end{align}
	By the assumption $s\left(\frac{N}{4}-\nu\right) \le \mu_0<\mu$, we have  
	\begin{align*}
	\frac{2N(1+s)}{N-4\sigma} \,&= \frac{2N(1+s)}{N -  4(\nu-\mu)}  
	\le   \frac{2N(1+s)}{N -  4\nu+4s\left(\frac{N}{4}-\nu \right) }   = \frac{2N}{N-4\nu} .
	\end{align*}
	Therefore, using (\ref{nhung1b}) yields that  
	$$W^{2\nu,2}(\Omega) \hookrightarrow L^{\frac{2N (1+s)}{N-4\sigma}} (\Omega)  \quad \textrm{since} \quad  0\le 2\nu <\frac{N}{2}, ~ \frac{2N(1+s)}{N-4\sigma} \le   \frac{2N}{N-4\nu}. $$
	Besides, using (\ref{HSem}) invokes that  
	$\mathbb{H}^{\nu}(\Omega) \hookrightarrow W^{2\nu,2}(\Omega),$ as  $ \nu \ge 0, $
	which consequently infers the embedding 
	\begin{align}
	\mathbb{H}^{\nu}(\Omega) \hookrightarrow 
	L^{\frac{2N (1+s)}{N-4\sigma}} (\Omega). \label{nhacdoquaB}
	\end{align}
	
	Thirdly, let us set the nonlinearity   $G(v):=\rho(t) |v|^{s}v$, and show that $G$ satisfies ($\mathscr H_3$). Indeed, it is obvious that 
	$ | G(v_1)-G(v_2) | $ is pointwise bounded by $ (1+s) \left( |v_1|^{s} + |v_2|^{s} \right)|v_1 - v_2| , $ 
	and so one can derive the following chain of   estimates  
	\begin{align}
	\left\| G(v_1)-G(v_2) \right\| _{\mathbb{H}^\sigma(\Omega)} \,& \lesssim   \left\| G(v_1)-G(v_2) \right\| _{L^{\frac{2N}{N-4\sigma}}(\Omega)} \nn\\      
	\,& \lesssim \rho(t)   \left[ \big\|  |v_1|^{s}  |v_1 - v_2|   \big\| _{L^{\frac{2N}{N-4\sigma}}(\Omega)} + \big\|  |v_2|^{s}  |v_1 - v_2|   \big\| _{L^{\frac{2N}{N-4\sigma}}(\Omega)} \right]  \nn\\
	\,& \lesssim \rho(t) \left(   \left\| v_1  \right\|^{s}_{L^{\frac{2N(1+s)}{N-4\sigma}} (\Omega)} + \left\| v_2  \right\|^{s}_{L^{\frac{2N(1+s)}{N-4\sigma}} (\Omega)} \right) \left\| v_1 - v_2  \right\| _{L^{\frac{2N (1+s)}{N-4\sigma}} (\Omega)}     \nn\\
	\,& \lesssim \rho(t) \left(   \left\| v_1  \right\|^{s}_{\mathbb{H}^\nu(\Omega)} + \left\| v_2  \right\|^{s}_{\mathbb{H}^\nu(\Omega)} \right) \left\| v_1 - v_2  \right\| _{\mathbb{H}^\nu(\Omega)}   ,  \nn
	\end{align} 
	where the embedding (\ref{nhacdoquaA}) has been used in the first estimate, the pointwise boundedness in the second estimate, the Holder inequality in the third one, and the embedding (\ref{nhacdoquaB}) in the last estimate. Therefore, we can take the Lipschitz coefficient in the form $K(t)=K_{\mathrm{Gin}} \rho(t)$ with some positive constant $K_{\mathrm{Gin}}$. Furthermore, the assumption $b >- \min \left\{ \frac{1}{\alpha}-(1+s)\vartheta ; (\vartheta-\mu)-s\vartheta  \right\}$ ensures that there always exists a real constant $\zeta$ such that 
	$$-b<\zeta<\min\left\{\frac{1}{\alpha}-(1+s)\vartheta;(\vartheta-\mu)-s\vartheta\right\},$$ 
	then one derives  
	$
	K(t) \le C_{\rho} K_{\mathrm{Gin}}   t^{-\alpha\zeta} t^{\alpha ( b+\zeta)} \le K_{0} t^{-\alpha\zeta}  ,
	$
	where $K_0=C_{\rho} K_{\mathrm{Gin}}T^{\alpha (b+ \zeta)}$. We conclude that $\zeta$ and  $G$ satisfy   assumptions of   Theorem \eqref{locally1}. It is obvious that all assumptions of this theorem also fulfill the assumptions of Theorem \ref{locally1}. Thus,  applying Theorem \ref{locally1} invokes that Problem (\ref{GinLan}) has a unique mild solution 
	$u \in C^{\alpha\vartheta}((0,T];\mathbb{H}^\nu(\Omega))\cap  C^{ \,\eta_{cri} } \hspace*{-0.05cm} \left([0,T];\mathbb H^{\nu-\eta} (\Omega)\right)$ with $C_\rho$ is small enough. 
	Now, the assumption $\nu\ge \dfrac{N}{4}-\dfrac{\varrho}{s}$ implies that 
	\begin{align}
	\frac{2N}{N-4\nu} \ge \frac{2N}{N-4\left(\frac{N}{4}-\frac{\varrho}{s}\right)} = \frac{2N}{\left(4\frac{\varrho}{s}\right)} \ge \frac{2N}{\left(4\frac{N}{8}\right)} = 4,  \label{hetmuchaicayviet}
	\end{align}
	where $\dfrac{\varrho}{s} \le \dfrac{N}{8}$. Hence, we infer from  $\Omega$ is a bounded domain that 
	$ L^{\frac{2N}{N-4\nu}}(\Omega)  \hookrightarrow  L^4(\Omega) .$
	Besides, applying the embedding (\ref{nhung1a}) again combined with the above embedding allow that  
	$$ W^{2\nu,2}(\Omega) \hookrightarrow L^{\frac{2N}{N-4\nu}}(\Omega)  \hookrightarrow  L^4(\Omega) ,  $$
	where we note that $0<2\nu<\frac{N}{2}, ~ \frac{2N}{N-4\nu} = \frac{2N}{N-2(2\nu)}.$
	Therefore,  we deduce  that 
	$$u \in C^{\alpha\vartheta}((0,T];L^4(\Omega))\cap  C^{ \,\eta_{cri} } \hspace*{-0.05cm} \left([0,T];\mathbb H^{\nu-\eta} (\Omega)\right)$$
	where $\vartheta<\eta \le \vartheta+1$ as in Part b of Theorem \ref{locally1} and 
	\begin{align}
	t^{ \alpha \vartheta } \left\|u(t)\right\|_{L^4(\Omega)}\mathbf{1}_{t > 0} + \frac{\left\| u(t+\gamma)- u(t) \right\|_{\mathbb H^{\nu-\eta} (\Omega)}     }{\gamma^{\,\eta_{cri}}} \mathbf{1}_{t\ge 0} \lesssim  \norm{f} _{\mathbb{H}^{\nu+(1- \vartheta)}(\Omega)}. \nn
	\end{align} 
	This shows inequality (\ref{bogiaoduca}). Finally, we need to prove   inequality (\ref{bogiaoducca}). Indeed, we have 
	\begin{align}
	\vartheta'-\vartheta \ge \frac{4-N}{8} \quad \textrm{as} \quad \vartheta'\in  \left[ \vartheta + \frac{4-N}{8} ;1\right)   ,  \label{kieunua}
	\end{align}
	which associates with $\nu\ge \dfrac{N}{8}$ that 
	$\displaystyle   \nu+ (\vartheta'-\vartheta) \ge \frac{N}{8} + \frac{4-N}{8} = \frac{1}{2} . $ Therefore, we obtain \vspace*{0.15cm}
	\begin{itemize}
		\item The Sobolev embedding  $\displaystyle  \mathbb{H}^{\nu+(\vartheta'-\vartheta)}(\Omega)   \hookrightarrow  W^{2\nu+2(\vartheta'-\vartheta),2}(\Omega) $ holds as $\nu+(\vartheta'-\vartheta) > 0$. \vspace*{0.15cm}
		\item The Sobolev embedding $W^{2\nu+2(\vartheta'-\vartheta),2}(\Omega) \hookrightarrow W^{1,\frac{4N}{3N-8\nu}} (\Omega)$ holds  by using the embedding  (\ref{nhung1}) as $  \dfrac{4N}{3N-8\nu} \ge 1$ and $2\nu+ 2(\vartheta'-\vartheta) \ge  1$, where we note from (\ref{kieunua})  that 
		$$ 2\nu+2(\vartheta'-\vartheta) -1 \ge 2\nu + 2 \left(\frac{4-N}{8}\right) - 1 = 2\nu - \frac{N}{4} = \frac{N}{2} - \frac{N}{\left(\frac{4N}{3N-8\nu}\right)} .$$  
	\end{itemize}
	Two above  embeddings consequently infer that the Sobolev embedding 
	$\mathbb{H}^{\nu+(\vartheta'-\vartheta)}(\Omega) \hookrightarrow  W^{1,\frac{4N}{3N-8\nu}} (\Omega)$ holds. 
	Hence,  we deduce from Part a of Theorem \ref{locally1} that $u(t)\in W^{1,\frac{4N}{3N-8\nu}} (\Omega)$ with respect to the   estimate  
	\begin{align}
	t^{ \alpha\vartheta'} \big\| \nabla u(t) \big\|_{L^{\frac{4N}{3N-8\nu}}(\Omega)} + t^{ \alpha\vartheta'} \big\|(-\Delta)^{  \vartheta'-\vartheta } u(t) \big\|_{L^4(\Omega)}    \lesssim \left\|f \right\|_{\mathbb{H}^{\nu+(1-\vartheta)}(\Omega)}, \nn
	\end{align}
	which finalizes the proof of Part a of this theorem. 
	
	\vspace*{0.2cm}
	
	\noindent b)  We note from Part a that the   number $\varrho$  belongs to the interval $\left(0;\frac{1}{2}\right)$. In this part, we try to extend the method in Part a with $\varrho=\frac{1}{2}$.  It is important   to explain the similarities and differences between the numbers in this part from Part a as follows    
	\begin{itemize}
		\item $\nu\ge   \dfrac{N}{8}$ since   $\nu \ge \dfrac{N}{4}-\dfrac{1}{2s} \ge \dfrac{N}{4} - \dfrac{1}{2\left(\frac{4}{N}\right)} = \dfrac{N}{8}$  by employing  $\nu\ge \dfrac{N}{4}-\dfrac{1}{2s}$ and $s  \le \dfrac{4}{N}$. \vspace*{0.15cm}
		\item The interval $\displaystyle \left(\mu_0;  \frac{N+4}{8} \right)$ is not really empty since    $$\displaystyle  s\left(\dfrac{N}{4}-\nu\right) \le  s \left( \frac{N}{4}- \left( \frac{N}{4} - \frac{1}{2s} \right) \right) = \frac{1}{2} < \dfrac{N+4}{8}.$$   
		\item The number $\sigma$ belongs to $\left(-\dfrac{N}{4};0\right)$ since  
		$$ \nu-\mu > \frac{N}{4} -\frac{1}{2s}-\frac{N+4}{8} \ge \frac{N}{4} - \frac{1}{2\left(\frac{4}{N}\right)} - \frac{N+4}{8} = -\frac{1}{2} \ge -\frac{N}{4}, $$ 
		by also   employing  $\nu\ge \dfrac{N}{4}-\dfrac{1}{2s}$ and $s  \le \dfrac{4}{N}$.   \vspace*{0.15cm} 
	\end{itemize}
	By using the same methods as Part a, one can establish the existence and uniqueness of a mild solution $u$ to Problem (\ref{GinLan}) in 
	$ C^{\alpha\vartheta}((0,T];\mathbb{H}^\nu(\Omega))\cap  C^{ \,\eta_{cri} } \hspace*{-0.05cm} \left([0,T];\mathbb H^{\nu-\eta} (\Omega)\right)$ with $C_\rho$ is small enough. Next, the inequality (\ref{hetmuchaicayviet}) can be modified as 
	\begin{align}
	\frac{2N}{N-4\nu} \ge \frac{2N}{N-4\left(\frac{N}{4}-\frac{1}{2s}\right)} = Ns. \nn  
	\end{align}
	Hence, we obtain the Sobolev embedding 
	$$ W^{2\nu,2}(\Omega) \hookrightarrow L^{\frac{2N}{N-4\nu}}(\Omega)  \hookrightarrow  L^{Ns}(\Omega) , $$
	which implies inequality (\ref{bogiaoducc}). Moreover, we also have $\nu+(\vartheta'-\vartheta)\ge \frac{1}{2}$ by noting the assumption $\vartheta'\in  \left[ \vartheta + \frac{4-N}{8} ;1\right)$ and the fact that $\nu\ge \frac{N}{8}$. Then, we obtain  the Sobolev  embedding $$\mathbb{H}^{\nu+(\vartheta'-\vartheta)}(\Omega) \hookrightarrow  W^{1,\frac{4N}{3N-8\nu}} (\Omega),$$ and   inequality (\ref{bogiaoduccb}) also holds. We finally complete the proof.
\end{proof}
\subsection{Time fractional Burgers equation}
In this subsection, we deal with a terminal  value problem for a {\it time fractional Burgers equation}  which
is given by 
\begin{align}
\left\{ \begin{array}{llllllcccccc}
\displaystyle  \partial_t^{\,\alpha }  u(x,t)   + \rho(t) (u\cdot \nabla)u (x,t) &=& \Delta u(x,t) , &&  x\in \Omega,~0<t<T,  \\
u(x,t)   &=&0,   && \hspace*{1.27cm} 0<t<T, \\
\partial_t  u(x,0) &=& 0, && x\in \Omega,  
\end{array}
\right. \label{BurPro}
\end{align}
associated with the final value data 
\eqref{mainprob4}. 
Here    $f$ and $\rho$ are given functions, and the operator $\mathcal A$ is $-\Delta$ which acts on $L^2(\Omega)$ with its domain $W_0^{1,2}(\Omega) \cap  W^{2,2}(\Omega)$.   In the following, we will apply Theorem \ref{locally1}  to obtain a mild solution of Problem (\ref{BurPro}), and then obtain the spatial regularity with an $L^q$-estimate for $\nabla u$ and a  $\mathbb{H}^\nu $-estimate for $(-\Delta)^{\vartheta'-\vartheta}u$.

\begin{theorem} Assume that $3\le N\le 4$ ($N$ is  dimension of $\Omega$). Let the numbers $\nu$, $\sigma$, $\mu$, $\vartheta$, $\vartheta'$ respectively satisfy that 
	$$\nu\in \left[ \frac{1}{2};\frac{N}{4} \right), \quad \mu \in \left[\mu_2 ;    \frac{N+4}{8}\right) , \quad          ~ \vartheta\in \left( \mu;\frac{N+4}{8}\right), ~ \vartheta'\in  \left[ \vartheta + \frac{4-N}{8} ;1\right) ,$$ 
	where $\mu=\nu-\sigma$ and $\mu_2:=\max \left\{ \nu;\frac{N+2}{4}-\nu \right\}$. If   $\rho(t)\le C_{\rho} t^{b}$ with $b >- \min \left\{-\mu;\frac{1}{\alpha}-2\vartheta \right\}$,   $f\in \mathbb{H}^{\nu+(1-\vartheta)}$, and $C_\rho$ is small enough, then Problem (\ref{BurPro}) has a unique mild solution $u$ such that \vspace*{0.1cm} 
	\begin{itemize}
		\item[a)] \textbf{(Time regularity)}
		Let $\vartheta<\eta \le \vartheta+1$. Then we have
		$$ u\in  C^{\alpha\vartheta}((0,T];L^{\frac{2N}{4\mu-2}}(\Omega)) \cap   C^{ \,\eta_{cri} } \hspace*{-0.05cm} \left([0,T];\mathbb H^{\nu-\eta} (\Omega)\right),$$ 
		and  time  regularity result for $u$ holds
		\begin{align}
		t^{ \alpha \vartheta } \left\|u(t)\right\|_{L^{\frac{2N}{4\mu-2}}(\Omega)} + \frac{\left\| u(t+\gamma)- u(t) \right\|_{\mathbb H^{\nu-\eta} (\Omega)}     }{\gamma^{\,\eta_{cri}}} \lesssim  \norm{f} _{\mathbb{H}^{\nu+(1- \vartheta)}(\Omega)} . \label{bogiaoduca}
		\end{align}
		\item[b)] \textbf{(Spatial regularity)} For each $t>0$, $u(t)$ belongs to $W^{1,\frac{4N}{3N-4}}(\Omega)$ and satisfies the following estimate
		\begin{align}
		t^{ \alpha\vartheta'} \big\| \nabla u(t) \big\|_{L^{\frac{4N}{3N-4}}(\Omega)} + t^{ \alpha\vartheta'}  \big\|(-\Delta)^{  \vartheta'-\vartheta } u(t) \big\|_{\mathbb H^{\nu } (\Omega) }   \lesssim \left\|f \right\|_{\mathbb{H}^{\nu+(1-\vartheta)}(\Omega)}. \label{bogiaoduc}
		\end{align}
	\end{itemize} 
\end{theorem}
\begin{proof} 
	In order to prove this theorem, we will apply Theorem \ref{locally1}, and then     improve the time and spatial regularities of the mild solution.  Let us set $G(t,v)=-\rho(t)( v \cdot \nabla) v$ and  show that $G$ satisfies   assumption ($\mathscr H_3$) corresponding to $s=1$.  Firstly, we analyze the values of the numbers $\nu$, $\sigma$, $\mu$, $\vartheta$, $\vartheta'$ as follows
	\begin{itemize}
		\item The interval $\left[\mu_2 ;    \dfrac{N+4}{8}\right)$ is not empty since $\displaystyle \nu <\frac{N+4}{8}$ as $\displaystyle \nu<\frac{N}{4}  \le \frac{N+4}{8}$ (here $N\le  4$), and  
		$ \frac{N+2}{4} - \nu <\frac{N+2}{4} - \frac{1}{2}   <\frac{N+2}{4} - \frac{N}{8} = \frac{N+4}{8}.   $
		\item The numbers $\displaystyle \frac{N-4\sigma}{4\mu-2}$, $\displaystyle \frac{N-4\sigma}{N+2-4\nu}$ are greater than $1$. Indeed, $\mu\ge \mu_2 \ge \nu \ge \dfrac{1}{2}$, and  
		\begin{align*}
		\,& \frac{N-4\sigma}{4\mu-2} =  \frac{N+2-4\nu}{4\mu-2} + 1 > \frac{N+2-4\left(\frac{N}{4}\right)}{4\mu-2} + 1 > 1;\\
		\,& \frac{N-4\sigma}{N+2-4\nu} =  \frac{ 4\mu-2}{N+2-4\nu} + 1 > 1;
		\end{align*}
		Moreover, these are  the dual numbers of each other. \vspace*{0.15cm}
	\end{itemize}
	Therewith, one can obtain  the following chains of the  Sobolev embeddings  by applying (\ref{nhung1}), (\ref{nhung1a}), and (\ref{nhung1b}). Indeed, we have  
	\begin{itemize}
		\item The  Sobolev embedding $\displaystyle L^{\frac{2N}{N-4\sigma}}(\Omega) \hookrightarrow W^{2\sigma,2}(\Omega)$ holds as $ -\dfrac{N}{2}<2\sigma\le 0$, $\dfrac{2N}{N-4\sigma}=\dfrac{2N}{N-2(2\sigma)}$, and $W^{2\sigma,2}(\Omega) \hookrightarrow \mathbb{H}^\sigma(\Omega)$ as $\sigma\le 0$, and so that
		\begin{align}
		L^{\frac{2N}{N-4\sigma}}(\Omega) \hookrightarrow W^{2\sigma,2}(\Omega) \hookrightarrow \mathbb{H}^\sigma(\Omega) .\label{buoisangonhaA}
		\end{align}
		\item  The  Sobolev embedding $\mathbb{H}^{\nu}(\Omega) \hookrightarrow W^{2\nu,2}(\Omega)$ holds  as $\nu\ge 0$, and $W^{2\nu,2}(\Omega) \hookrightarrow W^{1,\frac{2N}{N+2-4\nu}}(\Omega)$ as $\nu\ge \dfrac{1}{2}$,  $2\nu-1 = \dfrac{N}{2}-N/ \left(\frac{2N}{N+2-4\nu}\right) $ which implies   the following  Sobolev embedding
		\begin{align}
		\mathbb{H}^{\nu}(\Omega) \hookrightarrow W^{2\nu,2}(\Omega) \hookrightarrow W^{1,\frac{2N}{N+2-4\nu}}(\Omega).\label{buoisangonhaB}
		\end{align}
		\item  The  Sobolev embedding $W^{2\nu,2}(\Omega) \hookrightarrow L^{ \frac{2N}{ 4\mu-2}}(\Omega)$ holds as $0<2\nu<\dfrac{N}{2}$, $1\le \dfrac{2N}{4\mu-2} \le \dfrac{2N}{N-4\nu}$ ($\nu<\mu$), and henceforth 
		\begin{align}
		\mathbb{H}^{\nu}(\Omega) \hookrightarrow W^{2\nu,2}(\Omega) \hookrightarrow L^{ \frac{2N}{ 4\mu-2}}(\Omega).\label{buoisangonhaC}
		\end{align}
		\item  The  Sobolev embedding $\mathbb{H}^{\nu+(\vartheta'-\vartheta)} (\Omega)\hookrightarrow W^{2\nu+2(\vartheta'-\vartheta),2}(\Omega) \hookrightarrow W^{2-\frac{N}{4},2}(\Omega)$ holds since $\nu+(\vartheta'-\vartheta) \ge \dfrac{1 }{2} + \dfrac{4-N}{8} = 1 - \dfrac{N}{8}$ by using the assumption $\vartheta'\in  \left[ \vartheta + \dfrac{4-N}{8} ;1\right)  $. In addition,  $W^{2-\frac{N}{4},2}(\Omega)  \hookrightarrow W^{1,\frac{4N}{3N-4}}(\Omega)$ as $2-\dfrac{N}{4} \ge  1 $, $2-\dfrac{N}{4}-1=\frac{N}{2}-N/\left(\dfrac{4N}{3N-4}\right)$. Therefore,  we obtain the following Sobolev embedding
		\begin{align}
		\mathbb{H}^{\nu+(\vartheta'-\vartheta)} (\Omega)\hookrightarrow W^{2\nu+2(\vartheta'-\vartheta),2}(\Omega) \hookrightarrow W^{2-\frac{N}{4},2}(\Omega)  \hookrightarrow W^{1,\frac{4N}{3N-4}}(\Omega) .\label{buoisangonhaD}
		\end{align}
	\end{itemize}
On account of  the above embeddings,   the the H\"older inequality  we deduce the following chain of   estimates 
	\begin{align*}
	\left\| G(v_1)-G(v_2)  \right\|_{\mathbb{H}^{{\sigma}}(\Omega)} \,&\lesssim \rho(t) \left( \left\| (v_1\cdot \nabla )(v_1-v_2)  \right\|_{\mathbb{H}^{{\sigma}}(\Omega)} + \left\| ((v_1-v_2)\cdot \nabla ) v_2   \right\|_{\mathbb{H}^{{\sigma}}(\Omega)} \right) \nn\\
	\,&\lesssim \rho(t) \left( \left\| (v_1\cdot \nabla )(v_1-v_2)  \right\|_{L^{\frac{2N}{N-4\sigma}}(\Omega)} + \left\| ((v_1-v_2)\cdot \nabla ) v_2   \right\|_{L^{\frac{2N}{N-4\sigma}}(\Omega)} \right) \nn\\
	\,&\lesssim \rho(t) \left( \left\| v_1  \right\|_{L^{\frac{2N}{N-4\sigma}\frac{N-4\sigma}{4\mu-2}}(\Omega)} \left\| \nabla( v_1-v_2)  \right\|_{L^{\frac{2N}{N-4\sigma}\frac{N-4\sigma}{N+2-4\nu}}(\Omega)} \right. \nn\\
	\,& \hspace*{3.1cm} + \left. \left\| v_1-v_2  \right\|_{L^{\frac{2N}{N-4\sigma}\frac{N-4\sigma}{4\mu-2}}(\Omega)} \left\| \nabla v_2   \right\|_{L^{\frac{2N}{N-4\sigma}\frac{N-4\sigma}{N+2-4\nu}}(\Omega)} \right) \nn\\
	\,&  = \rho(t) \left( \left\| v_1  \right\|_{L^{\frac{2N}{4\mu-2} }(\Omega)} \left\| \nabla( v_1-v_2)  \right\|_{L^{\frac{2N}{N+2-4\nu} }(\Omega)} \right. \nn\\
	\,& \hspace*{3.1cm} + \left. \left\| v_1-v_2  \right\|_{L^{\frac{2N}{4\mu-2} }(\Omega)} \left\| \nabla v_2   \right\|_{L^{\frac{2N}{N+2-4\nu} }(\Omega)} \right) \nn\\
	\,& \lesssim \rho(t) \left( \left\| v_1  \right\|_{\mathbb{H}^\nu(\Omega)} + \left\| v_2  \right\|_{\mathbb{H}^\nu(\Omega)} \right)  \left\| v_1- v_2  \right\|_{\mathbb{H}^\nu(\Omega)}, 
	\end{align*}
	where the chain (\ref{buoisangonhaA}) has been used in the first estimate, the triangle inequality in the second estimate, the Holder inequality   
	with the dual numbers $  \frac{N-4\sigma}{4\mu-2}$, $ \frac{N-4\sigma}{N+2-4\nu}$  in the third one,  the chains (\ref{buoisangonhaB}) and (\ref{buoisangonhaC}) in the last one.   		
	This means that $G$ is really a critical nonlinearity from $\mathbb{H}^{\nu}(\Omega)$ to $\mathbb{H}^{{\sigma}}(\Omega)$ with respect to $s=1$ and  $\mathfrak{N}(v_1,v_2)=\left\| v_1   \right\|_{\mathbb{H}^{\nu}(\Omega)} + \left\| v_2  \right\|_{\mathbb{H}^{\nu}(\Omega)}$. Furthermore, we can write  $K(t) = \rho(t)K_{\mathrm{Bur}}$ with some positive constant $K_{\mathrm{Bur}}$.    Let us take $\zeta$ satisfying that $$-b<\zeta<\min\left\{-\mu;\frac{1}{\alpha}-2\vartheta\right\},$$ then one has 
	$
	K(t) \le   K_{0} t^{-\alpha\zeta}  ,
	$
	where $K_0=C_{\rho} K_{\mathrm{Bur}}T^{\alpha (b+ \zeta)}$. Due to the above arguments, we consequently conclude that  $G$ fulfills the assumption ($\mathscr H_3$). One can check that all numbers in this theorem obviously satisfy the assumptions of Theorem \eqref{locally1}.   
	Hence, we can apply  Theorem \eqref{locally1} and the chain (\ref{buoisangonhaC}) ensures that Problem (\ref{BurPro}) has a unique mild solution $$ u \in    C^{\alpha\vartheta}((0,T]; L^{\frac{2N}{4\mu-2}}(\Omega) ) \cap  C^{ \,\eta_{cri} } \hspace*{-0.05cm} \left([0,T];\mathbb H^{\nu-\eta} (\Omega)\right), $$ 
	with $C_\rho$ being small enough.  Besides, the boundedness (\ref{ubound}) and   
	Part b of Theorem \eqref{locally1}  can be combined to imply the following estimate
	\begin{align}
	t^{ \alpha \vartheta } \left\|u(t)\right\|_{L^{\frac{2N}{4\mu-2}}(\Omega)} + \frac{\left\| u(t+\gamma)- u(t) \right\|_{\mathbb H^{\nu-\eta} (\Omega)}     }{\gamma^{\,\eta_{cri}}} \lesssim  \norm{f} _{\mathbb{H}^{\nu+(1- \vartheta)}(\Omega)},  \nn
	\end{align}
	i.e.,   inequality (\ref{bogiaoduca}) is easily obtained.   We now   prove the spatial regularity. Indeed,      Part a of Theorem \eqref{locally1} can be rewritten as $(-\Delta)^{  \vartheta'-\vartheta } u(t)\in \mathbb H^{\nu } (\Omega)$ with respect to the estimate   
	\begin{align*}
	\big\|(-\Delta)^{  \vartheta'-\vartheta } u(t) \big\|_{ L^{\frac{2N}{4\mu-2}}(\Omega)} \lesssim \big\|(-\Delta)^{  \vartheta'-\vartheta } u(t) \big\|_{\mathbb H^{\nu } (\Omega) } \lesssim t^{-\alpha\vartheta'} \left\|f \right\|_{\mathbb{H}^{\nu+(1-\vartheta)}(\Omega)} .    
	\end{align*}
	On the other hand, by using the chain (\ref{buoisangonhaD}), 
	we deduce   that 
	\begin{align*}
	\big\| \nabla u(t) \big\|_{L^{\frac{4N}{3N-4}}(\Omega)} \lesssim  \big\|(-\Delta)^{  \vartheta'-\vartheta } u(t) \big\|_{\mathbb H^{\nu } (\Omega) } \lesssim t^{-\alpha\vartheta'} \left\|f \right\|_{\mathbb{H}^{\nu+(1-\vartheta)}(\Omega)} ,  
	\end{align*}
	which implies   inequality (\ref{bogiaoduc}).   
\end{proof}

\section{Proof of Theorem \ref{t1}, Theorem \ref{theo2}, and Theorem \ref{locally1} }
 	In this section, we provide  full proofs  for Theorem \ref{t1}, Theorem \ref{theo2}, and Theorem \ref{locally1}. In Subsection \ref{prooft1}, we prove  Theorem \ref{t1} by using some new techniques of the Picard approximation method. In Subsection \ref{prooft2}, we show Theorem \ref{theo2} by applying Banach fixed point theorem. And we end the section by proving Theorem \ref{locally1} in Subsection \ref{prooft3}.  For the sake of convenience, we will list some important constants, which cannot find in the proofs,	
in the part \textbf{(AP.4)} of the Appendix.  

\subsection{Proof of Theorems \ref{t1}} \label{prooft1}
Let us begin with the proof of  Theorem \ref{t1} by   using Picard's approximation method.  We   construct a Picard sequence defined by Lemma \ref{GSLVDanh}.  With some appropriate assumptions, we will bound the sequence  by a power function.  Then, we can prove it is a Cauchy sequence in a Banach space as Lemma \ref{GSLVDanhb}.  Now, we consider two following lemmas. 
\begin{lemma} \label{GSLVDanh}
	Let the Picard  sequence $\{w^{(k)}\}_{k=1,2,...}$ be   defined by 
	$\displaystyle w^{(1)}(t) = f, $ and
	\begin{align} \label{picardse}
	w^{(k+1)}(t)=   {\bf B}_{\alpha }(t,  T) f &+ \int_0^t  {\bf P}_{\alpha }  (t-r) G(r,w^{(k)}(r)) dr \nn\\
	&- \int_0^{ T}  {\bf B}_{\alpha }(t,  T)  {\bf P}_{\alpha }  (T-r) G(r,w^{(k)}(r)) dr, \quad 0\le t\le T. 
	\end{align}
	Then,  for all  $t>0$, $k\in\mathbb{N}$, $k\ge 1$,  it holds 
	\begin{align} \label{result}
	\norm{w^{(k)}(t)}_{\mathbb H^\nu({\Omega})} \le \mathcal N_{1}  t^{-\alpha (1-{ \theta }  )}  \|f \|_{   \mathbb H^{\theta+\nu}(\Omega) }  .
	\end{align} 
	where $\mathcal N_{1}  $ is given by \textbf{(AP.4.)} in the Appendix.
\end{lemma}

\begin{lemma} \label{GSLVDanhb} Let $\left\{u^{(k)}\right\}_{k=1,2,...}$ be
	the sequence defined by Lemma \ref{GSLVDanh}, then it 	
	is a bounded and Cauchy sequence in the Banach space  $ L^p (0, T ;   \mathbb H^\nu({\Omega}))$ with $p\in \left[1;\frac{1}{\alpha(1-\theta)}\right)$.
\end{lemma}

\begin{proof}[Proof of Lemma \ref{GSLVDanh}] 
	Let us consider the case $k=1$. Firstly,    $   \norm{f}_{ \mathbb H^\nu ({\Omega})}$ is   bounded by  $C_1(\nu, \theta)     \|f \|_{   \mathbb H^{\nu+\theta}(\Omega) }  $ upon the embedding (\ref{Embed1}). Furthermore, 
	it is obvious to see from \textbf{(AP.4.)} in the Appendix that $\mathcal N_{1}\ge C_1(\nu, \theta)  t^{\alpha (1-{ \theta }  )} $  by noting the number $ \al(1-\theta)$ is contained in the interval $(0;1)$. These easily imply   the desired inequality \eqref{result} for $k=1$. 
	Assume that \eqref{result} holds for $k=n$. 
	This means that
	\begin{align} \label{inductive}
	\norm{w^{(n)}(t)}_{\mathbb H^\nu({\Omega})} \le \mathcal N_{1}   t^{-\alpha (1-{ \theta }  )}  \|f \|_{   \mathbb H^{\nu+\theta}(\Omega) } .
	\end{align} 
	We show that \eqref{result}  holds for $k=n+1$. 
Thanks to definition \eqref{operatorAT},   using the fact that $\{\varphi_j\}$ is an orthonormal basis of $L^2(\Omega)$, and then using  Lemma  \ref{nayvesom}, one  arrives at
	\begin{align}
	\Big\| {\bf B}_{\alpha }(t,  T) f\Big\|_{\mathbb H^\nu({\Omega})} 
	%= \sqrt{\sum_{j=1}^\infty \lambda_j^{2 \nu} \Big|\mathscr E_{\al,T} (-\lambda_j t^\al) \Big|^2  \langle f,\varphi_j\rangle^2 }\nn\\
	&\le M_\al m_\al^{-1}  T^{\alpha(1-\theta)} \big( T^{\al \theta }+\lambda_1^{-\theta} \big) t^{- \alpha (1-  \theta )}  \Bigg( \sum_{j=1}^\infty \lambda_j^{2 \nu }  \langle f,\varphi_j\rangle^2 \Bigg)^{1/2} \nn\\
	&=  M_\al m_\al^{-1}  T^{\alpha(1-\theta)} \Big( T^{\al \theta }+\lambda_1^{-\theta} \Big)  t^{-\alpha (1-{ \theta }  )}  \norm{f}_{ \mathbb H^\nu ({\Omega})}\nn\\
	&  \le  M_\al m_\al^{-1}  T^{\alpha(1-\theta)} \Big( T^{\al \theta }+\lambda_1^{-\theta} \Big)  t^{-\alpha (1-{ \theta }  )} C_1(\nu, \theta) \norm{f}_{ \mathbb H^{\nu+\theta} ({\Omega})} , \nn  
	\end{align}
	where we have used 
	the Sobolev embedding $ \mathbb H^{\nu+\theta} ({\Omega})  \hookrightarrow  \mathbb H^{\nu} ({\Omega})  $. Next, let us estimate the integrals by using   assumption   $(\mathscr H_1)$. The idea is to  try to bound them by the convergent improper integrals. Indeed, one can show that  
	\begin{align} \label{important2}
	& \Bigg\|  \int_0^t  {\bf P}_{\alpha }  (t-r) G(r,w^{(n)}(r)) dr \Bigg\|_{\mathbb H^\nu({\Omega})}  \nn\\
	& \hspace*{3cm} \le  \int_0^t  \Bigg\|{ \sum_{j=1}^\infty  (t-r)^{\alpha-1}  E_{\alpha,\alpha}(-\lambda_j(t-r)^\alpha) G_j(r,w^{(n)}(r)) \varphi_j } \Bigg\|_{\mathbb H^\nu({\Omega})} dr \nn\\
	& \hspace*{3cm} \le M_\al  \lambda_1^{-\theta} \int_0^t (t-r)^{\alpha(1-\theta)-1}    \norm{G(r,w^{(n)}(r)) }_{\mathbb H^\nu({\Omega})} dr  \nn\\
	& \hspace*{3cm} \le  \| L_1\|_{L^\infty (0,T)}  M_\al  \lambda_1^{-\theta}  \int_0^t (t-r)^{\alpha(1-\theta)-1}    \norm{w^{(n)}(r) }_{\mathbb H^\nu({\Omega})} dr . 
	\end{align} 
	On the other hand,  the inequality
	$
	\mathscr E_{\al, T} (-\lambda_j t^\al) (T-r)^{\alpha-1} E_{\alpha,\alpha}(-\lambda_j(T-r)^\alpha) $ is obviously bounded by $\mathcal M_1 (T-r)^{\al(1-\theta)-1}   t^{-\al(1-\theta)}  
	$  due to Lemma \ref{nayvesom}. Here, the constant $\mathcal M_1$ is given by \textbf{(AP.4)} in the Appendix. 
	We then obtain the following estimate
	\begin{align} \label{i4}
	&\Bigg\| \int_0^T {\bf B}_{\alpha }(t,  T)  {\bf P}_{\alpha }  (T-r) G(r,w^{(n)}(r)) dr  \Bigg\|_{\mathbb H^\nu({\Omega})}   \nn\\
	&\quad \quad \quad \quad \le \int_0^T  \Bigg\|{ \sum_{j=1}^\infty  (T-r)^{\alpha-1} \mathscr E_{\al, T} (-\lambda_j t^\al) E_{\alpha,\alpha}(-\lambda_j(T-r)^\alpha) G_j(r,w^{(n)}(r)) \varphi_j }\Bigg\|_{\mathbb H^\nu({\Omega})} dr\nn\\
	&\quad \quad \quad \quad \le  \| L_1\|_{L^\infty (0,T)}  \mathcal M_1 t^{-\al(1-\theta)} \int_0^T (T-r)^{\alpha(1-\theta)-1}    \norm{w^{(n)}(r) }_{\mathbb H^\nu({\Omega})} dr , 
	\end{align} 
	where we note that  the constant $\mathcal M_1$ is given by \textbf{(AP.4.)} in the Appendix.
	According to the above inequalities, we need to estimate the integral $\int_0^t (t-r)^{\alpha(1-\theta)-1}    \norm{w^{(n)}(r) }_{\mathbb H^\nu({\Omega})} dr$. To do this, we will apply the   inductive hypothesis \eqref{inductive}. Moreover, by also using the fact that $1\le T^{\alpha(1-\theta)}t^{-\alpha(1-\theta)}$ for all $0<t\le T$, and $\int_0^t(t-r)^{\alpha(1-\theta)-1}r^{-\alpha(1-\theta)}dr$ is equal to $\pi/\sin (\pi\alpha(1-\theta) )$, we consequently obtain the following estimates  
	\begin{align} \label{a1}
	\int_0^t (t-r)^{\alpha(1-\theta)-1}    \norm{w^{(n)}(r) }_{\mathbb H^\nu({\Omega})} dr &   \le      \mathcal N_{1} \|f \|_{   \mathbb H^{\nu+\theta}(\Omega) } \int_0^t (t-r)^{\alpha(1-\theta)-1}    r^{-\alpha (1-{ \theta}  )}  dr     \nn\\
	& \le   \mathcal N_{1}	 \|f \|_{   \mathbb H^{\nu+\theta}(\Omega) } \frac{\pi T^{\al(1-\theta) } }{\sin (\pi\alpha(1-\theta) )} t^{-\al(1-\theta)}.
	\end{align}
	Here, in the last equality, we use \textbf{(AP.1.)} in the Appendix. 
	By  similar arguments as above, we obtain 
	\begin{align} \label{ab1}
	\Bigg\|  \int_0^t  {\bf P}_{\alpha }  (t-r) G(r,w^{(n)}(r)) dr + \int_0^T   {\bf B}_{\alpha }(t,  T)   {\bf  P}_{\alpha }  (T-r) G(r,w^{(n)}(r)) dr \Bigg\|_{\mathbb H^\nu({\Omega})}  & \nn\\
	\le \| L_1\|_{L^\infty (0,T)} \mathscr M_1 \mathcal N_1 & . 
	\end{align}
	From some preceding estimates and by some simple computations, we can find that
	\begin{align*}
	\norm{w^{(n+1)}(t)}_{\mathbb H^\nu({\Omega})}  &\le \Big\|   { \bf B}_{\alpha }(t,  T) f\Big\|_{\mathbb H^\nu({\Omega})}+\Big\|  \int_0^t  { \bf P}_{\alpha }  (t-r) G(r,w^{(n)}(r)) dr \Big\|_{\mathbb H^\nu({\Omega})}\nn\\
	& \hspace*{3.42cm} +\Bigg\|  \int_0^T   {\bf  B}_{\alpha }(t,  T)   {\bf  P}_{\alpha }  (T-r) G(r,w^{(n)}(r)) dr \Bigg\|_{\mathbb H^\nu({\Omega})}\nn\\
	& \le \Big( \mathcal M_1 M_\al^{-1}   C_1(\nu, \theta)    +  \|L_1\|_{L^\infty(0,T)} \mathscr M_1 \mathcal N_1 \Big) t^{-\alpha (1-{ \theta }  )}  \|f \|_{   \mathbb H^{\nu+\theta}(\Omega) } \nn\\
	&= \mathcal N_1 t^{-\alpha (1-{ \theta }  )}  \|f \|_{   \mathbb H^{\nu+\theta}(\Omega) }.
	\end{align*}
	By the induction method, we deduce that \eqref{inductive} holds for any $k=n\in \mathbb{N}$, $n\ge 1$.
\end{proof}

\begin{proof}[Proof of Lemma \ref{GSLVDanhb}] Since $ p \in [1, \frac{1}{\alpha(1-\theta)})$, we know that the function $t\mapsto t^{-\alpha (1-\theta)}$ is $ L^p(0,T)$-integrable which implies that $\big\{u^{(n)}\big\}$ is a bounded sequence in $ L^p(0, T ;   \mathbb H^\nu({\Omega}))$. Hence, it is necessary to prove $\big\{u^{(n)}\big\}$ is a Cauchy sequence. By using the notation $\mathbf{u}^{(n,k)}:= u^{(n+k)} -u^{(n)}$ and using triangle inequality, one has the following estimate
	\begin{align*}   
	\Big\| \mathbf{u}^{(n+1,k)}(t)  \Big\|_{\mathbb H^\nu({\Omega})}    
	%&\le \Bigg\|  \int_0^t   { \bf P}_{\alpha }  (t-r) \Big( G(r,u^{(l+k)}(r)) - G(r,u^{(l)}(r)) \Big) dr \Bigg\|_{\mathbb H^\nu({\Omega})} \nn\\
	%&  +\Bigg\|  \int_0^T  { \bf B}_{\alpha }(t,  T)   \mathscr { P}_{\alpha }  (T-r) \Big( G(r,u^{(l+k)}(r)) - G(r,u^{(l)}(r)) \Big) dr \Bigg\|_{\mathbb H^\nu({\Omega})}   \nn\\
	&  \le  \| L_1\|_{L^\infty (0,T)}   M_\al  \lambda_1^{-\theta}   \int_0^t (t-r)^{\alpha(1-\theta)-1}    \big\| \mathbf{u}^{(n,k)}(r)  \big\|_{\mathbb H^\nu({\Omega})}  dr \nn\\
	& +  \| L_1\|_{L^\infty (0,T)}  \mathcal M_1    t^{-\al(1-\theta)}    \int_0^T (T-r)^{\alpha(1-\theta)-1}    \big\| \mathbf{u}^{(n,k)}(r)  \big\|_{\mathbb H^\nu({\Omega})}  dr.  
	\end{align*}	
	Similarly to the proof of Lemma \ref{GSLVDanh}, one can use the inductive hypothesis to estimate the above right hand side.  Then by iterating the same computations in Lemma \ref{GSLVDanh}, we can bound   $\| \mathbf{u}^{(n+1,k)}(.,t) \|_{\mathbb H^\nu({\Omega})} $ by the quantity $2 \mathcal N_{1}  (  \| L_1\|_{L^\infty (0,T)}  	\mathscr M_1    )^{n} t^{-\al(1-\theta)}   $. Summarily, one can obtain the following conclusion by basing on the
	induction method  
	\begin{equation}  \label{qt10}
	\big\|\mathbf{u}^{(n,k)}(t)  \big\|_{\mathbb H^\nu({\Omega})} \le 2 \mathcal N_{1}     \Big(   \| L_1\|_{L^\infty (0,T)} 	\mathscr M_1    \Big)^{n-1}  t^{-\alpha (1-{ \theta }  )} ,
	\end{equation}
	which completed the proof by letting $n\to \infty$. 
\end{proof}

\begin{proof}[Proof of Theorem \ref{t1}] We firstly prove the existence of a mild solution $u$ in the space $L^p (0, T ;   \mathbb H^\nu({\Omega}))$, and then obtain its continuity in the following steps (1 and 2). After that, we will present the proofs the Parts a - d in the sequel. 
	
	\vspace*{0.2cm}
	
	\noindent \textbf{Step 1: Prove the existence of a mild solution $u$ in the space $L^p (0, T ;   \mathbb H^\nu({\Omega}))$:} 
	Since
	$ L^p (0, T ;   \mathbb H^\nu({\Omega})) $ is a Banach space and
	$\left\{u^{(n)}\right\}_{n=1,2,...}$ is a Cauchy sequence in $ L^p (0, T ;   \mathbb H^\nu({\Omega}))$ thanks to Lemmas (\ref{GSLVDanh}), (\ref{GSLVDanhb}), we deduce that there exists a function $  u \in
	L^p (0, T ;   \mathbb H^\nu({\Omega})) $ such that $\lim_{n \to +\infty } u^{(n)}=
	u$. 
	Now, we show that $u$ is a mild  solution of the problem  by showing  that $u= \overline {\bf  J} u$, where
	\begin{align} \label{Lmildsolution1}
	\overline {\bf  J} u(t):=    {\bf  B}_{\alpha }(t,  T) f &+ \int_0^t   {\bf  P}_{\alpha }  (t-r) G(r,u(r)) dr  - \int_0^{ T}   {\bf  B}_{\alpha }(t,  T)   { \bf P}_{\alpha }  (T-r) G(r,u(r)) dr. 
	\end{align}
	Since the sequence $\big\{ u^{(n)} \big\}$ converges to $u$ in $ L^p (0, T ;   \mathbb H^\nu({\Omega}))$-norm, there exists a sub-sequence $\big\{ u^{(n_m)}\big\}$  such that   point-wise converges to $u$, i.e.,  
	$ u^{(n_m)}(t) \longrightarrow u(t) $ in $\mathbb H^\nu({\Omega})$-norm for almost everywhere $t$ in $(0,T)$. Let us denote by $\mathbf v^{(n_m)}:= u^{(n_m)}   - u$.
	This fact and taking $k\to \infty$ in the estimate (\ref{qt10}) allow us to obtain 
	\begin{equation}   
	\big\| \mathbf v^{(n_m)}  \big\|_{\mathbb H^\nu({\Omega})} \le 2 \mathcal N_{1}   \Big(   \| L_1\|_{L^\infty (0,T)}  	\mathscr M_1  \Big)^{n_m-1}  t^{-\alpha (1-{ \theta }  )}    , 
	\label{PointwiseCon}
	\end{equation}
	for almost everywhere $t$ in $(0,T)$ also. 	Moreover, we note from Lemma 3.3  that this sub-sequence is also  bounded by the power function $t\mapsto  t^{-\alpha (1-{ \theta }  )} $.    
	These help us to apply dominated convergence theorem as follows. Indeed,  it is obvious that the quantities 
	\begin{align}
	\int_0^t (t-r)^{\alpha(1-\theta)-1} \big\| \mathbf  v^{(n_m)} \big\|_{\mathbb H^\nu({\Omega})} dr, \quad 
	t^{-\alpha(1-\theta)}  \int_0^T (T-r)^{\alpha(1-\theta)-1} \big\| \mathbf v^{(n_m)} \big\|_{\mathbb H^\nu({\Omega})} dr,   \nn
	\end{align}
	point-wise converge  to zero as (\ref{PointwiseCon}), and are bounded by    $ L^p(0,T)$-integrable functions. Thus, the same computations as (\ref{important2}), show 
	\begin{align} \label{00}
	&\Bigg\|  \int_0^t   {\bf  P}_{\alpha }  (t-r) \Big( G(r,u^{(n_m)}(r)) - G(r,u(r)) \Big) dr  \Bigg\|_{ L^p (0,T; \mathbb H^\nu({\Omega}) )}^p \nn\\
	&\hspace*{2cm} = \int_0^T \left\| \int_0^t   {\bf  P}_{\alpha }  (t-r) \Big( G(r,u^{(n_m)}(r)) - G(r,u(r)) \Big) dr  \right\|_{\mathbb H^\nu({\Omega})}^p dt \nn\\
	&\hspace*{2cm} \le       \int_0^T  \Bigg\{ \int_0^t (t-r)^{\alpha(1-\theta)-1} \big\| \mathbf v^{(n_m)}  \big\|_{\mathbb H^\nu({\Omega})} dr \Bigg\}^p  dt,
	\end{align}
	and by using a similar way as in \eqref{i4}, we have the following bound
	\begin{align} \label{01}
	&  \left\| \int_0^T   {\bf  B}_{\alpha }(t,  T)  {\bf P}_{\alpha }  (T-r) \Big( G(r,u^{(n_m)}(r)) - G(r,u(r)) \Big) dr   \right\|_{ L^p (0,T; \mathbb H^\nu({\Omega})  )}^p \nn\\
	&\hspace*{2cm} = \int_0^T \left\| \int_0^T    {\bf  B}_{\alpha }(t,  T)   {\bf  P}_{\alpha }  (T-r) \Big( G(r,u^{(n_m)}(r)) - G(r,u(r)) \Big) dr   \right\|_{\mathbb H^\nu({\Omega})}^p dt \nn\\
	&\hspace*{2cm}  \le    \int_0^T  \Bigg\{ t^{-\alpha(1-\theta)}  \int_0^T (T-r)^{\alpha(1-\theta)-1} \big\|  \mathbf v^{(n_m)} \big\|_{\mathbb H^\nu({\Omega})} dr \Bigg\}^p  dt.
	\end{align}
	The right hand-side of \eqref{00} and \eqref{01} tend  to zero when $n_m$ goes to positive infinity. The above arguments conclude that $u$ satisfies Equation (\ref{Lmildsolution1}), and so that is a mild solution of Problem (\ref{mainprob1})-(\ref{mainprob4}) in $ L^p (0, T ;   \mathbb H^\nu({\Omega}))$. 
	By taking the limit of the left handside of  \eqref{01}, we obtain  
	\begin{align} \label{e7}
	\|u\|_{ L^p (0, T ;   \mathbb H^\nu({\Omega}))}  \lesssim   \|f \|_{   \mathbb H^{\nu+\theta}(\Omega) }.
	\end{align}	
	
	\noindent \textbf{Step 2: Prove   $u\in  C^{\alpha(1-\theta)} \big((0,T];\mathbb H^{\nu} (\Omega)\big)$:} 	  We need to estimate   $u(\widetilde t )-u(t)$ in $\mathbb H^{\nu} $ norm, for all $0<t \le \widetilde t \le T$. By the formulation \eqref{Lmildsolution}, the triangle inequality yields that  
	\begin{align} \label{b1}
	\Big\|   u( \widetilde t)-u(t)  \Big\|_{\mathbb H^{\nu} (\Omega)}
	&\le \Big\| \Big(  {\bf  B}_{\alpha }(\widetilde t,  T)-  {\bf  B}_{\alpha }(t,  T)  \Big) f \Big\|_{\mathbb H^{\nu} (\Omega)} \nn\\
	& \hspace*{1.5cm} +\Bigg\| \int_0^t \Big(  {\bf  P}_{\alpha }  (\widetilde t-r) -  {\bf  P}_{\alpha }  (t-r) \Big) G(r,u(r)) dr \Bigg\|_{\mathbb H^{\nu} (\Omega)}\nn\\
	&+\Bigg\| \int_t^{\widetilde t} { \bf P}_{\alpha }  (\widetilde t-r)    G(r,u(r)) dr \Bigg\|_{\mathbb H^{\nu} (\Omega)}\nn\\
	&+\Bigg\| \int_0^T \Big(  { \bf B}_{\alpha }(\widetilde t,  T)-  { \bf B}_{\alpha }(t,  T)  \Big)   {\bf  P}_{\alpha }  (T-r)    G(r,u(r)) dr \Bigg\|_{\mathbb H^{\nu} (\Omega)}\nn\\
	&:=\|\mathfrak{M}_1\|_{\mathbb H^{\nu} (\Omega)}+\| \mathfrak{M}_2\|_{\mathbb H^{\nu} (\Omega)}+\|\mathfrak{M}_3  \|_{\mathbb H^{\nu} (\Omega)}+\| \mathfrak{M}_4 \|_{\mathbb H^{\nu} (\Omega)}.
	\end{align}
	In what follows, we will estimate the terms $\mathfrak{M}_j$ for $1\le j\le 4$. \\
	{\it \underline{Estimate of $\mathfrak{M}_1$}: } 
	Using the fact that  $\displaystyle \partial_t   E_{\alpha ,1}(-\lambda_j  t^\alpha )  = -\lambda_j  t^{\alpha -1}  E_{\alpha ,\alpha  }(-\lambda_j  t^\alpha )$, we find that
	$$
	\frac{E_{\alpha ,1}(-\lambda_j \widetilde t^\alpha )-E_{\alpha ,1}(-\lambda_j  t^\alpha )}{E_{\alpha ,1}(-\lambda_j T^\alpha )} 
	=    \int_{t}^{\widetilde t} -\lambda_j r^{\alpha -1}  \frac{ E_{\alpha ,\alpha  }(-\lambda_j r^\alpha )}{E_{\alpha ,1}(-\lambda_j  T^\alpha )} dr .
	$$
	The latter inequality leads to 
	\begin{align}
	\Big\|  \Big( {\bf  B}_{\alpha }(\widetilde t,  T)-   {\bf  B}_{\alpha }(t,  T)  \Big) f \Big\|_{ \mathbb H^{\nu} (\Omega)} 
	& \lesssim  \int_t^{\widetilde t} \Bigg\| \sum_{j \in {\mathbb N}} \lambda_j r^{\alpha -1}  \frac{ E_{\alpha ,\alpha  }(-\lambda_j r^\alpha )}{E_{\alpha ,1}(-\lambda_j  T^\alpha )}   f_j  \varphi_j \Bigg\|_{ \mathbb H^{\nu} (\Omega)} dr \nn\\
	&\lesssim   \left( \int_t^{\widetilde t}    r^{ \alpha(\theta-1) -1}  dr \right) \|f \|_{   \mathbb H^{\nu+\theta}(\Omega) }\nn\\
	& = \frac{\widetilde t^{\alpha(1-\theta)}-t^{\alpha(1-\theta)}}{\alpha(1-\theta)t^{\alpha(1-\theta)} \widetilde t^{\alpha(1-\theta)} } \|f \|_{   \mathbb H^{\nu+\theta}(\Omega) } ,\label{lem2aaaa} 
	\end{align}
	where we have used the fact that 
	\begin{align} \label{lem2aa}
	\lambda_j r^{\alpha -1}\frac{ E_{\alpha ,\alpha }(-\lambda_j r^\alpha ) }{E_{\alpha ,1}(-\lambda_j T^\alpha )}  ~ \lesssim ~ \lambda_j r^{\alpha -1} \frac{1+\lambda_j T^\alpha }{1+ (\lambda_j r^\alpha)^2 } 
	~\lesssim~   \lambda_j^{\theta } r^{\alpha(\theta-1) -1}.
	\end{align}
	By noting that  $0<\alpha(1-\theta)<1$, we now have $ \widetilde t ^{\alpha(1-\theta)}-t^{\alpha(1-\theta)} \le (\widetilde t-t)^{\alpha(1-\theta)}$ and furthermore $\alpha(1-\theta) t^{\alpha(1-\theta)}  \widetilde t ^{\alpha(1-\theta)} \ge \alpha(1-\theta)t^{2\alpha(1-\theta)}.$
	Consequently, thanks to the estimate \eqref{lem2aaaa},  we arrive at
	\begin{align} \label{b11}
	\|\mathfrak{M}_1\|_{\mathbb H^{\nu} (\Omega)}  \lesssim   t^{-2\alpha(1-\theta)} (\widetilde t-t)^{\alpha(1-\theta)}  \|f \|_{   \mathbb H^{\nu+\theta}(\Omega) }.
	\end{align} 
	{\it \underline{Estimate of $\mathfrak{M}_2$}: } 
	It follows from   differentiating $ \partial_\rho \Big(\rho^{\alpha -1}E_{\alpha ,\alpha  }(-\lambda_j \rho^\alpha )\Big) = \rho^{\alpha -2}E_{\alpha ,\alpha -1}(-\lambda_j \rho^\alpha ),$ for all $\rho>0$ (see  Lemma \ref{DeriML}), that 
	\begin{align}
	& \int_0^t \Bigg( (\widetilde t-r)^{\alpha-1} E_{\alpha,\alpha}(-\lambda_j(\widetilde t-r)^\alpha) - (t-r)^{\alpha-1} E_{\alpha,\alpha}(-\lambda_j(t-r)^\alpha)\Bigg) G_j(r,u(r)) dr \nn\\
	&\quad \quad \quad =\int_0^t \int_{t-r}^{\widetilde t-r} \rho^{\alpha  -2}E_{\alpha  ,\alpha  -1}(-\lambda_j \rho^\alpha  ) G_j(r,u(r)) d\rho  dr  .\label{G1thb}
	\end{align}
	This implies that
	\begin{align}
	\|\mathfrak{M}_2\|_{\mathbb H^{\nu} (\Omega)}
	&\lesssim    \int_0^t \int_{t-r}^{\widetilde t-r} \rho ^{\alpha  -2}  \norm{G(r,u(r))}_{ \mathbb H^{\nu} (\Omega)} d\rho  dr \nn\\
	& \lesssim   (\widetilde t-t)^{\alpha-1}   \int_0^t   \norm{u(r)}_{ \mathbb H^{\nu} (\Omega)} dr \lesssim  (\widetilde t-t)^{\alpha-1}  \|u\|_{ L^p (0, T ;   \mathbb H^\nu({\Omega}))} \nn\\
	&	\lesssim   (\widetilde t-t)^{\alpha-1} \|f \|_{   \mathbb H^{\nu+\theta}(\Omega) } , \label{b111}
	\end{align}
	where we note  that
	\begin{align}
	\int_{t-r}^{\widetilde t-r} \rho^{\alpha  -2}   d\rho=\frac{(\widetilde t-r)^{\alpha-1}-(t-r)^{\alpha-1}}{\alpha-1} \le \frac{(\widetilde t-t)^{\alpha-1}}{\alpha-1}, \nn\\
	\rho^{\alpha  -2}E_{\alpha  ,\alpha  -1}(-\lambda_j \rho^\alpha  )    \lesssim     \rho^{\alpha  -2} \frac{1}{1+ \lambda_j \rho^\alpha}   \lesssim    \rho^{\alpha  -2} .\nn 
	\end{align}
	{\it \underline{Estimate of $\mathfrak{M}_3$}: } For all $t<r<\widetilde t$,  using the inequality $ |E_{\alpha,\alpha}(-\lambda_j (\widetilde t-r)^\alpha)  |\le   {M}_\al $ and  the fact that $(\widetilde t-r)^{\alpha-1} \le (\widetilde t-t)^{\alpha-1}$ show   
	\begin{align}
	\Big\| \mathfrak{M}_3 \Big\|_{\mathbb H^{\nu} (\Omega)} 
	&\le   {M}_\al  (\widetilde t-t)^{\alpha-1}  \int_t^{\widetilde t}  \norm{ G(r,u(r)) }_{\mathbb H^\nu({\Omega})}  dr  \lesssim  (\widetilde t-t)^{\alpha-1} \|f \|_{   \mathbb H^{\nu+\theta}(\Omega) } .  \label{b1111}
	\end{align}
	\noindent
	{\it \underline{Estimate of $\mathfrak{M}_4$}: } 
	Using  \eqref{lem2aa}, we obtain
	\begin{align*}  
	&\Bigg\|\sum_{j=1}^\infty  (T-r)^{\alpha-1}  \lambda_j \rho^{\alpha -1}  \frac{ E_{\alpha ,\alpha  }(-\lambda_j \rho^\alpha )}{E_{\alpha ,1}(-\lambda_j  T^\alpha )}  E_{\alpha,\alpha}(-\lambda_j(T-r)^\alpha) G_j(r,u(r)) \varphi_j\Bigg\|_{\mathbb H^{\nu} (\Omega)} & \nn\\
	&  \hspace*{0.6cm}  
	%= \sqrt{ \sum_{j=1}^\infty \lambda_j^{2 \nu} \Bigg| (T-r)^{\alpha-1}  \lambda_j \rho^{\alpha -1}  \frac{ E_{\alpha ,\alpha  }(-\lambda_j r^\alpha )}{E_{\alpha ,1}(-\lambda_j  T^\alpha )}  E_{\alpha,\alpha}(-\lambda_j(T-r)^\alpha) G_j(u(r))  \Bigg|^2   }\nn\\
	\quad \quad \quad \quad \quad \lesssim  \mathcal  (T-r)^{\al(1-\theta)-1} \rho^{\al(1-\theta)-1} \norm{G(r,u(r))}_{\mathbb H^{\nu} (\Omega)}. &
	\end{align*}
	
	\vspace*{0.15cm}
	
	\noindent By using the fact that $\widetilde t^{\alpha(1-\theta)}-t^{\alpha(1-\theta)} \le (\widetilde t-t)^{\alpha(1-\theta)}$, and 
	$t^{\alpha(1-\theta)}  \widetilde t ^{\alpha(1-\theta)} \ge t^{2\alpha(1-\theta)},$ we derive that
	\begin{align} \label{aa1}
	\|\mathfrak{M}_4	\|_{\mathbb H^{\nu} (\Omega)}  
	&\lesssim  \int_0^T \int_{t}^{\widetilde t} (T-r)^{ \alpha(1-\theta)-1} \xi^{ \alpha(\theta-1)-1} \norm{G(r,u(r))}_{\mathbb H^{\nu} (\Omega)}   d\xi  dr    \nn \\
	&\lesssim  \frac{ \widetilde t ^{\alpha(1-\theta)}-t^{\alpha(1-\theta)} }{\alpha(1-\theta)t^{\alpha(1-\theta)} \widetilde t ^{\alpha(1-\theta)} } \int_0^T  (T-r)^{ \alpha(1-\theta)-1}   \norm{G(r,u(r))}_{\mathbb H^{\nu} (\Omega)}   dr    \nn\\
	&  \lesssim
	t^{-2\alpha(1-\theta)} (\widetilde t-t)^{\alpha(1-\theta)} \int_0^T  (T-r)^{ \alpha(1-\theta)-1}   \norm{G(r,u(r))}_{\mathbb H^{\nu} (\Omega)}   dr    .       
	\end{align} 
	It is easy to see that
	\begin{align}
	\int_0^T  (T-r)^{ \alpha(1-\theta)-1}   \norm{G(u(r))}_{\mathbb H^{\nu} (\Omega)} &  \lesssim  
	\int_0^T  (T-r)^{ \alpha(1-\theta)-1}   \norm{u(r)}_{\mathbb H^{\nu} (\Omega)}   dr\nn\\
	& \lesssim  \|f \|_{   \mathbb H^{\nu+\theta}(\Omega) }  \int_0^T  (T-r)^{ \alpha(1-\theta)-1}  r^{-\alpha (1-{ \theta }  )}  dr        \nn\\
	&\lesssim \|f \|_{   \mathbb H^{\nu+\theta}(\Omega) }. \nn
	\end{align}
	The latter estimate together with \eqref{aa1} lead  to
	\begin{equation} \label{b11111}
	\|	\mathfrak{M}_4 \|_{\mathbb H^{\nu} (\Omega)} \lesssim
	t^{-2\alpha(1-\theta)} (\widetilde t-t)^{\alpha(1-\theta)} \|f \|_{   \mathbb H^{\nu+\theta}(\Omega) }.
	\end{equation}
	
	\noindent \underline{\textit{Obtaining the estimate for $u( \widetilde t)-u(t)$:}} Combining \eqref{b1}, \eqref{b11}, \eqref{b111}, \eqref{b1111}, \eqref{b11111} lead us to
	$$
	\|   u(\widetilde t)-u(t)  \|_{\mathbb H^{\nu} (\Omega)}  \lesssim  \Big[ t^{-2\alpha(1-\theta)} (\widetilde t-t)^{\alpha(1-\theta)}+ (\widetilde t)^{\alpha-1}  \Big]     \|f \|_{   \mathbb H^{\nu+\theta}(\Omega) },
	$$
	which  implies that  $u\in  C^{\alpha(1-\theta)} \big((0,T];\mathbb H^{\nu} (\Omega)\big)$, and  the inequality (\ref{vanphong}) can be obtained by using (\ref{e7}) . \\
	
	\noindent	In what follows, we carry out the proofs of Part (a), (b), (c), (d).\\
	
	\noindent \underline{{\it Part (a)}.}  {\it  Prove $ u\in \displaystyle    L^p (0,    T ;  \mathbb H^{\nu+\theta-\theta'} (\Omega)) $, for any  $p \in \left[1, \frac{1}{\alpha(1-\theta')}\right)$.}\\
	
	\noindent    Lemma \ref{nayvesom} yields the estimate
	\begin{align}
	\Big\|   {\bf  B}_{\alpha }(t,  T) f\Big\|_{\mathbb H^{\nu+\theta-\theta'}({\Omega})}^2&= \sum_{j=1}^\infty \lambda_j^{2 \nu +2 \theta - 2\theta'} \Big|\mathscr E_{\al,T} (-\lambda_j t^\al) \Big|^2  \langle f,\varphi_j\rangle^2 \nn\\
	&\lesssim    t^{-2\alpha (1-{ \theta' }  )} \sum_{j=1}^\infty \lambda_j^{2 \nu +2 \theta - 2\theta'} \lambda_j^{2 \theta' }  \langle f,\varphi_j\rangle^2. \label{chieuroi} 
	\end{align}
	Therefore, $ {\bf  B}_{\alpha }(t,  T) f \in \displaystyle     L^p (0,    T ;  \mathbb H^{\nu+\theta-\theta'} (\Omega))$. Further, we infer from assumption $(\mathscr H_1)$  and the estimate  $(t-r)^{\alpha-1}  E_{\alpha,\alpha}(-\lambda_j(t-r)^\alpha) \lesssim \lambda_j^{-(\theta-\theta')} (t-r)^{\alpha(1-(\theta-\theta'))-1}$ as in  Lemma \ref{MLineqna}  
	\begin{align} \label{chieuroib}
	\Bigg\|  \int_0^t   { \bf P}_{\alpha }  (t-r) G(r,u(r)) dr \Bigg\|_{\mathbb H^{\nu+\theta-\theta'}({\Omega})}    &\lesssim    \int_0^t (t-r)^{\alpha(1-\theta+\theta')-1}    \norm{G(r,u(r)) }_{\mathbb H^\nu({\Omega})} dr  \nn\\
	&\lesssim   \int_0^t (t-r)^{\alpha(1-\theta+\theta')-1} r^{-\alpha(1-\theta)} dr \lesssim   t^{-\alpha(1-\theta')} , 
	\end{align} 
	where we used the definition of the Beta function as  \textbf{(AP.1.)} in the Appendix, and the fact that  $t^{\alpha \theta' } = t^{-\alpha(1-\theta')}t^{\alpha } \lesssim t^{-\alpha(1-\theta')}$. Now, we proceed to estimate the last term of $u$.
	Lemma 3.2 yields that 
	$$ \mathscr E_{\al, T} (-\lambda_j t^\al) (T-r)^{\alpha-1} E_{\alpha,\alpha}(-\lambda_j(T-r)^\alpha) \lesssim \lambda_j^{\theta'-\theta} t^{-\al(1-\theta')} (T-r)^{\al(1-\theta)-1}    .
	$$ 
	This  invokes from assumption $(\mathscr H_1)$ that 
	\begin{align} \label{chieuroic}
	&\Bigg\| \int_0^T   { \bf B}_{\alpha }(t,  T)  {\bf  P}_{\alpha }  (T-r) G(r,u(r)) dr  \Bigg\|_{\mathbb H^{\nu+\theta-\theta'}({\Omega})}   \nn\\
	&  \quad \quad \quad  \quad \quad \quad \quad\lesssim t^{-\al(1-\theta')} \int_0^T (T-r)^{\alpha(1-\theta)-1}    \norm{G(r,u(r)) }_{\mathbb H^\nu({\Omega})} dr \nn\\
	&  \quad \quad \quad \quad \quad \quad \quad\lesssim t^{-\al(1-\theta')} \int_0^T (T-r)^{\alpha(1-\theta)-1}    r^{-\alpha(1-\theta)} dr  \lesssim t^{-\al(1-\theta')}.  
	\end{align}  
	Taking the above estimates (\ref{chieuroi}), (\ref{chieuroib}), (\ref{chieuroic}) together, we imply that $ u\in \displaystyle     L^p (0,    T ;  \mathbb H^{\nu+\theta-\theta'} (\Omega)) $, for all $p \in \left[1, \frac{1}{\alpha(1-\theta')}\right)$, and complete this part. 
	
	\vspace*{0.2cm}

	\noindent \underline{{\it Part (b)}.} {\it Show  that  $ u \in  C\left([0,T];\mathbb H^{\nu-\nu'} (\Omega)\right) $.}\\
	
	\noindent 	Let $t$ and $\widetilde t$ be satisfied that $0\le t< \widetilde t\le T$.   	
	By using the equation \eqref{b1} and the embedding $ \mathbb H^{\nu} (\Omega)  \hookrightarrow \mathbb H^{\nu- \nu'} (\Omega) $, we obtain
	\begin{align} \label{c11}
	&\|   u(\widetilde t)-u(t)  \|_{\mathbb H^{\nu- \nu'} (\Omega) } \nn\\
	&\quad \quad \quad \le 
	\|\mathfrak{M}_1\|_{\mathbb H^{\nu- \nu'} (\Omega) }+\|\mathfrak{M}_2\|_{\mathbb H^{\nu- \nu'} (\Omega) }+\|\mathfrak{M}_3\|_{\mathbb H^{\nu- \nu'} (\Omega) }+\|\mathfrak{M}_4\|_{\mathbb H^{\nu- \nu'} (\Omega) }\nn\\
	&\quad \quad \quad\lesssim \|\mathfrak{M}_1\|_{\mathbb H^{\nu- \nu'} (\Omega) }+|\mathfrak{M}_4\|_{\mathbb H^{\nu- \nu'} (\Omega) }+ \|\mathfrak{M}_2\|_{\mathbb H^{\nu} (\Omega) }+\|\mathfrak{M}_3\|_{\mathbb H^{\nu} (\Omega) }.
	\end{align}
	For the right-hand side of \eqref{b1}, we thus need to  estimate the terms $ \|\mathfrak{M}_1\|_{\mathbb H^{\nu- \nu'} (\Omega) }$, $\|\mathfrak{M}_4\|_{\mathbb H^{\nu- \nu'} (\Omega) }$.  We now continue to consider the  following estimates.
	
	\vspace*{0.2cm}
	\noindent {\it Estimate $\|\mathfrak{M}_1\|_{\mathbb H^{\nu- \nu'} (\Omega) }$}: 
	It is easy to show that
	\begin{align} \label{c111}
	\|\mathfrak{M}_1\|_{\mathbb H^{\nu- \nu'} (\Omega) }  \lesssim     \frac{ (\widetilde t-t)^{\al(\theta+\nu'-1)} }{\al(\theta+\nu'-1)} \|f \|_{   \mathbb H^{\nu+\theta}(\Omega) }  . \nn
	\end{align}
	\noindent
	{\it Estimate $\|\mathfrak{M}_4\|_{\mathbb H^{\nu- \nu'} (\Omega)}$}: 
	By a similar argument as in \eqref{b11111}, we obtain
	\begin{align}
	\|\mathfrak{M}_4\|_{\mathbb H^{\nu- \nu'} (\Omega)}  
	%&\le   {  \int_0^T \int_{t}^{t+\gamma} \Bigg\|\sum_{j=1}^\infty  (T-r)^{\alpha-1}  \lambda_j \rho^{\alpha -1}  \frac{ E_{\alpha ,\alpha  }(-\lambda_j \xi^\alpha )}{E_{\alpha ,1}(-\lambda_j  T^\alpha )}  E_{\alpha,\alpha}(-\lambda_j(T-r)^\alpha) G_j(u(r)) \varphi_j\Bigg\|_{\mathbb H^{\nu-\nu'} (\Omega)} d\rho dr }\nn\\ 
	&\lesssim    {  \int_0^T \int_{t}^{\widetilde t} (T-r)^{\al(1-\theta)-1} \rho^{\al(\theta+\nu'-1)-1} \norm{G(r,u(r))}_{\mathbb H^{\nu} (\Omega)} d\rho dr. }
	\end{align}
	Since $0< \al(\theta+\nu'-1)<2$, we split it into the two following cases:\\
	{\it Case 1}. If $0< \al(\theta+\nu'-1) \le 1$ then apply $(a+b)^\sigma \le a^\sigma +b ^\sigma ,~a, b \ge 0, 0<\sigma<1$, we have
	\begin{equation}
	(\widetilde t)^{\al(\theta+\nu'-1)}-t^{\al(\theta+\nu'-1)} \le (\widetilde t-t)^{\al(\theta+\nu'-1)}.
	\end{equation}
	From two latter observations, we find that
	\begin{align} \label{c11111}
	\|\mathfrak{M}_4\|_{\mathbb H^{\nu- \nu'} (\Omega)} 
	&\lesssim   \int_0^T \int_{t}^{\widetilde t} (T-r)^{\al(1-\theta)-1} \rho^{\al(\theta+\nu'-1)-1}  \norm{G(r,u(r))}_{\mathbb H^{\nu} (\Omega)}  d\rho  dr    \nn \\
	&\lesssim \frac{ (\widetilde t)^{\al(\theta+\nu'-1)} - t^{\al(\theta+\nu'-1)}}{\al(\theta+\nu'-1)}   \int_0^T  (T-r)^{ \alpha(1-\theta)-1}   \norm{G(r,u(r))}_{\mathbb H^{\nu} (\Omega)}   dr    \nn\\
	& \lesssim  
	h^{\alpha(\theta+\nu'-1)} \int_0^T  (T-r)^{ \alpha(1-\theta)-1}   \norm{G(r,u(r))}_{\mathbb H^{\nu} (\Omega)}    dr   \nn\\
	&\lesssim  h^{\alpha(\theta+\nu'-1)}   \|f \|_{   \mathbb H^{\nu+\theta}(\Omega) } .  
	\end{align}  
	{\it Case 2}.  If $1< \al(\theta+\nu'-1)  \le 2$ then since $0 \le t \le \widetilde t \le T$,  we have
	\begin{align}
	&(\widetilde t)^{\al(\theta+\nu'-1) }-t^{\al(\theta+\nu'-1) } \nn\\
	&\quad \quad \quad = \Big[(\widetilde t)^{\al(\theta+\nu'-1)}- t^{\al(\theta+\nu'-1) -1} (\widetilde t) \Big]+ \Big[t^{\al(\theta+\nu'-1) -1} \widetilde t - t^{\al(\theta+\nu'-1) } \Big]\nn\\
	& \quad \quad \quad= \widetilde t \Big[(\widetilde t)^{\al(\theta+\nu'-1) -1}- t^{\al(\theta+\nu'-1) -1}  \Big]+ (\widetilde t-t) t^{\al(\theta+\nu'-1) -1}  \nn\\
	&\quad \quad \quad\le \max \Big(T, T^{\al(\theta+\nu'-1) -1}\Big) \Big[ (\widetilde t-t)^{\al(\theta+\nu'-1) -1}+(\widetilde t-t)\Big].
	\end{align} 
	Therefore
	\begin{align} \label{c111111}
	\|\mathfrak{M}_4\|_{\mathbb H^{\nu- \nu'} (\Omega)} 
	\lesssim 
	% \max \Big(1,T, T^{\al(\theta+\nu'-1) -1}\Big) 
	\left\{
	\begin{array}{llllcccc}
	(\widetilde t-t)^{\al(\theta+\nu'-1)}\mathbf{1}_{0<\al(\theta+\nu'-1)\le 1} \vspace*{0.1cm} \\
	\left( (\widetilde t-t)^{\al(\theta+\nu'-1)-1}+\gamma 
	\right)\mathbf{1}_{1<\al(\theta+\nu'-1)< 2} \\
	\end{array} \right\}	
	\|f \|_{   \mathbb H^{\nu+\theta}(\Omega) } .  
	\end{align}  
	Collecting the results \eqref{c11}, \eqref{c111},\eqref{b111}, \eqref{b1111}, \eqref{c11111} and \eqref{c111111}, we deduce that $ u \in  C\left([0,T];\mathbb H^{\nu-\nu'} (\Omega)\right)$ and finish the desired inequality. \\
	
	\noindent \underline{{\it Part (c)}.} {\it Show that 
		\begin{align*}
		\Vert \partial_t 
		u(t)\Vert_{\mathbb H^{\nu-\nu_1-\frac{1}{\alpha}}(\Omega)}   \lesssim t^{- \alpha (1-{ \theta } -\nu_1  )} \|f\|_{\mathbb{H}^{\nu+\theta}(\Omega)} .
		\end{align*}	
	}
	
	\noindent	In order to establish the result, let us define the following projection operator, for  any $v= \sum_{j=1}^\infty  \langle v,\varphi_j\rangle\varphi_j$ and any $M>0$,  
	$$ {\mathcal P}_{M}    v:=\sum_{j=1}^M    \langle v,\varphi_j \rangle\varphi_j$$  and the 	 two following operator
	\begin{align}
	{ \mathscr  D}_{{\it 1}\hspace*{-0.05cm},\alpha}(t,T) v &:=\sum_{j=1}^\infty  \frac{-\lambda_j t^{\alpha -1} E_{\alpha ,\alpha }(-\lambda_j t^\alpha )}{E_{\alpha ,1}(-\lambda_j T^\alpha )} \langle v,\varphi_j\rangle\varphi_j, \label{D1}\\
	{\mathscr D}_{{\it 2}\hspace*{-0.05cm},\alpha}(t)v &:= \sum_{j=1}^\infty   t^{\alpha -2}E_{\alpha ,\alpha -1}(-\lambda_j t^\alpha )  \langle v,\varphi_j\rangle\varphi_j .   \label{D2}    
	\end{align} 
	Noting that  ${\mathcal P}_{M}  $ has   finite  rank, we have  the following equality after some simple computations
	\begin{align} \label{derivative1}
	\partial_t 
	{\mathcal P}_{M}    u(t)=  { \mathscr D}_{1,\alpha }(t,T) {\mathcal P}_{M}   f &+ \int_0^t  {\mathscr  D}_{2,\alpha }(t-r) {\mathcal P}_{M}    G(r,u(r)) dr \nn\\
	&- \int_0^T  {\mathscr D}_{1,\alpha }(t,T)   {\bf P}_{\alpha }  (T-r) {\mathcal P}_{M}    G(r,u(r)) dr.
	\end{align}
	One can infer from $0\le \nu_1\le 1-\theta$ that $\displaystyle 0\le\nu_1<\frac{2\alpha-1}{\alpha}-\theta$. Hence, this can be  associated with  $\displaystyle\frac{\alpha-1}{\alpha}<\theta<1$  that $\displaystyle 1<\theta+\nu_1+\frac{1}{\al}<2$, and this implies that
	\begin{align}
	t^{\al-1}	\left|\frac{E_{\alpha ,\alpha }(-\lambda_j t^\alpha )}{E_{\alpha ,1}(-\lambda_j T^\alpha )}\right|  &
	\lesssim  	t^{\al-1}\left(  \frac{1+\lambda_j T^\alpha  }{1+\lambda_j t^\alpha  }  \right)^{2-\theta-\nu_1-\frac{1}{\al}} \left(  \frac{1+\lambda_j T^\alpha  } {1+\lambda_j t^\alpha  }  \right)^{ \frac{1}{\al} +\nu_1 +\theta-1}  \nn\\
	&\lesssim  t^{-\al (1-\theta-\nu_1)} \lambda_j^{\frac{1}{\al} +\nu_1 +\theta-1}   . \label{troimua}
	\end{align}
	It is easy to see that
	\begin{align}
	\Bigg\|{\mathscr  D}_{{\it 1}\hspace*{-0.05cm},\alpha}(t,T) \Big( {\mathcal P}_{M'}   -{\mathcal P}_{M}  \Big) f\Bigg\|^2_{\mathbb H^{\nu-\nu_1-\frac{1}{\alpha}}(\Omega)}   	&=\sum_{j=M+1}^{M'}  \lambda_j^{2\nu-2\nu_1-\frac{2}{\alpha}}\Big| \frac{-\lambda_j t^{\alpha -1} E_{\alpha ,\alpha }(-\lambda_j t^\alpha )}{E_{\alpha ,1}(-\lambda_j T^\alpha )} \langle f,\varphi_j\rangle \Big|^2 \nn\\
	&	\lesssim  	  t^{-2\alpha (1-{ \theta } -\nu_1  )}  \sum_{j=M+1}^{M'}  \lambda_j^{ 2\nu+2\theta }   \left|\langle f,\varphi_j\rangle\right|^2. \label{moichuyennhaA}
	\end{align}
	On the other hand, it is certain that 
	\begin{align}  
	&	\Bigg\|  \int_0^t {\mathscr  D}_{2,\alpha }(t-r)   \Big( {\mathcal P}_{M'}   -{\mathcal P}_{M}  \Big)   G(r,u(r)) dr \Bigg\|_{\mathbb H^{\nu-\nu_1-\frac{1}{\alpha}}({\Omega})}  \nn\\
	&\quad \quad \quad \quad \quad \quad \quad \quad \lesssim    \int_0^t (t-r)^{\alpha -2} \left( \sum_{j=M+1}^{M'}   \lambda_j^{2\nu-2\nu_1-\frac{2}{\alpha}}  G_j^2(r,u(r))  \right)^{1/2} dr  \nn\\
	&\quad \quad \quad \quad \quad  \quad \quad \quad\lesssim    \int_0^t (t-r)^{\alpha -2} \left( \sum_{j=M+1}^{M'}   \lambda_j^{2\nu}  G_j^2(r,u(r))  \right)^{1/2} dr . \label{moichuyennhaB}
	\end{align} 
	Now, let us estimate the third term on the right hand side of \eqref{derivative1}.  We see that 
	\begin{align}  
	&\Bigg\| \int_0^T  {\mathscr  D}_{{\it 1}\hspace*{-0.05cm},\alpha}(t,T)   {\bf  P}_{\alpha }  (T-r) \Big( {\mathcal P}_{M'}   -{\mathcal P}_{M}  \Big)  G(r,u(r)) dr \Bigg\|_{\mathbb H^{\nu-\nu_1-\frac{1}{\alpha}}(\Omega)}   \nn\\
	&\lesssim \int_0^T (T-r)^{\alpha-1}  \Bigg( \sum_{j=M+1}^{M'}	\lambda_j^{2\nu-2\nu_1-\frac{2}{\alpha}} \left| t^{-\al (1-\theta-\nu_1)} \lambda_j^{\frac{1}{\al} +\nu_1 +\theta}  \lambda_j^{-\theta} (T-r)^{-\alpha \theta} G_j(r,v(r))   \right|^2 \Bigg)^{\frac{1}{2}} dr  \nn\\
	&\lesssim t^{- \alpha (1-{ \theta } -\nu_1  )} \int_0^T (T-r)^{\alpha(1-\theta)-1}  \Bigg( \sum_{j=M+1}^{M'}	\lambda_j^{2\nu}     G_j^2(r,v(r))     \Bigg)^{\frac{1}{2}} dr, \label{moichuyennhaC}
	\end{align} 
	where we have used  the estimates $ 
	\Big|E_{\alpha,\alpha}(-\lambda_j (T-r)^\alpha)\Big|  \lesssim  \lambda_j^{-\theta} (T-r)^{-\alpha \theta},$  and
	\begin{align} \label{i5}
	&	\Bigg|\lambda_j^{\nu-\nu_1-\frac{1}{\alpha}}\frac{-\lambda_j t^{\al-1} E_{\alpha,\al}(-\lambda_j t^\alpha)}{E_{\alpha,1}(-\lambda_j T^\alpha)}    E_{\alpha,\alpha}(-\lambda_j(T-\tau)^\alpha)  G_j(v(r))   \Bigg|  \nn\\
	& \hspace*{7cm} \lesssim   (T-r)^{-\alpha \theta} t^{-\alpha(1-\theta)}  \lambda_j^{\nu-\nu_1} \Big|  G_j(r,v(r)) \Big|.
	\end{align}	
	Applying the Lebesgue's dominated convergence theorem, we deduce  that three terms    $${\mathscr  D}_{{\it 1}\hspace*{-0.05cm},\alpha}(t,T) \mathcal P_M f, \quad \int_0^t  {\mathscr  D}_{2,\alpha }(t-r)  \mathcal P_M  G(r,u(r)) dr, \quad \int_0^T  {\mathscr  D}_{{\it 1}\hspace*{-0.05cm},\alpha}(t,T)   {\bf  P}_{\alpha }  (T-r) \mathcal P_M  G(r,u(r)) dr$$ 
	are the  Cauchy sequences in the  space $\mathbb H^{\nu-\nu_1-\frac{1}{\alpha}}(\Omega)$.
	Then, we obtain three convergences in the space $\mathbb H^{\nu-\nu_1-\frac{1}{\alpha}}(\Omega)$ as follows
	\begin{align}
	&	\lim_{M \to \infty} {\mathscr  D}_{{\it 1}\hspace*{-0.05cm},\alpha}(t,T) \mathcal P_M f= {\mathscr  D}_{{\it 1}\hspace*{-0.05cm},\alpha}(t,T)   f\nn\\
	&	\lim_{M \to \infty} \int_0^t  {\mathscr  D}_{2,\alpha }(t-r)  \mathcal P_M  G(r,u(r)) dr=\int_0^t  {\mathscr  D}_{2,\alpha }(t-r)    G(r,u(r)) dr \nn\\
	&	\lim_{M \to \infty} \int_0^T  {\mathscr  D}_{{\it 1}\hspace*{-0.05cm},\alpha}(t,T)   {\bf  P}_{\alpha }  (T-r) \mathcal P_M  G(r,u(r)) dr=\int_0^T  {\mathscr  D}_{{\it 1}\hspace*{-0.05cm},\alpha}(t,T)   {\bf  P}_{\alpha }  (T-r)  G(r,u(r)) dr .
	\end{align}
	The above equality  implies that  $\partial_t 
	\mathcal P_M  u(t)$    consequently converges	 to $\partial_t 
	u(t)$ in $\mathbb H^{\nu-\nu_1-\frac{1}{\alpha}}(\Omega)$. Further, the following estimates also hold  
	\begin{align}
	\Vert {\mathscr  D}_{{\it 1}\hspace*{-0.05cm},\alpha}(t,T)   f\Vert _{\mathbb H^{\nu-\nu_1-\frac{1}{\alpha}}(\Omega)}   	 &	\lesssim  	  t^{- \alpha (1-{ \theta } -\nu_1  )}  \|f\|_{\mathbb{H}^{\nu+\theta}(\Omega)}, \nn
	\end{align}
	and
	\begin{align}  
	& \Bigg\|  \int_0^t {\mathscr  D}_{2,\alpha }(t-r) G(r,u(r)) dr \Bigg\|_{\mathbb H^{\nu-\nu_1-\frac{1}{\alpha}}({\Omega})}  \nn\\
	&  \hspace*{3cm} \lesssim     \int_0^t (t-r)^{\alpha -2} \left\| G(r,u(r)) \right\|_{\mathbb{H}^{\nu}(\Omega)}  dr  \nn\\ 
	& \hspace*{3cm} \lesssim \|f\|_{\mathbb{H}^{\nu+\theta}} \int_0^t (t-r)^{\alpha -2}  r^{-\alpha(1-\theta)}  dr  \lesssim t^{- \alpha (1-{ \theta } -\nu_1  )} \|f\|_{\mathbb{H}^{\nu+\theta}} , \nn\\
	& \Bigg\| \int_0^T  {\mathscr  D}_{{\it 1}\hspace*{-0.05cm},\alpha}(t,T)  {\bf  P}_{\alpha }  (T-r)   G(r,u(r)) dr \Bigg\|_{\mathbb H^{\nu-\nu_1-\frac{1}{\alpha}}(\Omega)}   \nn\\
	& \hspace*{3cm} \lesssim t^{- \alpha (1-{ \theta } -\nu_1  )} \int_0^T (T-r)^{\alpha(1-\theta)-1}  \left\| G(r,u(r)) \right\|_{\mathbb{H}^{\nu}(\Omega)} dr  \nn\\
	& \hspace*{3cm} \lesssim t^{- \alpha (1-{ \theta } -\nu_1  )} \|f\|_{\mathbb{H}^{\nu+\theta}(\Omega)} \int_0^T (T-r)^{\alpha(1-\theta)-1} r^{-\alpha(1-\theta)} dr  \nn \\
	& \hspace*{3cm} \lesssim t^{- \alpha (1-{ \theta } -\nu_1  )} \|f\|_{\mathbb{H}^{\nu+\theta}(\Omega)} . \nn 
	\end{align} 	
	Here, we note that   $\int_0^t (t-r)^{\alpha -2}  r^{-\alpha(1-\theta)}  dr \lesssim t^{\alpha \theta - 1}$ and  $t^{\alpha \theta - 1} = t^{- \alpha (1-{ \theta } -\nu_1  )} t^{\alpha\left(1-\nu_1-\frac{1}{\alpha}\right)} \lesssim t^{- \alpha (1-{ \theta } -\nu_1  )} $  in the second estimate  by using \textbf{(AP.1.)} in the Appendix and noting that  $1-\nu_1-\dfrac{1}{\alpha} \ge 0$ as $\nu_1 \le \dfrac{\alpha-1}{\alpha}$. 
	Consolidating all the above arguments, we obtain  that 
	\begin{align}
	\Big\Vert \partial_t 
	u(t)\Big\Vert_{\mathbb H^{\nu-\nu_1-\frac{1}{\alpha}}(\Omega)}  & \le \Big\| { \mathscr D}_{1,\alpha }(t,T) f \Big\|_{\mathbb H^\mu(\Omega)} + \Bigg\| \int_0^T  {\mathscr  D}_{{\it 1}\hspace*{-0.05cm},\alpha}(t,T)  \mathscr { P}_{\alpha }  (T-r) G(r,u(r)) dr \Bigg\|_{\mathbb H^{\nu-\nu_1-\frac{1}{\alpha}}(\Omega)}  \nn\\
	& +\Bigg\| \int_0^t  {\mathscr  D}_{2,\alpha }(t-r) G(r,u(r)) dr \Bigg\|_{\mathbb H^{\nu-\nu_1-\frac{1}{\alpha}}(\Omega)} \lesssim t^{- \alpha (1-{ \theta } -\nu_1  )} \|f\|_{\mathbb{H}^{\nu+\theta}(\Omega)} . \nn
	\end{align}
	Since $\nu_1\le 1-\theta$ and $\displaystyle 1<\theta+\nu_1+\frac{1}{\al}$, these straightforwardly imply that $0<\alpha(1-\theta-\nu_1)<1$ and $\partial_t 
	u \in  L^p (0,    T ;  \mathbb H^{\nu-\nu_1-\frac{1}{\alpha}}(\Omega))$, for all $p \in \left[1; \frac{1}{\alpha(1-\theta-\nu_1)}\right)$.  This completes Part c. \\
	
	\noindent \underline{{\it Part (d)}.} {\it Show that 
		\begin{align}
		\Vert \partial_t^\alpha 
		u(t)\Vert_{\mathbb H^{\nu-\nu_\alpha-\frac{1}{\alpha}}(\Omega)}   
		\lesssim t^{- \alpha \min \left\{ (1-{ \theta } -\nu_\alpha  );  (1-\theta) \right\} } \|f\|_{\mathbb{H}^{\nu+\theta}} . \nn
		\end{align}
	}
	To study the fractional derivative of order $\alpha$ of the mild solution $u$, let us consider the following operators given by
	\begin{align*}
	{ \mathscr D}_{{\it 3}\hspace*{-0.05cm},\alpha}(t,T) w &:= \sum_{j=1}^\infty \frac{-\lambda_j  E_{\alpha ,\alpha }(-\lambda_j t^\alpha )}{E_{\alpha ,1}(-\lambda_j T^\alpha )} \langle w,\varphi_j\rangle\varphi_j,  \nn\\
	{\mathscr D}_{{\it 4}\hspace*{-0.05cm},\alpha}(t)w &:= - \sum_{j=1}^\infty \lambda_j  t^{\alpha -1}E_{\alpha ,\alpha }(-\lambda_j t^\alpha )  \langle w,\varphi_j\rangle\varphi_j .     
	\end{align*} 
	By applying the projection $ \Big( {\mathcal P}_{M'}   -{\mathcal P}_{M}  \Big) $ to the solution $u$, and then calculating the fractional differentiation $\partial^\alpha_t$  	
	\begin{align} \label{derivative2}
	\partial_t ^{\,\alpha} {\mathcal P}_{M}   u(t)&= { \mathscr D}_{3,\alpha }(t,T) {\mathcal P}_{M}  f + \int_0^t  {\mathscr  D}_{4,\alpha }(t-r) {\mathcal P}_{M}   G(r,u(r)) dr \nn\\
	&- \int_0^T  {\mathscr D}_{3,\alpha }(t,T)  \mathscr { P}_{\alpha }  (T-r) {\mathcal P}_{M}  G(r,u(r)) dr+ {\mathcal P}_{M}   G(t,u(t)).
	\end{align}
	By using the fact that  $\displaystyle \frac{1}{\alpha} -\theta <  \frac{2\alpha - 1}{\alpha}-\theta$, it follows from the assumption $\displaystyle \displaystyle \frac{\alpha- 1}{\alpha} -\theta <  \nu_\alpha \le  \frac{1}{\alpha} -\theta$ that   $\displaystyle \frac{\alpha-1}{\alpha}-\theta < \nu_\alpha < \frac{2\alpha - 1}{\alpha}-\theta$. Thus, we find that $\displaystyle 1<\theta + \nu_\alpha + \frac{1}{\alpha}<2$. Therewith, the same techniques as (\ref{moichuyennhaA}) invokes that $  { \mathscr D}_{{\it 3}\hspace*{-0.05cm},\alpha}(t,T) f$ exists in the space $\mathbb H^{\nu-\nu_\alpha-\frac{1}{\alpha}}(\Omega)$ if  $f\in \mathbb H^{\nu+\theta}(\Omega)$, and  
	\begin{align}
	\Vert { \mathscr D}_{{\it 3}\hspace*{-0.05cm},\alpha}(t,T) f\Vert_{\mathbb H^{\nu-\nu_\alpha-\frac{1}{\alpha}}(\Omega)}   \lesssim   t^{-\al\left( \frac{1}{\alpha}-\nu_\alpha-\theta \right)} \|f\|_{\mathbb H^{\nu+\theta}(\Omega)}. \label{dachuyennhaA}
	\end{align}	
	The proof for  integrals $\int_0^t  {\mathscr  D}_{4,\alpha }(t-r)   G(r,u(r)) dr$, and  ~$\int_0^T  {\mathscr D}_{3,\alpha }(t,T)   { \bf P}_{\alpha }  (T-r)   G(r,u(r)) dr$ 	
	in the space $\mathbb H^{\nu-\nu_\alpha-\frac{1}{\alpha}}(\Omega)$ can be done by using the same argument of (\ref{moichuyennhaB}) and (\ref{moichuyennhaC}) by using assumption $(\mathscr H_1)$ and the argument of Cauchy sequences. Aside from the above existence results, we can also verify the following estimates
	\begin{align}  
	\Bigg\|  \int_0^t {\mathscr  D}_{4,\alpha }(t-r)    G(r,u(r)) dr \Bigg\|_{\mathbb H^{\nu-\nu_\alpha-\frac{1}{\alpha}}({\Omega})}    
	&\lesssim    \int_0^t (t-r)^{\alpha -2}  \left\| G(r,u(r)) \right\|_{\mathbb{H}^{\nu}(\Omega)} dr \nn\\
	&\lesssim t^{-\al\left( \frac{1}{\alpha}-\nu_\alpha-\theta \right)} t^{\alpha\left(1-\nu_\alpha-\frac{1}{\alpha}\right)} \|f\|_{\mathbb{H}^{\nu+\theta}(\Omega)} \nn\\
	&\lesssim t^{-\al\left( \frac{1}{\alpha}-\nu_\alpha-\theta \right)} \|f\|_{\mathbb{H}^{\nu+\theta}(\Omega)},\label{dachuyennhaB}
	\end{align} 
	where $1-\nu_\alpha-\dfrac{1}{\alpha} \ge 0$ as $\nu_\alpha \le \dfrac{\alpha - 1}{\alpha}$, and by a similar argument, we obtain
	\begin{equation}
	\Bigg\| \int_0^T  {\mathscr  D}_{{\it 3}\hspace*{-0.05cm},\alpha}(t,T)  {\bf  P}_{\alpha }  (T-r)   G(r,u(r)) dr \Bigg\|_{\mathbb H^{\nu-\nu_\alpha-\frac{1}{\alpha}}(\Omega)}   
	\lesssim t^{- \alpha (1-{ \theta } -\nu_\alpha  )} \|f\|_{\mathbb{H}^{\nu+\theta}(\Omega)} . \label{dachuyennhaC}
	\end{equation}
	Henceforth, we find that
	$
	\Big\Vert \partial_t^\alpha 
	u(t)\Big\Vert_{\mathbb H^{\nu-\nu_\alpha-\frac{1}{\alpha}}(\Omega)}   
	\lesssim t^{- \alpha (1-{ \theta } -\nu_\alpha  )} \|f\|_{\mathbb{H}^{\nu+\theta}(\Omega)} . \nn
	$
	Since $\nu_\alpha\le \dfrac{\alpha-1}{\alpha}$, we have $\alpha( 1-\theta-\nu_\alpha) \ge \alpha\left( 1-\theta - \frac{\alpha-1}{\alpha} \right) = \frac{1}{\alpha}-\theta >\frac{\alpha-1}{\alpha}-\theta >0$. In addition, it can be deduced from $\displaystyle \frac{\alpha-1}{\alpha}-\theta < \nu_\alpha$ that $\displaystyle \alpha(1-\theta-\nu_\alpha) < 1$. Hence, we straightforwardly infer that $0<\alpha(1-\theta-\nu_\alpha)<1$. Moreover, it results from $\nu_\alpha>\dfrac{\alpha-1}{\alpha}-\theta$ that $\nu-\nu_\alpha-\dfrac{1}{\alpha} < \nu - (1-\theta)<\nu$, and the Sobolev  embedding $ \mathbb{H}^{\nu }(\Omega ) \hookrightarrow \mathbb{H}^{\nu-\nu_\alpha-\frac{1}{\alpha}}(\Omega) $ holds. This implies that
	$$\left\| G(t,u(t)) \right\|_{\mathbb{H}^{\nu-\nu_\alpha-\frac{1}{\alpha}}(\Omega)} \lesssim \left\| G(t,u(t)) \right\|_{\mathbb{H}^{\nu}(\Omega)} \lesssim  t^{- \alpha (1-{ \theta }  )}  \|f\|_{\mathbb{H}^{\nu+\theta}(\Omega)}.$$
	Combining the above inequalities finally shows that 
	\begin{align}
	\Vert \partial_t^\alpha 
	u(t)\Vert_{\mathbb H^{\nu-\nu_\alpha-\frac{1}{\alpha}}(\Omega)}   
	\lesssim t^{- \alpha \min \left\{ (1-{ \theta } -\nu_\alpha  );  (1-\theta) \right\} } \|f\|_{\mathbb{H}^{\nu+\theta}(\Omega)} , \nn
	\end{align}   
	and $ \partial_t^{\alpha}  u \in   L^p (0,    T ;  \mathbb H^{\nu-\nu_\alpha-\frac{1}{\alpha}}(\Omega))$, for all $p \in \left[1; \frac{1}{\alpha\min \left\{ (1-{ \theta } -\nu_\alpha  );  (1-\theta) \right\}}\right)$.  The proof is accomplished. 
\end{proof}

\subsection{Proofs of  Theorem \ref{theo2}}
\label{prooft2}
The proof of Theorem \ref{theo2} relies on a contraction mapping principle. 
In order to prove this,  we  prove the  following Lemma
\begin{lemma} \label{BodeSala} Let us pick  $\dfrac{\alpha-1}{\alpha}<\theta<1$,  $0\le \nu\le \sigma \le \nu +1 $ and $1\le q < \dfrac{1}{\alpha(1-\theta)}$. Assume that $f \in  \mathbb H^{\nu+\theta+1}(\Omega) $ and $G$ sastisfies $(\mathscr H_2)$ with $ \| L_2 \|_{L^\infty(0,T)}\in \left (0;   \mathscr M_2^{-1} \right )$.    
	Set 	
	\begin{align} 
	\mathscr{T} v(t) :=  	{\bf B}_{\alpha }(t,  T) f &+ \int_0^t  {\bf P}_{\alpha }  (t-r) G(r,v(r)) dr - \int_0^{ T}  {\bf B}_{\alpha }(t,  T) {\bf P}_{\alpha }  (T-r) G(r,v(r)) dr.  \label{FTheo}
	\end{align}
	Then, for any $v\in   C\left([0,T]; \mathbb H^\nu(\Omega)\right) \cap  L^q(0,T; \mathbb H^{\sigma}(\Omega)) $, it holds that
	$$\mathscr{T} v \in   C\left([0,T]; \mathbb H^\nu(\Omega)\right) \cap  L^q(0,T; \mathbb H^{\sigma}(\Omega)).$$ 
\end{lemma}

\begin{proof}[Proof of Lemma \ref{BodeSala}] We split this proof into the following steps.
	
	\vspace*{0.2cm}	
	
	\noindent \underline{\textit{Prove $\mathscr{T}  v \in  C\left([0,T]; \mathbb H^\nu(\Omega)\right)$}:}	Namely, we need to estimate the norm $\|\mathscr{T} v(\widetilde t)-\mathscr{T} v(t)  \|_{\mathbb H^\nu({\Omega})}$ for all $0 \le t  < \widetilde t \le T$. For more convenience, we will use notation $\mathfrak M_j$, $1\le j\le 4$ as (\ref{b1}) again.  However, the estimates for $\mathfrak M_j$ in Step 2 of the proof of Theorem (\ref{t1}) will be modified suitably to fit the assumptions of $f$ and $G$ is this theorem. Indeed, a slight modification of the techniques in the estimates (\ref{lem2aaaa}) and (\ref{lem2aa}) invokes that    
%	\begin{align} \label{b111}
%	\Big\|  \mathscr{T} v(t_2)-\mathscr{T} v(t_1)  \Big\|_{\mathbb H^{\nu} (\Omega)}
%	&\le \Big\|{  \Big( {\bf B}_{\alpha }(t_2,  T)- {\bf B}_{\alpha }(t_1,  T)  \Big) f }\Big\|_{ \mathbb H^{\nu} (\Omega)} \nn\\
%	& \hspace*{1.5cm} +\Bigg\| \int_0^t \Big(  {\bf  P}_{\alpha }  ( t_2-r) -  {\bf  P}_{\alpha }  (t_1-r) \Big) G(r,u(r)) dr \Bigg\|_{\mathbb H^{\nu} (\Omega)}\nn\\
%	&+\Bigg\| \int_{t_1}^{t_2} { \bf P}_{\alpha }  ( t_2-r)    G(r,u(r)) dr \Bigg\|_{\mathbb H^{\nu} (\Omega)}\nn\\
%	&+\Bigg\| \int_0^T \Big(  { \bf B}_{\alpha }( t_2,  T)-  { \bf B}_{\alpha }(t_1,  T)  \Big)   {\bf  P}_{\alpha }  (T-r)    G(r,u(r)) dr \Bigg\|_{\mathbb H^{\nu} (\Omega)}\nn\\
%	&=\|\mathfrak{M}_5\|_{\mathbb H^{\nu} (\Omega)}+\| \mathfrak{M}_6\|_{\mathbb H^{\nu} (\Omega)}+\|\mathfrak{M}_7  \|_{\mathbb H^{\nu} (\Omega)}+\| \mathfrak{M}_8 \|_{\mathbb H^{\nu} (\Omega)}.
%	\end{align}
%	The term $\|\mathfrak{M}_5\|_{\mathbb H^{\nu} (\Omega)}$ can be bounded by 
	\begin{align}
	\|\mathfrak{M}_1\|_{\mathbb H^{\nu} (\Omega)}
	%= \Big\|{  \Big( {\bf B}_{\alpha }(\widetilde t,  T)- {\bf B}_{\alpha }(t ,  T)  \Big) f }\Big\|_{ \mathbb H^{\nu} (\Omega)}   
	&\le  \int_{t }^{\widetilde t}    \Bigg\|{\sum_{j=1}^\infty   \lambda_j   r^{\alpha -1}  \frac{ E_{\alpha ,\alpha  }(-\lambda_j r^\alpha )}{E_{\alpha ,1}(-\lambda_j  T^\alpha )}   f_j  \varphi_j }\Bigg\|_{ \mathbb H^\nu({\Omega})} d\tau \nn\\
	& \lesssim   \left( \int_{t}^{\widetilde t}    r^{ \alpha \theta -1}  dr \right) \|f \|_{   \mathbb H^{\nu+\theta+1}(\Omega) } \nn\\
	& 
	\lesssim \left\{
	\begin{array}{llllcccc}
	(\widetilde t-t )^{\al\theta}\mathbf{1}_{0<\al\theta\le 1} \vspace*{0.1cm} \\
	\left( (\widetilde t-t )^{\al\theta-1}+\gamma 
	\right)\mathbf{1}_{1<\al\theta< 2} \\
	\end{array} \right\} \big\|f \big\|_{   \mathbb H^{\nu+\theta+1}(\Omega) }, \nn 
	\end{align}
	where we note that $\dfrac{1-\theta}{2}$   belongs to $(0,1)$  ,  and it notes that $0<\alpha-1<\alpha\theta<\alpha<2$ as $\dfrac{\alpha-1}{\alpha}<\theta<1$. 
Next, estimates for the terms $\mathfrak M_j$, $2\le j\le 4$   will be based on the assumption ($\mathscr H_2$) of the nonlinearity $G$. 
	We see that 
	\begin{align}
	\|\mathfrak{M}_2\|_{\mathbb H^{\nu} (\Omega)}
	&\lesssim 
	%{ M}_\al  
	\int_0^{t } \int_{t -r}^{\widetilde t-r} \rho^{\alpha  -2}  \norm{G(r,u(r))}_{\mathbb  H^\nu({\Omega})} d\rho  dr \nn\\
	&\lesssim  
	%	D_1 (\nu, \sigma)  
	\int_0^{t } \int_{t -r}^{\widetilde t-r} \rho^{\alpha  -2}  \norm{G(r,v(r))}_{  \mathbb H^{\nu +1} (\Omega)} d\rho  dr  \nn\\
	&	\lesssim  
	\big\|v\big\|_{  C\left([0,T]; \mathbb H^\nu(\Omega)\right) \cap \mathscr L^q(0,T; \mathbb H^{\sigma}(\Omega))} \int_0^{t } \int_{t -r}^{\widetilde t-r} \rho^{\alpha  -2}    d\rho  dr  ,  \nn
	\end{align}
	where the norm $\norm{G(r,u(r))}_{\mathbb  H^\nu({\Omega})}$ is certainly $\lesssim$-bounded by $\norm{G(r,v(r))}_{  \mathbb H^{\nu +1} (\Omega)}$ due to the embedding $\mathbb H^{\nu+1 }(\Omega)\hookrightarrow \mathbb H^\nu (\Omega)$. Observe from the above estimate that the last right hand side clearly tends to zero as $\widetilde t$ tends to $t$. Hence, the preceding estimate implies the continuity of the term $\mathfrak M_2$ on $\mathbb{H}^\nu(\Omega)$. In addition, the continuity of the term $\mathfrak M_3$ is obvious by using  similar arguments as in \eqref{b1111} and the assumption ($\mathscr H_2$). Precisely, 
	\begin{equation}
	\|\mathfrak{M}_3\|_{\mathbb H^{\nu} (\Omega)} \lesssim  
	\int_{t }^{\widetilde t}  \norm{ G(r,v(r)) }_{\mathbb H^{\nu+1}(\Omega)}  dr \lesssim 
	(\widetilde t-t )  	\big\|v\big\|_{ C\left([0,T]; \mathbb H^\nu(\Omega)\right) \cap  L^q(0,T; \mathbb H^{\sigma}(\Omega))} .  \nn
	\end{equation}
Finally, we consider the  term $\|\mathfrak{M}_4\|_{\mathbb H^{\nu} (\Omega)}$. The idea is to combine  similar arguments as in Step 5 and the modification in the above estimates for $\mathfrak M_1$. Here, the maximum of the spatial smoothness of $G$ should be estimated in the space $\mathbb{H}^{\nu+1}$. Indeed, the following chain of the estimates can be checked 
	\begin{align}
	& \|\mathfrak{M}_8\|_{\mathbb H^{\nu} (\Omega)} \nn\\
	& \le \int_0^T \int_{t_1}^{t_2}\Bigg( \sum_{j=1}^\infty \Big | \lambda_j^\nu (T-\tau)^{\alpha-1}  \lambda_j \rho^{\alpha -1}  \mathscr E_{\al, T} (-\lambda_j \rho^\al)  E_{\alpha,\alpha}(-\lambda_j(T-r)^\alpha) G_j(r,u(r))  \Big|^2    \Bigg)^{\frac{1}{2}} d\rho dr\nn\\
	& \lesssim 
	\int_0^T \int_{t_1}^{t_2} (T-r)^{\al(1-\theta)-1} \rho^{\al \theta-1}  \norm{ G(r,v(r)) }_{\mathbb H^{\nu+1}(\Omega)}  d \rho dr \nn\\
	& \lesssim \big\|v\big\|_{  C\left([0,T]; \mathbb H^\nu(\Omega)\right) \cap  L^q(0,T; \mathbb H^{\sigma}(\Omega))} 
	\Big[(t_2)^{\alpha \theta} -(t_1)^{\alpha \theta}\Big]    \nn\\
	&	\lesssim  \big\|v\big\|_{  C\left([0,T]; \mathbb H^\nu(\Omega)\right) \cap  L^q(0,T; \mathbb H^{\sigma}(\Omega))}
	\left\{
	\begin{array}{llllcccc}
	(t_2-t_1)^{\al\theta}\mathbf{1}_{0<\al\theta\le 1} \vspace*{0.1cm} \\
	\left( 	(t_2-t_1)^{\al\theta-1}+	(t_2-t_1) 
	\right)\mathbf{1}_{1<\al\theta< 2} 
	\end{array} \right\} 		  . \nn
	\end{align}
	The preceding estimates lead to $\mathscr{T} v \in  C\left([0,T]; \mathbb H^\nu(\Omega)\right)$.
	
	\vspace*{0.2cm}	
	
	\noindent \underline{\textit{Prove $\mathscr{T} v \in  L^q(0,T; \mathbb H^{\sigma}(\Omega))$}:}		
	We observe that 
	\begin{align}
	\Big\| {\bf B}_{\alpha }(t,  T) f\Big\|_{\mathbb H^\sigma({\Omega})}^2&= \sum_{j=1}^\infty \lambda_j^{2 \sigma} \left|  \frac{E_{\alpha,1}(-\lambda_j t^\alpha)}{E_{\alpha,1}(-\lambda_j T^\alpha)} \right|^2  \langle f,\varphi_j\rangle^2 \lesssim    t^{-2\alpha (1-{ \theta }  )} \sum_{j=1}^\infty \lambda_j^{2\theta+2\sigma}  \langle f,\varphi_j\rangle^2  \nn\\
	& \lesssim t^{-2\alpha (1-{ \theta }  )} \sum_{j=1}^\infty \lambda_j^{2\theta+2\nu+2}  \langle f,\varphi_j\rangle^2 = t^{-2\alpha (1-{ \theta }  )} \|f\|_{\mathbb{H}^{\nu+\theta+1}}^2   .\label{ef1}
	\end{align}
	Thus, ${\bf B}_{\alpha }(t,  T) f\in  L^q(0,T; \mathbb H^{\sigma}(\Omega))$ since $t^{-\alpha (1-{ \theta }  )} \in L^q(0,T; \mathbb{R})$.  
		 
	Next, we estimate the second term of $\mathcal Jv$ where we will bound the operator norm of $\mathbf P_\alpha(t-r)$ on $\mathbb{H}^\sigma(\Omega)$ by $M_\alpha(t-r)^{\alpha-1}$, and then we estimate $\|G(r,v(r))\|_{\mathbb{H}^\sigma(\Omega)}$ by $\|G(r,v(r))\|_{\mathbb{H}^{\nu+1}(\Omega)}$ upon the assumption  $(\mathscr H_2)$ and the embedding $\mathbb{H}^{\nu+1}(\Omega)  \hookrightarrow \mathbb{H}^\sigma(\Omega)$ as $0\le \sigma\le \nu+1$, see  (\ref{Embed2}). Precisely, these arguments can be performed  as follows   
	\begin{align} \label{e2}
	&	\Bigg\|  \int_0^t  {\bf P}_{\alpha }  (t-r) G(r,v(r)) dr \Bigg\|_{ L^q(0,T; \mathbb H^{\sigma}(\Omega))}^q\nn\\
	& \hspace*{2cm} =  
	\int_0^T  	\left\|  \int_0^t  {\bf P}_{\alpha }  (t-r) G(r,v(r)) dr \right\|_{\mathbb H^{\sigma}(\Omega)}^q  dt\nn\\  
	& \hspace*{2cm} \le \big( C_2(\nu,\sigma){M}_\al \big)^q  \int_0^T  \left( \int_0^t (t-r)^{\alpha-1}    \norm{G(r,v(r)) }_{\mathbb H^{\nu+1}(\Omega)} dr \right)^q dt \nn\\
	& \hspace*{2cm} \le \Big(   \| L_2\|_{L^\infty (0,T)}   \overline{\mathcal M} _1 \Big)^q  	\big\|v\big\|^q_{  C\left([0,T]; \mathbb H^\nu(\Omega)\right) \cap  L^q(0,T; \mathbb H^{\sigma}(\Omega))},
	\end{align} 
where the constant $\overline{\mathcal M} _1$ is given by \textbf{(AP.4)} in the Appendix.
	
Next, we will estimate the $L^q(0,T; \mathbb H^{\sigma}(\Omega))$-norm of the last term of $\mathcal Jv$. The idea is   to combine estimates for the operators ${\bf B}_{\alpha }(t,  T)$ and $\mathbf P_\alpha(T-r)$. We can estimate that the operator ${\bf B}_{\alpha }(t,  T)$  maps  $\mathbb{H}^{\sigma+\theta}$ into $\mathbb{H}^\sigma(\Omega)$, and the operator $\mathbf P_\alpha(T-r)$ maps $\mathbb{H}^\sigma(\Omega)$ into $\mathbb{H}^{\sigma+\theta}(\Omega)$. Therefore, by using   assumption ($\mathscr H_2$),   this term can be estimated on the space $\mathbb{H}^{\sigma}$. In the technical aspect, we also note that the assumption $1-1/(\alpha q)<\theta<1$ guarantees that $(\alpha-1)/\alpha <\theta<1$, and so $\alpha(1-\theta)\in (0;1)$. Moreover, the power function $t^{-\alpha(1-\theta)q}$ is clearly integrable on $(0,T)$.   Indeed, one can show the following chain of   estimates  
	\begin{align}  \label{e3}
	&\Bigg\| \int_0^{ T}  {\bf B}_{\alpha }(t,  T)  {\bf P}_{\alpha }  (T-r)  G(r,v(r)) dr \Bigg\|_{ L^q(0,T; \mathbb H^{\sigma}(\Omega))}^q\nn\\	
	&\hspace*{2cm}=\int_0^T  \left\| \int_0^{ T}  {\bf B}_{\alpha }(t,  T)  {\bf P}_{\alpha }  (T-r) G(r,v(r)) dr \right\|_{\mathbb H^{\sigma}(\Omega)}^q  dt \nn\\
	&\hspace*{2cm} \le \mathcal M_{\alpha,T,\theta}^q  \int_0^T   t^{-\alpha(1-\theta)q}\left( \int_0^T  \norm{{\bf P}_{\alpha }  (T-r) G(r,v(r)) }_{\mathbb H^{\sigma+\theta}(\Omega)} dr \right)^q dt   \nn\\
	&\hspace*{2cm} \le \overline{\mathcal M}_{\alpha,T,\theta}^q  \int_0^T   t^{-\alpha(1-\theta)q}\left( \int_0^T (T-\tau)^{\alpha(1-\theta)-1}    \norm{G(r,v(r)) }_{\mathbb H^{\sigma }(\Omega)} dr \right)^q dt   \nn\\
	&\hspace*{2cm}\le \Big( \|L_2\|_{L^\infty(0,T)} \overline{\mathcal M} _2 \Big)^q	\big\|v\big\|^q_{  C\left([0,T]; \mathbb H^\nu(\Omega)\right) \cap  L^q(0,T; \mathbb H^{\sigma}(\Omega))},
	\end{align} 
where $\mathcal M_{\alpha,T,\theta}:=T^{\alpha(1-\theta)}( T^{\alpha\theta}+\lambda_1^{-\theta}) $, $\overline{\mathcal M}_{\alpha,T,\theta} =\mathcal M_{\alpha,T,\theta}M_\alpha$, and the constant $\overline{\mathcal M} _2$ is given by \textbf{(AP.4)} in the Appendix. 
	A collection of the derived estimates \eqref{ef1}, \eqref{e2}, \eqref{e3}, reveals  $\mathscr{T} v  \in L^q(0,T; \mathbb H^{\sigma}(\Omega)) $. Finally, we wrap up the proof. 
\end{proof}

\begin{proof}[Proof of Theorem \ref{theo2}]   In order 
	to show that Problem (\ref{mainprob1})-(\ref{mainprob4}) has a unique mild solution, we will prove the operator $\mathscr{Q}$ has a unique fixed point in $   C\left([0,T]; \mathbb H^\nu(\Omega)\right) \cap  L^q(0,T; \mathbb H^{\sigma}(\Omega)) $. The proof is  based on the Banach contraction principle.
	We have
	\begin{align}  
	\,& \Big\| \mathscr{T} v_1- \mathscr{T} v_2\Big \|_{  C\left([0,T]; \mathbb H^\nu(\Omega)\right) \cap  L^q(0,T; \mathbb H^{\sigma}(\Omega))} \nn\\
	=\,& \|\mathscr{T} v_1- \mathscr{T} v_2 \|_{ L^q(0,T; \mathbb H^{\sigma}(\Omega))}+\|\mathscr{T} v_1- \mathscr{T} v_2 \|_{ C([0,T];\mathbb H^\nu({\Omega}))} \nn\\
	:=\,&  \mathfrak M_5+ \mathfrak M_{6}. \label{f1}
	\end{align}
	For estimating $\mathfrak M_5$,  we apply the previous results in estimating \eqref{e2} and \eqref{e3} to obtain   
	\begin{align} \label{f2}
	\mathfrak M_5 & \le  \left\|  \int_0^t   {\bf  P}_{\alpha }  (t-r) \left(  G(r,v_1(r)) - G(r,v_2(r))  \right)dr \right\|_{ L^q(0,T; \mathbb H^{\sigma}(\Omega))} \nn\\
	& \hspace*{1.5cm} +\left\|  \int_0^T  {\bf  B}_{\alpha }(t,  T) {\bf  P}_{\alpha }   (T-r) \left(  G(r,v_1(r)) - G(r,v_2(r))  \right)dr \right\|_{  L^q(0,T; \mathbb H^{\sigma}(\Omega))} \nn\\
	&\le   \| L_2\|_{L^\infty (0,T)} \big(\overline{\mathcal M}_1 + \overline{\mathcal M}_2\big) \big\|v_1-v_2\big\|_{  C\left([0,T]; \mathbb H^\nu(\Omega)\right) \cap  L^q(0,T; \mathbb H^{\sigma}(\Omega)) }.
	\end{align}
On the other hand, to bound the term  $\mathfrak M_{6}$, we   estimate the operator norm of ${\bf P}_{\alpha }  (t-r)$ acting on $\mathbb{H}^{\nu}(\Omega)$ by $M_\alpha (t-r)^{\alpha-1}$, and of ${\bf B}_{\alpha }(t,  T)  {\bf  P}_{\alpha }  (T-r)$ acting from $\mathbb{H}^\nu(\Omega)$ to $\mathbb{H}^{\nu+1}(\Omega)$ by $M_\alpha^2 m_\alpha^{-1} (T^\alpha+\lambda_1^{-1})$. By applying the embedding $\mathbb{H}^\nu(\Omega)\hookrightarrow \mathbb{H}^{\nu+1}(\Omega)$ and   assumption ($\mathscr H_2$) we can deduce the following estimates 
	\begin{align} 
	\mathfrak M_{6}& \le  \left\|  \int_0^t  {\bf P}_{\alpha }  (t-r) \left(  G(r,v_1(r)) - G(r,v_2(r))  \right)dr \right\|_{\mathscr C([0,T];\mathbb H^\nu({\Omega}))} \nn\\
	& \hspace*{2cm} +\left\|  \int_0^T  {\bf B}_{\alpha }(t,  T)  {\bf  P}_{\alpha }  (T-r) \left(  G(v_1(r)) - G(v_2(r))  \right)dr \right\|_{  C([0,T];\mathbb H^\nu({\Omega}))} \nn\\
	&\le {M}_\al \sup_{0\le t\le T}\Bigg(  \int_0^t (t-r)^{\alpha-1}    \norm{G(r,v_1(r)) - G(r,v_2(r))}_{H^{\nu}(\Omega)} dr \Bigg)\nn\\
	&+ \frac{M_\alpha^2}{m_\alpha}(T^\alpha+\lambda_1^{-1}) \int_0^T (T-r)^{\alpha-1}    \norm{G(r,v_1(r))- G(r,v_2(r)) }_{\mathbb H^{\nu+1}(\Omega)} dr  \nn\\
	&\le \| L_2\|_{L^\infty (0,T)} \overline{\mathcal M}_3 \big\|v_1-v_2\big\|_{  C\left([0,T]; \mathbb H^\nu(\Omega)\right) \cap  L^q(0,T; \mathbb H^{\sigma}(\Omega)) } . \label{f3}
	\end{align}
	 
	A collection of the   estimates \eqref{f1}, \eqref{f2}, \eqref{f3} implies that
	$$\Big\|\mathcal J v_1- \mathcal J v_2 \Big\|_{{ C\left([0,T]; \mathbb H^\nu(\Omega)\right) \cap  L^q(0,T; \mathbb H^{\sigma}(\Omega)) }} \le    \| L_2\|_{L^\infty (0,T)} \mathscr M_2 \Big\|v_1-v_2\Big\|_{ C\left([0,T]; \mathbb H^\nu(\Omega)\right) \cap  L^q(0,T; \mathbb H^{\sigma}(\Omega)) }.$$
	Since $     \| L_2\|_{L^\infty (0,T)}   \mathscr M_2 <1$, we conclude that $\mathcal J $ is a contraction in $ C\left([0,T]; \mathbb H^\nu(\Omega)\right) \cap  L^q(0,T; \mathbb H^{\sigma}(\Omega))$ which 
	ensures the existence and uniqueness of a fixed point. 
	The desired inequality is easy to obtain.  Hence, we finalize the proof.
\end{proof}

\subsection{Proof of Theorems \ref{locally1}} 
\label{prooft3}
To start with, let us prove the following lemmas.

\begin{lemma} \label{locally} Assume that all assumptions of Theorem \ref{locally1} are fulfilled.   
\begin{itemize}
\item[a)] For $t>0$, and $\mathcal N_2 $   given by \textbf{(AP.4.)} in the Appendix, we have
\begin{align}
	\left\|{\bf B}_\alpha(t,T)f \right\|_{\mathbb{H}^{\nu}(\Omega)} \le \mathcal N_2  t^{-\alpha\vartheta } \|f\|_{\mathbb{H}^{\nu+( 1-\vartheta) }(\Omega)} , \label{hhhh1A}
	\end{align} 
	  Moreover,  the following convergence holds
	\begin{align}
	{	\bf B}_\alpha(\widetilde t,T) f \xrightarrow{\widetilde t \to t} {	\bf B}_\alpha(t,T)   f  \quad \textrm{in } \quad  \mathbb{H}^{\nu}(\Omega). \label{hhhh1B}
	\end{align}
\item[b)] For $t>0$, $w\in 	\mathfrak X_{\al,\vartheta, \nu,T }$, and $\mathscr N_2$   given by \textbf{(AP.4)} in the Appendix, it follows	
\begin{align}
	\left\| \int_0^t {	\bf P}_{\alpha }  (t-r) G(r,w(r)) dr \right\|_{\mathbb{H}^{\nu}(\Omega)}  \le  \mathscr N_2 K_0 \big( T^{s\alpha \vartheta } + \mathcal R^s \big) t^{-\alpha\vartheta} \|w\|_{  C^{\alpha\vartheta}((0,T];\mathbb{H}^{\nu}(\Omega))} .  \label{hhhh2A}
	\end{align}   
	 Moreover,  the following convergence holds
	\begin{align}
	\int_0^{\widetilde t} {	\bf P}_{\alpha }  (\widetilde t-r) G(r,w(r)) dr \xrightarrow{\widetilde  t \to t}  \int_0^{ t}  { \bf P}_{\alpha }  (t-r) G(r,w(r)) dr \quad \textrm{in} \quad  \mathbb{H}^{\nu}(\Omega).  \label{hhhh2B}
	\end{align}
\item[c)]  For $t>0$, $w\in 	\mathfrak X_{\al,\vartheta, \nu,T }$, it holds
\begin{align}
	\left\|\int_0^t  { \bf  P}_{\alpha }  (t-r) G(r,w(r)) dr\right\|_{\mathbb{H}^{\nu+1-\vartheta}(\Omega)} \le  \mathscr N_2 K_0 \big( T^{s\alpha \vartheta } + \mathcal R^s \big) t^{-\alpha \vartheta } \|w\|_{  C^{\alpha \vartheta }((0,T];\mathbb{H}^{\nu}(\Omega))} . \label{hhhh22A}
	\end{align}
\end{itemize} 
\end{lemma}

\begin{proof}
	
	{\underline {\it Proof of Part (a)}. }
	
	\noindent 	By applying the first part of Lemma \ref{nayvesom}, we obtain
	\begin{align*}
	\Big \|{\bf B}_\alpha(t,T)f \Big\|_{\mathbb{H}^{\nu}(\Omega)}^2 = \sum_{j=1}^\infty \left| \frac{E_{\alpha,1}(-\lambda_jt^\alpha)}{E_{\alpha,1}(-\lambda_jT^\alpha)} \right|^2 f_j^2 \le \mathcal N_2^2 t^{-2\alpha \vartheta } \norm{f}^2_{\mathbb{H}^{\nu+(1- \vartheta)}(\Omega)}, 
	\end{align*}
	where $\mathcal N_2$ is given by   \textbf{(AP.4.)} in the Appendix. This directly implies the inequality (\ref{hhhh1A}). Let us proceed to prove the convergence (\ref{hhhh1B}). 
	By the fact that $E_{\alpha,\alpha}(-z)\lesssim (1+z^2)^{-1}$ for all $z\ge 0$, see e.g. \cite{Samko,Podlubny,Diethelm}, one can apply the same techniques as (\ref{lem2aa}) to show the following inequalities 
	\begin{align}
	\left| \frac{ E_{\alpha ,\alpha  }(-\lambda_j r^\alpha )}{E_{\alpha ,1}(-\lambda_j  T^\alpha )} \right| \lesssim \frac{1+\lambda_j T^\alpha}{\left[1+(\lambda_j r^\alpha)^2\right]^{1-\frac{1-\vartheta }{2}}} \lesssim \lambda_j \left[ (\lambda_j r^\alpha)^2\right]^{ \frac{1-\vartheta }{2}-1}, \label{donvo}
	\end{align}
	where $0<(1-\vartheta)/2<(1-\mu )/2<1$. 
	Hence,  we derive that   
	\begin{align*}
	\Big\| {\bf B}_\alpha(\widetilde t,T) f  - {\bf B}_\alpha(t,T)   f  \Big\|_{\mathbb{H}^{\nu}(\Omega)}  &\le \int_{t}^{\widetilde t} r^{\alpha -1} \left\| \sum_{j=1}^\infty \lambda_j   \frac{ E_{\alpha ,\alpha  }(-\lambda_j r^\alpha )}{E_{\alpha ,1}(-\lambda_j  T^\alpha )} f_j \varphi_j \right\|_{\mathbb{H}^{\nu}(\Omega)}  dr,  \nn\\
	&\lesssim \|f\|_{\mathbb{H}^{\nu+(1- \vartheta) }(\Omega)} \int_{t}^{\widetilde t} r^{-\alpha\vartheta  -1}   dr.
	\end{align*}
	Since the integral in the above inequality tends to zero as $t$ approaches $\widetilde t$ from the right, we obtain (\ref{hhhh1B}) and finish the proof of Part (a).
	
	{\underline {\it Proof of Part (b)}. }\\
	We divide this proof into two parts as follows.\\
	\noindent {\it Step  1.} Prove the inequality (\ref{hhhh2A}).   It follows from $E_{\alpha,\alpha}(-\lambda_j(t-r)^\alpha) \le M_\alpha \lambda_j^{-\mu } (t-r)^{-\alpha \mu  }$ that 
	\begin{align}\label{loo1}
	&\left\| \int_0^t {\bf  P}_{\alpha }  (t-r) G(r,w(r)) dr \right\|_{\mathbb{H}^{\nu}(\Omega)} \nn\\
	&\quad \quad \quad \quad \quad \quad \le \int_0^t (t-r)^{\alpha-1} \left\| \sum_{j=1}^\infty E_{\alpha,\alpha}(-\lambda_j(t-r)^\alpha) G_j(w(r)) \varphi_j  \right\|_{\mathbb{H}^{\nu}(\Omega)} dr\nn\\
	&\quad \quad \quad \quad \quad \quad \le   M_\alpha \int_0^t (t-r)^{\alpha(1-\mu )-1} \left\| G(r,w(r))  \right\|_{\mathbb{H}^{{\sigma} }(\Omega)} dr, 
	\end{align}
	where $\sigma=\nu-\mu$. 		Since $w \in 	\mathfrak X_{\al,\vartheta, \nu,T } (\mathcal R) $ , we see    that   $\|w(r)\|_{\mathbb{H}^{\nu}(\Omega)} \le \mathcal R r^{-\alpha\vartheta  }$. Thus, we have in view of \eqref{H3} that
	$ 
	\left\| G(r,w(r))  \right\|_{\mathbb{H}^{{\sigma} }(\Omega)}\le {L}_3 (r)  \Big( 1+\|w(r)\|^s_{\mathbb{H}^{\nu}(\Omega)} \Big)\|w(r)\|_{\mathbb{H}^{\nu}(\Omega)}.$ 
	It follows from \eqref{loo1} that
	\begin{align}
	\left\| \int_0^t { \bf P}_{\alpha }  (t-r) G(r,w(r)) dr \right\|_{\mathbb{H}^{\nu}(\Omega)}  
	\le  M_\alpha \|w\|_{  C^{\alpha \vartheta }((0,T];\mathbb{H}^{\nu}(\Omega))} \widehat{{L}}_3(t)  \label{chieuCN}
	\end{align}
	where
	\begin{align}
	\widehat{{L}}_3(t):= \int_0^t (t-r)^{\alpha(1-\mu )-1}  \left(r^{-\alpha \vartheta  } + {\mathcal R}^{s}r^{-(1+s)\alpha \vartheta } \right) {{L}}_3(r) dr. \label{truaCN}
	\end{align}
	Our next purpose is to find an upper bound of $	\widehat{{L}}_3(t)$. In order to control this term, we observe from $0<r<T$ that 
	$r^{-\alpha \vartheta  } \le T^{s\alpha \vartheta  } r^{-(1+s)\alpha \vartheta }$, and from $K_0= \|	{L}_3(t) t^{\al \zeta }\|_{L^\infty(0,T)} 
	$ that ${L}_3(r)\le K_0r^{-\alpha\zeta}$ which yields the following estimates
	\begin{align}
	\widehat{{L}}_3(t) &\le    \big( T^{s\alpha \vartheta } + {\mathcal R}^s \big)	\int_0^t (t-r)^{\alpha(1-\mu )-1}    r^{-(1+s)\alpha \vartheta }  {L}_3(r) dr \nn\\
	& \le K_0 \big( T^{s\alpha \vartheta } + {\mathcal R}^s \big) \int_0^t (t-r)^{\alpha(1-\mu )-1}    r^{-\alpha\left((1+s) \vartheta  +\zeta\right)} dr. \nn\ 
	\end{align}
	By noting  
	$\min \Big \{ \alpha(1-\mu )-1;~ -\alpha\left((1+s) \vartheta +\zeta\right) \Big\} >-1$ as $0<\mu<1$, $\zeta<\alpha^{-1}-(1+s)\vartheta$, and using \textbf{(AP.1.)} in the Appendix, 
	we find  that
	\begin{equation}
	\int_0^t (t-r)^{\alpha(1-\mu )-1}    r^{-\alpha\left((1+s) \vartheta  +\zeta\right)} dr \le \mathscr N_1  t^{\alpha\left((1 - \mu )-(1+s) \vartheta -\zeta \right)}. \nn
	\end{equation}
	This implies that
	\begin{align}
	\widehat{L}_3(t)  \le
	K_0 \big( T^{s\alpha \vartheta } + {\mathcal R}^s \big) \mathscr N_1  T^{\alpha\left((1-\mu )-s \vartheta -\zeta\right)} ~ t^{-\alpha \vartheta }, \nn
	\end{align}
	where we have noted that  $\zeta \le  (1 -\mu   )- s \vartheta$ since $\zeta<\alpha^{-1}-\vartheta-s\vartheta\le (1-\mu)-s\vartheta$ as $\alpha^{-1}<1$ and $\vartheta>\mu$. 
	The latter estimate together with   (\ref{chieuCN}) and (\ref{truaCN})  that
	\begin{align}
	\left\| \int_0^t {\bf  P}_{\alpha }  (t-r) G(r,w(r)) dr \right\|_{\mathbb{H}^{\nu}(\Omega)}  
	\,&\le  \mathscr N_2 K_0\big( T^{s\alpha \vartheta } + \mathcal R^s \big) t^{-\alpha \vartheta  } \big\|w\big\|_{ C^{\alpha \vartheta  }((0,T];\mathbb{H}^{\nu}(\Omega))}      ,\nn
	\end{align}
	where we recall that $\mathscr N_2$ is given by \textbf{(AP.4)} in the Appendix. 
	
	\vspace*{0.2cm}
	
	\noindent {\it Step  2.} Show that  (\ref{hhhh2B}) holds. 
	By dealing  with $\left\| G(r,w(r))  \right\|_{\mathbb{H}^{{\sigma}}(\Omega)}$ using   the same arguments  in   Step 1, we derive that
	\begin{align}
	\|\mathfrak M_2 \|_{\mathbb{H}^{\nu}(\Omega)} \,& \le  \int_0^t \int_{t-r}^{\widetilde t-r}   \rho^{\alpha  -2} \left\| \sum_{j=1}^\infty  E_{\alpha  ,\alpha  -1}(-\lambda_j \rho^\alpha  ) G_j(r,u(r))\varphi_j \right\|_{\mathbb{H}^{\nu}(\Omega)} d\rho  dr   \nn\\
	\,& \lesssim  \int_0^t \int_{t-r}^{\widetilde t-r}    \rho^{\alpha(1-\mu )  -2} \|G(r,u(r))\|_{\mathbb{H}^{{\sigma}}(\Omega)} d\rho  dr   \nn \\
	\,& \lesssim   \int_0^t \int_{t-r}^{\widetilde t-r}   \rho^{\alpha(1-\mu )  -2} \left(\mathcal Rr^{-\alpha \vartheta } +\mathcal R^{1+s}r^{-(1+s)\alpha \vartheta } \right) L_3 (r) d\rho  dr   \nn\\
	\,& \lesssim   \left|  \int_0^t   \left( (\widetilde t-r)^{\alpha(1-\mu )  -1} - (t-r)^{\alpha(1-\mu )  -1} \right)  r^{-\alpha\left((1+s) \vartheta +\zeta\right)}  dr \right|  , \nn 
	\end{align} 
	where $\mathfrak{M}_2$ is formulated by (\ref{b1}). 	 By the fact that $\alpha(1-\mu )>0$ and $1-\alpha\left((1+s) \vartheta  +\zeta\right)>0$ and  using  \textbf{(AP.2.)} in the Appendix, we know that the right hand-side of the latter inequality   tends to zero,  as $\widetilde t$ approaches $t$.   Hence, $	\|\mathfrak M_2 \|_{\mathbb{H}^{\nu}(\Omega)}  \xrightarrow{\widetilde t \to t } 0.$
	Now, in the same way as above, we obtain  
	\begin{align} \label{loo2}
	\|\mathfrak M_3\|_{\mathbb{H}^{\nu}(\Omega)} 
	&\le  \int_t^{\widetilde t} (\widetilde t-r)^{\alpha-1}  \left\| \sum_{j=1}^\infty     	 E_{\alpha,\alpha}(-\lambda_j (\widetilde{t}-\tau)^\alpha) G_j(r,u(r)) \varphi_j   \right\|_{\mathbb{H}^{\nu}(\Omega)}  dr \nn\\
	&\lesssim  \int_t^{\widetilde t} (\widetilde t-r)^{\alpha(1-\mu )-1}  \|G(r,u(r))\|_{\mathbb{H}^{{\sigma}}(\Omega)}  dr\nn\\ &\lesssim \int_t^{\widetilde t} (\widetilde t-r)^{\alpha(1-\mu )-1}  r^{-\alpha\left((1+s) \vartheta +\zeta\right)}  dr,
	\end{align}  
	where $\mathfrak{M}_3$ is formulated by (\ref{b1}).	From  that $(\widetilde t-r)^{\alpha(1-\vartheta) } \le (\widetilde t-t)^{\alpha(1-\vartheta)}$ as $t\le r \le \widetilde t$, we bound the right hand-side of   \eqref{loo2} as follows
	\begin{align}
	\text{	(RHS) of } \eqref{loo2} \,&\le (\widetilde t-t)^{\alpha(1-\vartheta)} \int_0^{\widetilde t} (\widetilde t-r)^{\alpha( \vartheta-\mu  )-1}  r^{-\alpha\left((1+s) \vartheta +\zeta\right)}  dr \nn\\
	\,&\lesssim (\widetilde t-t)^{\alpha(1-\vartheta)} \int_0^{\widetilde t} (\widetilde t-r)^{\alpha(\vartheta -\mu   )-1}  r^{-\alpha\left((1+s) \vartheta +\zeta\right)}  dr, \nn
	\end{align}
	Noting that $\alpha(\vartheta-\mu  )>0$ and  $1-\alpha\left((1+s)\vartheta+\zeta\right)>0$ , we ensure that $\int_0^{\widetilde t} (\widetilde t-r)^{\alpha(\vartheta -\mu   )-1}  r^{-\alpha\left((1+s) \vartheta +\zeta\right)}  dr$ is convergent.   The above observations imply that 
	$	\|\mathfrak M_3\|_{\mathbb{H}^{\nu}(\Omega)}  \xrightarrow{\widetilde  t \to t } 0.$
	Since    $$ \int_0^{\widetilde t }   {\bf  P}_{\alpha }  (\widetilde t-r) G(r,w(r)) dr - \int_0^{t} { \bf P}_{\alpha }  (t-r) G(r,w(r)) dr =\mathfrak M_2+\mathfrak M_3,$$ we finish this step.
	
	\vspace*{0.2cm}
	
	\noindent 	{\underline {\it Proof of Part (c)}. }
	\noindent	In view of $0\le 1+[(\nu-\sigma)-\vartheta]\le 1$, one can see that
	\begin{align*}
	\,&\left\| \int_0^t  { \bf P}_{\alpha }  (t-r) G(r,w(r)) dr \right\|_{\mathbb{H}^{\nu+(1-\vartheta)}(\Omega)}  \nn\\
	& \hspace*{3cm}	\le  M_\alpha \int_0^t (t-r)^{\alpha( \vartheta-\mu  )-1} \left\| G(r,w(r))  \right\|_{\mathbb{H}^{{\sigma} }(\Omega)} dr \nn\\
	&	\hspace*{3cm} \le M_\alpha K_0(T^{s\alpha\vartheta}+\mathcal R^s) \|w\|_{ C^{\alpha \vartheta }((0,T];\mathbb{H}^{\nu}(\Omega))} \int_0^t (t-r)^{\alpha( \vartheta-\mu )-1} r^{-\alpha\left((1+s) \vartheta +\zeta\right)} dr \nn\\
	& \hspace*{3cm} \le	M_\alpha K_0(T^{s\alpha\vartheta}+\mathcal R^s) \|w\|_{ C^{\alpha \vartheta }((0,T];\mathbb{H}^{\nu}(\Omega))} \mathscr N_1  t^{\alpha\left(( \vartheta-\mu   )-(1+s) \vartheta -\zeta \right)}  \nn\\
	& \hspace*{3cm} \le   \mathscr N_2 K_0 \big( T^{s\alpha \vartheta } + \mathcal R^s \big)  t^{-\alpha \vartheta } \|w\|_{ C^{\alpha \vartheta }((0,T];\mathbb{H}^{\nu}(\Omega))},
	\end{align*} 	where we also recall that $\mathscr N_2$ is given by \textbf{(AP.4)} in the Appendix.  
	This   completes the proof.
\end{proof}

\begin{proof}[Proof of Theorem \ref{locally1}]  
	
	The proof will be based on a contraction mapping theorem on a Banach space. For this purpose, let us define the mapping $$\mathscr Q: 	\mathfrak X_{\al,\vartheta, \nu,T }(\mathcal R)  \longrightarrow 	\mathfrak X_{\al,\vartheta, \nu,T }(\mathcal R) $$ given by
	\begin{equation}\label{fixedpoint}
	\mathscr Qw= 	   {\bf B}_\alpha(t,  T) f + \int_0^t  {\bf B}_\alpha  (t-r) G(r,w(r)) dr - \int_0^{ T}    {\bf  B}_{\alpha }(t,  T)    {\bf  P}_{\alpha } (T-r) G(r,w(r)) dr.
	\end{equation}
	Since $f\in \mathbb{H}^{\nu+(1- \vartheta)}(\Omega)$, the convergence (\ref{hhhh1B}) in Part a of Lemma \ref{locally} yields that the first term  of $\mathscr Q$ is time-continuous for all  $0<t\le T$. The estimate (\ref{hhhh1A}) means that this term belongs to $ C^{\alpha \vartheta }((0,T]; \mathbb{H}^{\nu}(\Omega))$. Similarly, we observe from $G$ satisfying assumption $(\mathscr H_3)$  and the estimate (\ref{hhhh2A}),  the convergence (\ref{hhhh2B}) in in Part b of Lemma \ref{locally} that the second term of $\mathscr Q$  belongs to $ C^{\alpha \vartheta }((0,T];\mathbb{H}^{\nu}(\Omega))$. On the other hand,  using    Part c of Lemma \ref{locally} shows that the integral   
	$  \int_0^T  { \bf P}_{\alpha }  (T-r) G(r,w(r)) dr $ belongs to $ \mathbb{H}^{\nu+(1- \vartheta)}(\Omega)$ , so we deduce from Part a of Lemma \ref{locally} that
	\begin{equation}
	{\bf B}_\alpha(t,T)\int_0^T  {\bf  P}_{\alpha }  (T-r) G(r,w(r)) dr ~~~\text{belongs to }~~~  C^{\alpha \vartheta  }((0,T]; \mathbb{H}^{\nu}(\Omega) ).
	\end{equation}
	Therefore, the last term of $\mathscr Q$ also belongs to $ C^{\alpha \vartheta }((0,T];\mathbb{H}^{\nu}(\Omega))$.  
	
	\vspace*{0.2cm}
	
	\underline{Prove $\mathscr Q$   maps $\mathfrak X_{\al,\vartheta, \nu,T }(\mathcal R)$ into itself:} 
	Indeed, let $w^\dagger$, $w^\ddagger$ belong to  the space $\mathfrak X_{\al,\vartheta, \nu,T }(\mathcal R)$, then using the formula  \eqref{fixedpoint}  implies the following chain of   estimates 
	\begin{align} \label{fixedpoint}
	\,& t^{ \alpha \vartheta} \left\|\mathscr Q w^\dagger(t) - \mathscr Q w^\ddagger(t) \right\|_{\mathbb{H}^{\nu}(\Omega)} \nn\\
	\,& \hspace*{3cm} \le t^{ \alpha \vartheta }\left\| \int_0^t {\bf  P}_{\alpha }  (t-r) \Big( G(r,w^\dagger(r)) - G(r,w^\ddagger(r))\Big) dr \right\|_{\mathbb{H}^{\nu}(\Omega)} \nn\\
	\,&\hspace*{3cm} + t^{ \alpha \vartheta }\left\| {\bf B}_\alpha(t,T)\int_0^T  { \bf P}_{\alpha }  (T-r) \Big( G(r,w^\dagger(r)) - G(r,w^\ddagger(r))\Big) dr \right\|_{\mathbb{H}^{\nu}(\Omega)} \nn\\
	\,& \hspace*{3cm} \le \mathscr N_2 K_0 (T^{s\alpha\vartheta}+\mathcal R^s)  \left\|w^\dagger-w^\ddagger \right\|_{ C^{\alpha \vartheta }((0,T]; \mathbb{H}^{\nu}(\Omega))} \nn\\
	\,&\hspace*{3cm}  +  \mathcal N_2     \left\| \int_0^T  {\bf  P}_{\alpha }  (T-r) \Big( G(r,w^\dagger(r)) - G(r,w^\ddagger(r))\Big) dr \right\|_{\mathbb{H}^{\nu+(1-\vartheta)}(\Omega)} \nn \\
	\,& \hspace*{3 cm}  \le \overline{\mathscr N_2} K_0   (T^{s\alpha\vartheta}+\mathcal R^s)    \|w^\dagger-w^\ddagger\|_{  C^{\alpha \vartheta }((0,T];\mathbb{H}^{\nu}(\Omega))},\nn
	\end{align}
	where on the right hand-side of \eqref{fixedpoint}, 
	we have used  the inequalities (\ref{hhhh1A}) of  Lemma \ref{locally}, (\ref{hhhh2A}) of Lemma  \ref{locally} in the first  estimate, and the inequality (\ref{hhhh22A}) of Lemma \ref{locally} in the second estimate. 
	This implies that
	\begin{equation}
	\left\|\mathscr Q w^\dagger - \mathscr Q w^\ddagger \right\|_{  C^{\alpha \vartheta }((0,T];\mathbb{H}^{\nu}(\Omega))}	\le 	\overline{\mathscr N_2} K_0   (T^{s\alpha\vartheta}+\mathcal R^s)  \|w^\dagger-w^\ddagger\|_{  C^{\alpha \vartheta }((0,T];\mathbb{H}^{\nu}(\Omega))}. \label{ngaymoia}
	\end{equation}
	By letting  $w^\ddagger=0$ into the latter equality and noting that $\mathscr Q \overline w^\ddagger (t)=   {\bf B}_\alpha(t,  T) f$ if $w^\ddagger=0$ , we derive  
	\begin{equation}
	\left\|\mathscr Q w^\dagger -  {\bf B}_\alpha(t,  T) f \right\|_{  C^{\alpha \vartheta }((0,T];\mathbb{H}^{\nu}(\Omega))} \le 	\overline{\mathscr N_2} K_0   (T^{s\alpha\vartheta}+\mathcal R^s)  \|w^\dagger\|_{  C^{\alpha \vartheta }((0,T];\mathbb{H}^{\nu}(\Omega))}. \nn
	\end{equation}
	From \eqref{hhhh1A} and using the triangle inequality, we know that 
	\begin{align}
	\left\|\mathscr Q w^\dagger \right\|_{  C^{\alpha \vartheta }((0,T];\mathbb{H}^{\nu}(\Omega))}  &\le 
	\left\|\mathscr Q w^\dagger -  {\bf B}_\alpha(t,  T) f \right\|_{  C^{\alpha \vartheta }((0,T];\mathbb{H}^{\nu}(\Omega))} + \sup_{0\le t \le T} t^{\alpha\vartheta} \|{\bf B}_\alpha(t,  T) f \|_{\mathbb{H}^{\nu }(\Omega)}\nn\\
	&\le 	\overline{\mathscr N_2} K_0   (T^{s\alpha\vartheta}+\mathcal R^s)  \|w^\dagger\|_{  C^{\alpha \vartheta }((0,T];\mathbb{H}^{\nu}(\Omega))} + \mathcal N_2   \|f\|_{\mathbb{H}^{\nu+( 1-\vartheta) }(\Omega)} .\nn
	\end{align}	
	Since $ w^\dagger \in \mathfrak X_{\al,\vartheta, \nu,T }(\mathcal R) $, we have $\|w^\dagger\|_{  C^{\alpha \vartheta }((0,T];\mathbb{H}^{\nu}(\Omega))}  \le \mathcal R $. It implies that  
	\begin{align}
	\left\|\mathscr Q w^\dagger \right\|_{  C^{\alpha \vartheta }((0,T];\mathbb{H}^{\nu}(\Omega))}  \le \underbrace{  \overline{\mathscr N_2} K_0   (T^{s\alpha\vartheta}+\mathcal R^s) \mathcal R +\mathcal N_2   \|f\|_{\mathbb{H}^{\nu+( 1-\vartheta) }(\Omega)} }_{:=\pi(\mathcal R)} .\label{ngaymoi}
	\end{align}

	Due to the assumption $ K_0T^{s\alpha \vartheta} \in \big(  0;\min\big\{\frac{1}{2}\overline{\mathscr N_2}^{\,-1};\mathcal N_f\big\}  \big),$ we now show that there exists $0< \overline{\mathcal R}<\widehat{\mathcal R}$ which is a solution to the equation $\pi(\mathcal R)=\mathcal R$, where we denote by the constant
	$$\widehat{\mathcal R}:=\left(\frac{1-\overline{\mathscr N_2} K_0T^{s\alpha \vartheta} }{(1+s)\overline{\mathscr N_2} K_0}\right)^{1/s}. $$
	We note that  the function $\mathcal R \mapsto \widehat{\pi}(\mathcal R):= \pi(\mathcal R)-\mathcal R$ is continuous on  $(0;\widehat{\mathcal R})$ with the   values  $\widehat{\pi}(0)=\mathcal N_2   \|f\|_{\mathbb{H}^{\nu+( 1-\vartheta) }(\Omega)}$ and 
	\begin{align*}
	\widehat{\pi}(\widehat{\mathcal R})\,&=\overline{\mathscr N_2} K_0   (T^{s\alpha\vartheta}+\widehat{\mathcal R}^s) \widehat{\mathcal R} +\mathcal N_2   \|f\|_{\mathbb{H}^{\nu+( 1-\vartheta) }(\Omega)} - \widehat{\mathcal R} \nn\\
	\,&= \Big( \overline{\mathscr N_2} K_0    \widehat{\mathcal R}^s    - (1-\overline{\mathscr N_2} K_0 T^{s\alpha\vartheta} )\Big) \widehat{\mathcal R} +\mathcal N_2   \|f\|_{\mathbb{H}^{\nu+( 1-\vartheta) }(\Omega)}  \nn\\
	\,&= \left(1-\overline{\mathscr N_2} K_0T^{s\alpha \vartheta}\right)\left( \frac{1}{1+s} -1 \right)\widehat{\mathcal R} + \mathcal N_2   \|f\|_{\mathbb{H}^{\nu+( 1-\vartheta) }(\Omega)} \nn\\
	\,&=    \mathcal N_2   \|f\|_{\mathbb{H}^{\nu+( 1-\vartheta) }(\Omega)} - \frac{s}{1+s}\left(1-\overline{\mathscr N_2} K_0T^{s\alpha \vartheta}\right)\widehat{\mathcal R} \nn\\
	\,&=    \mathcal N_2   \|f\|_{\mathbb{H}^{\nu+( 1-\vartheta) }(\Omega)} - \frac{s}{1+s} \frac{\left(1-\overline{\mathscr N_2} K_0T^{s\alpha \vartheta}\right)^{1+1/s}}{(1+s)^{1/s} (\overline{\mathscr N_2} K_0)^{1/s} }  \nn\\
	\,& <  \mathcal N_2   \|f\|_{\mathbb{H}^{\nu+( 1-\vartheta) }(\Omega)} \left(1 - \frac{s}{1+s} \frac{\left(1/2\right)^{1+1/s}}{(1+s)^{1/s} s  }   (2(1+s))^{1+1/s} \right) \nn\\
	\,& = 0,
	\end{align*}
	where we note that $1-\overline{\mathscr N_2} K_0T^{s\alpha \vartheta}>\frac{1}{2}$. Therefore, there exists $0< \overline{\mathcal R}<\widehat{\mathcal R}$ such that $\pi(\mathcal R)=\mathcal R$. So it follows from (\ref{ngaymoi}) that  $\mathscr Q$   maps $\mathfrak X_{\al,\vartheta, \nu,T }(\mathcal R)$ into itself.\\

	\underline{Prove $\mathscr Q$   is a contraction mapping, then establish the existence of the mild solution:} We note that  
	\begin{align*}
	\overline{\mathscr N_2} K_0   (T^{s\alpha\vartheta}+\widehat{\mathcal R}^s)\,&= \overline{\mathscr N_2} K_0   \left( T^{s\alpha\vartheta} +  \frac{1-\overline{\mathscr N_2} K_0T^{s\alpha \vartheta} }{(1+s)\overline{\mathscr N_2} K_0} \right) \nn\\
	\,&= \frac{1-\overline{\mathscr N_2} K_0T^{s\alpha \vartheta}}{1+s} - \left(1-\overline{\mathscr N_2} K_0T^{s\alpha \vartheta}\right) +1 \nn\\
	\,&= 1- \frac{s}{1+s}   \left(1-\overline{\mathscr N_2} K_0T^{s\alpha \vartheta}\right) 
 < \frac{2+s}{2+2s}.
	\end{align*}
	Hence, we can deduce from  (\ref{ngaymoia}) that  
	\begin{align}
	\left\|\mathscr Q w^\dagger - \mathscr Q w^\ddagger \right\|_{  C^{\alpha \vartheta }((0,T];\mathbb{H}^{\nu}(\Omega))}	\,&\le 	\overline{\mathscr N_2} K_0   (T^{s\alpha\vartheta}+\widehat{\mathcal R}^s)  \|w^\dagger-w^\ddagger\|_{  C^{\alpha \vartheta }((0,T];\mathbb{H}^{\nu}(\Omega))} \nn\\
	\,&\le 	\frac{2+s}{2+2s}  \|w^\dagger-w^\ddagger\|_{  C^{\alpha \vartheta }((0,T];\mathbb{H}^{\nu}(\Omega))}. \nn  
	\end{align}
	We imply that $\mathscr Q$ is a contraction mapping on $ \mathfrak X_{\al,\vartheta, \nu,T }(\mathcal R) $ which has a unique fixed point $u$ in this space. This fixed point is the unique  mild solution of Problem (\ref{mainprob1})-(\ref{mainprob4}). In addition,   inequality (\ref{ubound}) can be easily obtained. The remain of the proof is split as the following steps. 
	
	\vspace*{0.2cm}
	
	\noindent \textbf{Part a)} \underline{Show that $ u \in   L^p (0,    T ;  \mathbb H^{\nu+(\vartheta'-\vartheta)} (\Omega))$ for all $1\le p < \frac{1}{\alpha\vartheta'}$:}

\noindent	It is easy to see   the estimate $\displaystyle \left\| {\bf B}_\alpha(t,T)f \right\|_{\mathbb H^{\nu+(\vartheta'-\vartheta)} (\Omega)} \lesssim t^{-\alpha\vartheta'} \left\|f \right\|_{\mathbb{H}^{\nu+(1-\vartheta)}(\Omega)} $ for all $t>0$. Moreover, by  applying  Lemma \ref{locally}, we obtain    
	\begin{align*}
	&\left\| {\bf B}_\alpha(t,T) \int_0^T  { \bf P}_{\alpha }  (T-r) G(r,u(r)) dr \right\|_{\mathbb H^{\nu+(\vartheta'-\vartheta)} (\Omega)}  \nn\\
	&\quad \quad \quad \quad \lesssim  t^{-\alpha\vartheta'} \left\|\int_0^T  { \bf P}_{\alpha }  (T-r) G(r,u(r)) dr\right\|_{\mathbb{H}^{\nu+(1-\vartheta)}(\Omega)}   \lesssim t^{-\alpha\vartheta'} \left\|f \right\|_{\mathbb{H}^{\nu+(1-\vartheta)}(\Omega)}.
	\end{align*}
	On the other hand, it follows from $\nu+(\vartheta'-\vartheta)\le\nu+( 1- \vartheta)$ that the Sobolev embedding $\mathbb{H}^{\nu+( 1- \vartheta)}(\Omega) \hookrightarrow \mathbb{H}^{\nu+(\vartheta'-\vartheta)}(\Omega)$ holds. Hence, we can infer from  Lemma \ref{locally}  that    
	\begin{align}
	\left\|\int_0^t  {\bf  P}_{\alpha }  (t-r) G(r,u(r)) dr\right\|_{\mathbb{H}^{\nu+(\vartheta'-\vartheta)}(\Omega)} &\lesssim  \left\|\int_0^t  { \bf P}_{\alpha }  (t-r) G(r,u(r)) dr\right\|_{\mathbb{H}^{\nu+( 1- \vartheta)}(\Omega)} \nn\\ 
	&\lesssim   t^{-\alpha \vartheta } \left\|f \right\|_{\mathbb{H}^{\nu+( 1- \vartheta)}(\Omega)}  \lesssim  t^{-\alpha \vartheta' } \left\|f \right\|_{\mathbb{H}^{\nu+( 1- \vartheta)}(\Omega)} .  \nn
	\end{align}
	Summarily, the solution $u\in  L^p (0,    T ;  \mathbb H^{\nu+(\vartheta'-\vartheta)} (\Omega))$ for all $1\le p < \dfrac{1}{\alpha\vartheta'}$ since $t^{-\alpha \vartheta' }$ clearly belongs to $L^p (0,    T ;  \mathbb R)$ 
	for all $1\le p < \dfrac{1}{\alpha\vartheta'}$. The proof is  finalized.  
	
	\vspace*{0.2cm}
	
	\noindent \textbf{Part b)} \underline{Show that $ u \in   C\left([0,T];\mathbb H^{\nu -\eta} (\Omega)\right)$:}
	Let  $t$, $t'$ such that $0\le t \le  \widetilde t\le T$. Our purpose here is to find an upper bound of  the norm $\Big\| u(\widetilde t)  -u(t) \Big\|_{\mathbb{H} ^{\nu-\eta}(\Omega)}   $  .  Since $\vartheta<\eta \le \vartheta+1$ and $0<\vartheta<1$, the number $\dfrac{1+\vartheta-\eta}{2}$  consequently belongs to $[0,1]$. Hence, replacing $\displaystyle 1-\frac{1-\vartheta }{2}$ by $\dfrac{1+\vartheta-\eta}{2}$ helps to improve the inequalities (\ref{donvo}). Indeed, we have   
	\begin{align}
	\left| \frac{ E_{\alpha ,\alpha  }(-\lambda_j r^\alpha )}{E_{\alpha ,1}(-\lambda_j  T^\alpha )} \right|  \lesssim  r^{\alpha(\eta-\vartheta-1)} \lambda_j^{\eta-\vartheta}.  
	\end{align}
	As a consequence of the above inequality, we have 
	\begin{align*}
	\Big\| {\bf B}_\alpha(\widetilde t,T) f  - {\bf B}_\alpha(t,T)   f  \Big\|_{\mathbb{H} ^{\nu-\eta}(\Omega)}  \,&\lesssim \left\|f \right\|_{\mathbb{H}^{\nu+(1-\vartheta)}(\Omega)} \int_{t}^{\widetilde t} r^{\alpha(\eta-\vartheta) -1} dr     \nn\\
	\,& \lesssim \left\|f \right\|_{\mathbb{H}^{\nu+(1-\vartheta)}(\Omega)}  \left\{
	\begin{array}{llllcccc}
	(\widetilde t-t)^{\alpha(\eta-\vartheta)}\mathbf{1}_{0<\alpha(\eta-\vartheta)\le 1} \vspace*{0.1cm} \\
	\left( 	(\widetilde t-t)^{\alpha(\eta-\vartheta)-1}+\gamma 
	\right)\mathbf{1}_{1<\alpha(\eta-\vartheta)< 2} \\
	\end{array} \right\},
	\end{align*}
	where the number $\alpha(\eta-\vartheta)$ includes in $(0;2)$. Employing   Lemma \ref{locally} allows that
	\begin{align*}
	\big\| \mathfrak M_4 \big\|_{\mathbb{H} ^{\nu-\eta}(\Omega)}   
	\,& \lesssim \left\|\int_0^T \mathscr { P}_{\alpha }  (T-r) G(r,u(r)) dr\right\|_{\mathbb{H}^{\nu+(1-\vartheta)}(\Omega)}  \left\{
	\begin{array}{llllcccc}
	(\widetilde t-t)^{\alpha(\eta-\vartheta)}\mathbf{1}_{0<\alpha(\eta-\vartheta)\le 1} \vspace*{0.1cm} \\
	\left( (\widetilde t-t)^{\alpha(\eta-\vartheta)-1}+\gamma 
	\right)\mathbf{1}_{1<\alpha(\eta-\vartheta)< 2} \\
	\end{array} \right\}  \nn\\
	\,&  \lesssim  \left\|f \right\|_{\mathbb{H}^{\nu+(1-\vartheta)}(\Omega)} \left\{
	\begin{array}{llllcccc}
	(\widetilde t-t)^{\alpha(\eta-\vartheta)}\mathbf{1}_{0<\alpha(\eta-\vartheta)\le 1} \vspace*{0.1cm} \\
	\left( (\widetilde t-t)^{\alpha(\eta-\vartheta)-1}+\gamma 
	\right)\mathbf{1}_{1<\alpha(\eta-\vartheta)< 2} \\
	\end{array} \right\} , \nn
	\end{align*}
	provided that notation $\mathfrak M_4 $ is given by (\ref{b1}). Now, let us  consider the terms $\mathfrak M_2 $ and $\mathfrak M_3 $. It indicates from $\mu<\vartheta$ and $\vartheta <\eta$ that $0\le -{\sigma}\le\eta-\nu	$, and it results   $\mathbb{H}^{{\sigma}}(\Omega) \hookrightarrow \mathbb{H}^{\nu-\eta}(\Omega)$. This suggests to estimate the term $\mathfrak M_2$. In actual fact, we have 
	\begin{align*}
	\|\widetilde t\|_{\mathbb H^{\nu-\eta} (\Omega)}
	&\lesssim    \int_0^t \int_{t-r}^{t'-r} \rho ^{\alpha  -2}  \norm{G(r,u(r))}_{ \mathbb H^{{\sigma}} (\Omega)} d\rho  dr \\
	&\lesssim  \left\|f \right\|_{\mathbb{H}^{\nu+(1-\vartheta)}(\Omega)}  \int_0^t \int_{t-r}^{\widetilde t-r} \rho ^{\alpha  -2}  \left(r^{-\alpha \vartheta  } +  r^{-(1+s)\alpha \vartheta } \right)\mathfrak{L}_3 (r)  d\rho  dr\\
	&\lesssim  \left\|f \right\|_{\mathbb{H}^{\nu+(1-\vartheta)}(\Omega)}  (\widetilde t-t)^{\alpha-1} \int_0^t   \left(r^{-\alpha \left( \vartheta +\zeta \right)  } +  r^{- \alpha \left((1+s)  \vartheta + \zeta \right) } \right)  dr,
	\end{align*}
	provided that ${L}_3(t)\le K _0 t^{-\alpha\zeta}$ as $G$ satisfies \eqref{H3}. 
	Since $ \zeta < \displaystyle\frac{1}{\alpha} - (1+s) \vartheta $, we derive that the integral on the right-hand side of the previous
	expression is convergent. Hence, we obtain immediately the estimate
	$$ \|\mathfrak M_2\|_{\mathbb H^{\nu-\eta} (\Omega)}
	\lesssim     \|u\|_{ C^{\alpha \vartheta }((0,T];\mathbb{H}^{\nu}(\Omega))}  (\widetilde t-t)^{\alpha-1} $$
	and from the local  property of $G$ as in  \eqref{H3}, we find that
	\begin{align}
	\|\mathfrak M_3\|_{\mathbb H^{\nu-\eta} (\Omega)} 
	\,&\lesssim  \int_t^{\widetilde t} (\widetilde t-\tau)^{\alpha-1}  \|G(r,u(r))\|_{\mathbb{H}^{{\sigma}}(\Omega)}  dr \nn\\
	\,&\lesssim \left\|f \right\|_{\mathbb{H}^{\nu+(1-\vartheta)}(\Omega)} \int_t^{\widetilde t} (\widetilde t-r)^{\alpha-1} \left(r^{-\alpha \vartheta  } +  r^{-(1+s)\alpha \vartheta } \right)\mathfrak{L}_3 (r)  dr    \nn\\
	&\lesssim \left\|f \right\|_{\mathbb{H}^{\nu+(1-\vartheta)}(\Omega)} \int_t^{\widetilde t} (\widetilde t-r)^{\alpha-1} \left(r^{-\alpha (\vartheta +\zeta) } +  r^{-\alpha((1+s)\vartheta+\zeta) } \right)   dr \lesssim ( \widetilde t-t). \nn
	\end{align}    
	The above explanations imply $ u \in   C\left([0,T];\mathbb H^{\nu -\eta} (\Omega)\right)$. 
\end{proof}
\section*{Appendix}
\subsection*{\textbf{(AP.1.)} A singular integral} It is useful to recall some basic properties of a singular integral. For given $z_1>0$, $z_2>0$, and $0\le a <b\le T$, 
we denote by
\begin{align}
\displaystyle \mathcal K(z_1,z_2,a,b):=\displaystyle \int_{a}^{b} (b-\tau)^{z_1-1}(\tau-a)^{z_2-1}d\tau =(b-a)^{z_1+z_2-1} {\bf B}(z_1,z_2), \label{Kdefined}
\end{align}
where $\bf B$ is the Beta function, 
$  
{\bf B}(z_1,z_2) := \int_0^1 t^{z_1-1}(1-t)^{z_2-1} dt
$.  
%It obviously follows from (\ref{Kdefined}) that {\color{blue} if $z_1+z_2\ge 1$ then 
%\begin{align}
%\mathcal K(z_1,z_2,0,t) \le \displaystyle  \mathcal  K(z_1,z_2,0,T), \label{Kprop}
%\end{align} 
% for all $0 \le t \le T$.} 
Moreover, a special case of the Beta function is  ${\bf B}(z,1-z) = \pi/\sin(\pi z)$, see e.g. \cite{Samko,Podlubny,Lax,Brezis,
	Diethelm}. 

\subsection*{\textbf{(AP.2.)} A useful limit} For  $a>0$, $b>0$, $t>0$, $h>0$, the following convergence holds 
$$ \int_0^t (t+h-r)^{a-1}r^{b-1}dr \xrightarrow{h  \to 0^+}  \int_0^t (t-r)^{a-1}r^{b-1}dr .$$
Indeed, it can be proved by noting that $\displaystyle  \int_0^t (t-r)^{a-1}r^{b-1}dr=t^{a+b-1} \mathbf B(a,b)$ and $$\displaystyle \int_0^t   (t+h-r)^{a-1}   r^{b-1}dr = (t+h)^{a+b-1}\int_0^{t/(t+h)} (1-s)^{a-1}s^{b-1} ds  \xrightarrow{h \to 0^+} t^{a+b-1}\mathbf B(a,b). $$

\subsection*{\textbf{(AP.3.)} Proof of Lemma \ref{nayvesom}} The first inequality is estimated as follows
\begin{align}
z^{\al-1} E_{\al, \al} (-\lambda_j z^\al) &\le M_\al z^{\al-1} \frac{1}{1+\lambda_j z^\al}  = M_\al z^{\al-1} \Big(  \frac{1}{1+\lambda_j z^\al}  \Big)^\theta  \Big(  \frac{1}{1+\lambda_j z^\al}  \Big)^{1-\theta}\nn\\
& \le { M}_\al \lambda_j^{-\theta} z^{\al(1-\theta)-1}, \nn
\end{align}
and the second inequality is showed as follows
\begin{align}
\mathscr E_{\al, T} (-\lambda_j t^\al)&=	\left| \frac{E_{\alpha,1}(-\lambda_j t^\al)}{E_{\alpha,1}(-\lambda_j T^\alpha)} \right| \le \frac{ M_\al  }{ m_\al } \left(  \frac{1+\lambda_j T^\alpha} {1+\lambda_j t^\al}\right)^{1-\theta} \left(   \frac{1+\lambda_j T^\alpha}{1+\lambda_j t^\al}\right)^{\theta} \nn\\
&\le \frac{ M_\al  }{ m_\al }   T^{\alpha(1-\theta)}  (1+\lambda_j  T^\alpha)^{\theta}  t^{-\al(1-\theta)}   \nn\\
&\le M_\al m_\al^{-1}  T^{\alpha(1-\theta)} \Big( T^{\al \theta }+\lambda_1^{-\theta} \Big) \lambda_j^\theta t^{-\al(1-\theta)}.\nn
\end{align}
The proof is completed.

\subsection*{\textbf{(AP.4.)} List of  constants}   Here, we list some important constants appeared in this paper, where some of them contain the constant $ C_1(\nu, \theta)$, $C_2(\nu,\sigma)$ in the embeddings (\ref{Embed1}) and (\ref{Embed2}). These constants cannot be omitted in some proofs of this paper.
\begin{align}
%%%%%%%%%%%%%%%%%%%%%%%%%%%%%%%
\left\{
\begin{array}{llllcccc}
\displaystyle \mathcal M_1 &=& \mathcal M_1(\al,\theta,T) &:=& \displaystyle  M_\al^2  m_\al^{-1}  T^{\alpha(1-\theta)} \Big( T^{\al \theta }+\lambda_1^{-\theta} \Big), \vspace*{0.2cm}  \\
%%%%%%%%%%%%%%%%%%%%%%%%%%%%%%%
\displaystyle \mathcal M_2 &=& \mathcal M_2(\al,\theta,T) &:=& \displaystyle  \frac{M_\alpha^2}{m_\alpha}
\frac{T^{\alpha(1-\theta)}}{\alpha^2\theta(1-\theta)}(T^\alpha+\lambda_1^{-1})   ,  \vspace*{0.2cm}\\
%%%%%%%%%%%%%%%%%%%%%%%%%%%%%%%
\displaystyle  \mathscr M_1   &=& \mathscr M_1 (\al,\theta,T) &:=& \displaystyle \frac{\pi M_\al  \lambda_1^{-\theta} T^{\al(1-\theta)    }+\pi \mathcal M_1  }{\sin\Big(\pi\alpha(1-\theta)\Big)}, \vspace*{0.2cm}  \\
%%%%%%%%%%%%%%%%%%%%%%%%%%%%%%%
\displaystyle \overline{\mathcal M} _1 &=& \overline{\mathcal M} _1(q,\alpha,\nu,\sigma,T) &:=& \displaystyle             C_2(\nu,\sigma)  {M}_\al \frac{T^{\al +1/q}}{\al (\alpha q+1)^{1/q} } , \vspace*{0.2cm} \\
%%%%%%%%%%%%%%%%%%%%%%%%%%%%%%%
\displaystyle \overline{\mathcal M} _2 &=& \overline{\mathcal M} _2(q,\alpha,\theta,T) &:=& \displaystyle             \frac{T^{\alpha(1-\theta)}( T^{\alpha\theta+1/q}+\lambda_1^{-\theta})  }{(1-\alpha(1-\theta)q)^{1/q} \alpha(1-\theta) } , \vspace*{0.2cm} \\
%%%%%%%%%%%%%%%%%%%%%%%%%%%%%%%
\displaystyle \overline{\mathcal M} _3 &=& \overline{\mathcal M} _3(\alpha,T) &:=& \displaystyle  \frac{M_\alpha T^\alpha}{\alpha} \left( 1+ \frac{M_\alpha}{m_\alpha}(T^\alpha+\lambda_1^{-1})  \right)  , \vspace*{0.2cm} \\
%%%%%%%%%%%%%%%%%%%%%%%%%%%%%%%
\displaystyle  \mathscr M_2   &=& \mathscr M_2 (q,\alpha,\nu,\sigma,T) &:=& \displaystyle \|L_2\|_{L^\infty(0,T)} \sum_{1\le j\le3} \overline{\mathcal M}_j , \vspace*{0.2cm}  \\
%%%%%%%%%%%%%%%%%%%%%%%%%%%%%%%
\displaystyle \mathcal N_{1} &=& \mathcal N_{1}(\al,\theta,\nu,T)  &:=& \displaystyle   \mathcal M_1 M_\al^{-1}   C_1(\nu, \theta) \Big(1-  \|L_1\|_{L^\infty(0,T)}  \mathscr M_1\Big)^{-1} , \vspace*{0.2cm} \\
%\displaystyle \overline{\mathcal N_{1}} &=& \overline{\mathcal N_{1}} (\al,\vartheta)  &:=& \displaystyle  \frac{ C_1(\nu, \vartheta) M_\al m_\al^{-1}  T^{\alpha(1-\vartheta)} \Big( T^{\al \vartheta }+\lambda_1^{-\theta} \Big)+ C_1(\nu, \theta)  T^{\alpha(1-\vartheta)} }{1- \mathfrak  L_1  \mathscr M_1  }, \vspace*{0.2cm} \\
%%%%%%%%%%%%%%%%%%%%%%%%%%%%%%%
\mathcal N_2 &=& \mathcal N_2(\alpha,\vartheta,T) &:=& \displaystyle  M_\al   m_\al^{-1}  T^{\alpha\vartheta} \Big( T^{\al (1-\vartheta) }+\lambda_1^{\vartheta-1} \Big) , \vspace*{0.2cm} \\ 
%%%%%%%%%%%%%%%%%%%%%%%%%%%%%%%
\displaystyle \mathscr N_1  &=& \mathscr N_1(\alpha,\mu,\vartheta ,\zeta) &:=& \displaystyle \max   \Big\{ \mathbf B(\alpha z_j;1-\alpha(1+s) \vartheta  -\alpha\zeta ), j=1,2 \Big\},   ~  \left\{ \hspace*{-0.2cm} \begin{array}{llcc}
z_1= \vartheta -\mu   , \\
z_2=	1 -\mu  ,
\end{array} \right.	   \vspace*{0.2cm} \\
%%%%%%%%%%%%%%%%%%%%%%%%%%%%%%%
\displaystyle \mathscr N_2  &=& \mathscr N_2(\alpha,\mu ,\vartheta ,\zeta,s,T) &:=& \displaystyle   M_\alpha   \mathscr N_1 \max\Big\{   T^{\alpha\left(z_j -s \vartheta -\zeta\right)}, j=1,2 \Big\} , \vspace*{0.2cm} \\
%%%%%%%%%%%%%%%%%%%%%%%%%%%%%%%
\displaystyle   \overline{\mathscr N_2}  &=& \overline{\mathscr N_2}(\alpha,\mu ,\vartheta ,\zeta,s,T) &:=& \displaystyle   \mathscr N_2   (  1   + \mathcal N_2T^{-\alpha\vartheta}   ) , \vspace*{0.2cm} \\
%%%%%%%%%%%%%%%%%%%%%%%%%%%%%%%
\displaystyle   \mathcal N_f  &=& \mathcal N_f(\alpha,\nu,\vartheta,T,s) &:=& \displaystyle  \left( \frac{s}{\mathcal N_2   \|f\|_{\mathbb{H}^{\nu+( 1-\vartheta) }(\Omega)}} \right)^s \frac{1}{(2 +2s)^{1+s}} , \vspace*{0.2cm} \\
%%%%%%%%%%%%%%%%%%%%%%%%%%%%%%%
\displaystyle \widehat{\mathcal R} &=&  \widehat{\mathcal R}(\alpha,\mu ,\vartheta ,\zeta,s,T) &:=&  \displaystyle\left( \frac{1-\overline{\mathscr N_2} K_0T^{s\alpha \vartheta}}{(1+s)\overline{\mathscr N_2} K_0}  \right)^{1/s}  , \vspace*{0.2cm}\\
%%%%%%%%%%%%%%%%%%%%%%%%%%%%%%%
\eta_{glo} &=&  \eta_{glo}(\alpha,\theta,\nu')  &:=& \left\{
\begin{array}{llllcccc}
\min\left\{\alpha(\theta +\nu'-1);\alpha-1\right\}\mathbf{1}_{0<\alpha(\theta +\nu'-1)\le 1} \vspace*{0.1cm} \\
\min  \left\{  \alpha(\theta +\nu'-1)-1;\alpha-1 \right\}
\mathbf{1}_{1<\alpha(\theta +\nu'-1)< 2} \\
\end{array} \right\} , \vspace*{0.2cm}\\
%%%%%%%%%%%%%%%%%%%%%%%%%%%%%%%
\eta_{cri} &=& \eta_{glo}(\alpha,\eta,\vartheta)  &:=& \left\{
\begin{array}{llllcccc}
\min\left\{\alpha(\eta-\vartheta);\alpha-1\right\}\mathbf{1}_{0<\alpha(\eta-\vartheta)\le 1} \vspace*{0.1cm} \\
\min  \left\{  \alpha(\eta-\vartheta)-1;\alpha-1 \right\}
\mathbf{1}_{1<\alpha(\eta-\vartheta)< 2} \\
\end{array} \right\}.  
\end{array}
\right.  \nn
\end{align}


\begin{thebibliography}{par}
	
	\bibitem{Adams} R.A. Adams, \textit{Sobolev spaces}, Academic Press, 1975.
	
	\bibitem{Samko} S.G. Samko, A.A. Kilbas, O.I. Marichev; \emph{Fractional integrals and derivatives}, Theory and Applications, Gordon and Breach Science, Naukai Tekhnika, Minsk, 1987.
	
	\bibitem{Podlubny} I. Podlubny; \emph{Fractional differential equations}, Academic Press, London, 1999.   
	
	\bibitem{Lax} P.D. Lax; \emph{Functional analysis}, Wiley Interscience, New York, 2002.
	
	\bibitem{Brezis} H. Brezis; \emph{Functional analysis}, Springer, New York, 2011.
	
	\bibitem{Diethelm} K. Diethelm; \emph{The analysis of fractional differential equationst}, Springer, Berlin, 2010.
	
	
	
	\bibitem{r2}
	W. Fan, F. Liu, X. Jiang, I. Turner, \emph{A novel unstructured mesh finite element method for solving the time-space fractional wave equation on a two-dimensional irregular convex domain}, Fractional Calculus and Applied Analysis, \textbf{20}(2) (2017) 352--383.
	
	\bibitem{r3}
	S. Guo, L. Mei, Y. Li, \emph{An efficient Galerkin spectral method for two-dimensional fractional nonlinear reaction-diffusion-wave equation}, Computers \& Mathematics with Applications, \textbf{74}(10) (2017) 2449--2465.
	
	\bibitem{r4}
	D. Kumar, J. Singh, D. Baleanu, \emph{A new analysis for fractional model of regularized long-wave equation arising in ion acoustic plasma waves}, Mathematical Methods in the Applied Sciences, \textbf{40}(15) (2017) 5642--5653.
	
	\bibitem{Showater} R. E. Showalter, \emph{The final value problem for evolution equations}, Journal of Mathematical Analysis and Applications, \textbf{47}(3) (1974) 563--572.
	
	\bibitem{Carasso} A. Carasso, \emph{Error Bounds in the Final Value Problem for the Heat Equation}, SIAM J. Math. Anal., \textbf{7} (1976) 195--199. 
	
	
	\bibitem{Baumeister} J. Baumeister, \emph{Stable Solution of Inverse Problems}, Springer-Verlag, Mar 9, 1986.
	
	\bibitem{Nochetto} R. H. Nochetto, E. Ot\'arola, A. J. Salgado; \emph{A PDE Approach to Space-Time Fractional Parabolic Problems}, SIAM J. Numer. Anal., \textbf{54} (2016) 848--873. 
	
	\bibitem{Yamamoto1} K. Sakamoto, M. Yamamoto, \emph{Initial value/boudary value problems for fractional diffusion-wave equations and applications to some inverse problems}, J. Math. Anal. Appl.,  \textbf{382} (2011) 426--447.
	
	\bibitem{Yamamoto2} Y. Kian, M. Yamamoto, \emph{On existence and uniqueness of solutions for semilinear fractional wave equations}, Fractional Calculus and Applied Analysis,  \textbf{20}(1) (2017) 117--138.
	
	\bibitem{Courant} R. Courant, D. Hilbert; \emph{Methods of Mathematical Physics}, Vol. 1, Interscience, New York, 1953.
	
	\bibitem{Kato} T. Kato; \emph{Perturbation Theory for Linear Operators}, Springer-Verlag Berlin Heidelberg, 1995. 
	
	\bibitem{McLean}  W. McLean; \emph{Strongly Elliptic Systems and Boundary Integral Equations}, Cambridge University Press, Cambridge, 2000. 
	
	
	\bibitem{NHTuanN} N.H. Tuan, A. Debbouche, T.B. Ngoc; \textit{Existence and regularity of final value problems for time fractional wave equations}. Comput. Math. Appl. 78 (2019), no. 5, 1396--1414.
	
	\bibitem{Wei} T. Wei, Y. Zhang; \emph{The backward problem for a time-fractional diffusion-wave equation in a bounded domain}, Computers and Mathematics with Applications, \textbf{75}(10) (2018) 3632--3648.
	
	\bibitem{Dang} D.T. Dang, E. Nane, D.M.  Nguyen, N.H. Tuan; \emph{Continuity of Solutions of a Class of Fractional Equations}, Potential Anal, \textbf{49} (2018) 423--478.
	
	
	\bibitem{Taylor} M. Taylor, \emph{Remarks on Fractional Difiusion Equations,}\\ www.unc.edu/math/Faculty/met/fdif.pdf.
	
	\bibitem{Kian} Y. Kian, L. Oksanen,  E. Soccorsi, M.Yamamoto; \emph{Global uniqueness in an inverse problem for time fractional diffusion equations
	}, Journal of Differential Equations, \textbf{264}(2) (2018) 1146--1170.
	
	\bibitem{Chen}
	W. Chen and S. Holm, \emph{Physical interpretation of fractional diffusion-wave equation via lossy media obeying frequency
		power law,}  preprint (2003), https://arxiv.org/abs/math-ph/0303040.
	
	\bibitem{Pi}	T. Picon M. D’Abbicco, M.R. Ebert \emph{ Global existence of small data solutions to the semilinear fractional
		wave equation}  New trends in analysis and interdisciplinary applications, 465471, Trends Math. Res.
	Perspect., Birkhuser/Springer, Cham, 2017.
	
	\bibitem{Al}	E. Alvarez, C.G. Gal,  V. Keyantuo,  M. Warma, \emph{Well-posedness results for a class of semi-linear super-diffusive equations} Nonlinear Anal. 181 (2019), 24--61.
	
	\bibitem {Mai} F. Mainardi. \emph{Fractional calculus and waves in linear viscoelasticity} Imperial College Press, London, 2010. An introduction
	to mathematical models.
	
	\bibitem {Mai1} F. Mainardi,  \emph{The fundamental solutions for the fractional diffusion-wave equation} Appl. Math. Lett. 9 (1996)
	23–28.
	\bibitem {Mai2} F. Mainardi, P. Paradisi, \emph{Fractional diffusive waves,} Journal of Computational Acoustics, 9 (2001) 1417–-1436.
	
	\bibitem {Top} E. Topp and M. Yangari \emph{ Existence and uniqueness for parabolic problems with Caputo time
		derivative} J. Differential Equations, 262(12):6018–-6046, 2017.
	
	\bibitem {Li} L. Li, J.-G. Liu, and L. Wang \emph{ Cauchy problems for Keller-Segel type time-space fractional diffusion equation} J. Differential Equations, 265(3):1044–-1096, 2018.
	
	\bibitem {wei} J. Xian, T. Wei \emph{ Determination of the initial data in a time-fractional diffusion-wave problem by a final time data}
	Computers and  Mathematics with Applications, In press, corrected proof, Available online. 
	
	\bibitem {Jan} J. Janno,  N. Kinash, \emph{ Reconstruction of an order of derivative and a source term in a fractional diffusion equation from final measurements} Inverse Problems 34 (2018), no. 2, 025007, 19 pp.
	
	\bibitem{Dong}  H. Dong,  D. Kim, \emph{ $L_p$-estimates for time fractional parabolic equations with coefficients measurable in time}  Adv. Math. 345 (2019), 289--345.
	
	\bibitem{Met}	R. Metzler and J. Klafter \emph{ Boundary value problems for fractional diffusion equations}  Physica
	A: Statistical Mechanics and its Applications, 278(1):107–125, 2000.
	
	\bibitem{Bar} B. Kaltenbacher, W. Rundell \emph{ On an inverse potential problem for a fractional reaction–diffusion equation   }, Inverse Problems, Volume 35, Number 6, 2019.
	
	
	
	\bibitem{Bar1} B. Kaltenbacher, W. Rundell \emph{	Regularization of a backward parabolic equation by fractional operators}  Inverse Probl. Imaging 13 (2019), no. 2, 401--430.
	
	\bibitem{Jia}	J. Jia,  J. Peng,  J. Gao,  Y. Li,  \emph{ Backward problem for a time-space fractional diffusion equation} Inverse Probl. Imaging 12 (2018), no. 3, 773--799.
	
	
	\bibitem {Liu} L. Wang,  J. Liu, \emph{ Total variation regularization for a backward time-fractional diffusion problem}  Inverse Problems 29 (2013), no. 11, 115013, 22 pp.
	
	 \bibitem {Nieto1} X.L. Ding, J.J. Nieto, Analytical solutions for multi-term time-space fractional partial differential equations with nonlocal damping terms. Fractional Calculus and Applied Analysis 21 (2018), 312-335.
	 
	
	 \bibitem{Nieto2} P. Agarwal, J.J. Nieto,  M.J. Luo, \emph{ Extended Riemann-Liouville type fractional derivative operator with applications} Open Math. 15 (2017), no. 1, 1667--1681.
	 
	
	
	
	\bibitem {Nieto3} A.  Deiveegan,  J. J. Nieto,  P. Prakash, Periasamy,  \emph{The revised generalized Tikhonov method for the backward time-fractional diffusion equation} J. Appl. Anal. Comput. 9 (2019), no. 1, 45--56.
	
	\bibitem{Du1} X. J. Yang, D. Baleanu, and H. M. Srivastava, \emph{ Local fractional integral transforms and their applications,} Academic Press, New York, 2015.
	
	
	\bibitem{Du2} D. Baleanu, K. Diethelm, E. Scalas, and J. J. Trujillo, \emph{Fractional calculus: models and numerical
		methods,} World Scientific, Singapore, 2016.
	
	\bibitem{Du3} D. Baleanu, J.A.T. Machado, Z.B. Guvenc (Eds.), \emph{New Trends in Nanotechnology and Fractional Calculus Applications,} Springer, Netherlands, 2010.
	
	
	\bibitem{zhou1} Y. Zhou, \emph{ Basic Theory of Fractional Differential Equations,}  World Scientific, Singapore,
	2014.
	
	\bibitem{zhou2}  Y. Zhou,  \emph{Fractional Evolution Equations and Inclusions: Analysis and Control, }  Elsevier 
	Academic Press, 2015.
	
	\bibitem{Ca} D. del Castillo-Negrete, B.A. Carreras, and V.E. Lynch, \emph{Fractional diffusion in plasma turbulence,} Physics of Plasmas (2004).
	
	\bibitem {Yang} M. Yang, J. Liu \emph{Solving a final value fractional diffusion problem by boundary
		condition regularization},  Appl. Numer. Math. 66 (2013), 45--58. 
		
		
	
	\bibitem {Andrade} B. de Andrade, A.N. Carvalho, P.M. Carvalho-Neto, P. Marin-Rubio; \emph{ Semilinear fractional
		differential equations: global solutions, critical nonlinearities and comparison results },  Topological
	Methods in Nonlinear Analysis, 45 (2015), 439--467.
	
	\bibitem {Gu}	B.H. Guswanto, T. Suzuki; \emph{Existence and uniqueness of mild solutions for fractional semilinear
		differential equations} Electronic Journal of Diff. Equ., 2015 (2015), 16 pp.
		
	
	\bibitem{Ja1} J. Janno, N. Kinash; \emph{Reconstruction of an order of derivative and a source term in a fractional diffusion equation from final measurements}, Inverse Problems, 34 (2018), 19 pp.
	
	\bibitem{Ja2} J. Janno, K. Kasemets; \emph{Uniqueness for an inverse problem for a semilinear time-fractional diffusion equation},  Inverse Probl. Imaging, 11 (2017), 125--149.
	
	\bibitem {Ya5}  D. Jiang, Z. Li, Y. Liu, M. Yamamoto; \emph{Weak unique continuation property and a related inverse source problem for time-fractional diffusion-advection equations,}
	Inverse Problems, 33 (2017), 21 pp.
	
	
	\bibitem {Ya6} {\color{black} Z. Li, O.Y. Imanuvilov, M. Yamamoto; \emph{Uniqueness in inverse boundary value
			problems for fractional diffusion equations,}
		Inverse Problems 32, (2016), 16 pp.}
	
	\bibitem {Ya10} {\color{black} G. Li, D. Zhang, X. Jia, M. Yamamoto; \emph{Simultaneous inversion for the space-dependent diffusion coefficient and the fractional order in the time-fractional diffusion equation,}
		Inverse Problems, 29 (2013), 36 pp.}
		
		
				

	
	\bibitem {Ya9} {\color{black} Y. Luchko, W. Rundell, M. Yamamoto, L. Zuo; \emph{Uniqueness and reconstruction of an unknown semilinear term in a time-fractional reaction-diffusion equation,}
		Inverse Problems, 29 (2013), 16 pp.}
	
	\bibitem {Ya8} {\color{black} L. Miller, M. Yamamoto; \emph{Coefficient inverse problem for a fractional diffusion equation,}
		Inverse Problems, 29 (2013), 8 pp. }
		
		
	
	\bibitem{Run}	W. Rundell, Z. Zhang, \emph {Recovering an unknown source in a fractional diffusion problem} J. Comput. Phys. 368 (2018), 299--314. 
	
	\bibitem{Run1}	W. Rundell, Z. Zhang, \emph{Fractional diffusion: recovering the distributed fractional derivative from overposed data}  Inverse Problems 33 (2017), no. 3, 035008, 27 pp.
	
	\bibitem{Luc} N.H. Luc, N.H. Tuan, Y. Zhou,  \emph{ Regularity of the solution for a final value problem for the Rayleigh-Stokes equation}  Math. Methods Appl. Sci. 42 (2019), no. 10, 3481--3495.
	
	\bibitem{3} J. Ginibre and G. Velo, \emph{The global Cauchy problem for nonlinear Klein-Gordon equation,}  Math. Z, 189, (1985)
	487--505.
	
	\bibitem{5} M. G. Grillakis, \emph{Regularity and asymptotic behavior of the wave equation with a critical non-linearity,} Ann. of
	Math., 132 (1990), 485--509.
	
	\bibitem{22} J. Shatah and M. Struwe, \emph{ Regularity results for nonlinear wave equations,}  Ann. of Math , 2 (138) (1993), 503--518.
	
	\bibitem{23} J. Shatah and M. Struwe,  \emph{ Well-Posedness in the energy space for semilinear wave equation with critical growth,}
	IMRN, 7 (1994), 303--309.
	
	\bibitem {Carvalho} M.J. Arrieta,  A.N. Carvalho, \emph{ Abstract parabolic problems with critical nonlinearities and applications
		to Navier–Stokes and heat equations. } Trans. Am. Math. Soc. 352, 285-–310 (1999).
	
	
	
\end{thebibliography}
\end{document}